\documentclass[11pt,a4paper]{article}

\usepackage{amssymb} 
\usepackage[intlimits]{amsmath}
\usepackage{amsfonts}
\usepackage{amsthm} 
\usepackage{mathrsfs} 

\usepackage[pdfpagelabels]{hyperref}
\hypersetup{colorlinks=true,linkcolor=RoyalPurple,citecolor=RoyalPurple}
\usepackage[capitalize]{cleveref}
\crefname{equation}{}{} %no "eqn." for counter of equations

\usepackage[utf8]{inputenc}

\usepackage{graphicx}
\usepackage[margin=0.9in]{geometry}
\usepackage{color}
\usepackage[dvipsnames]{xcolor}

\usepackage[backgroundcolor=white, bordercolor=blue, linecolor=blue]{todonotes}

\usepackage{cite}

\let\oldsqrt\sqrt
\def\sqrt{\mathpalette\DHLhksqrt}
\def\DHLhksqrt#1#2{%
\setbox0=\hbox{$#1\oldsqrt{#2\,}$}\dimen0=\ht0
\advance\dimen0-0.2\ht0
\setbox2=\hbox{\vrule height\ht0 depth -\dimen0}
{\box0\lower0.4pt\box2}}

\allowdisplaybreaks

\newcommand{\R}{\mathbb{R}} % reelle Zahlen
\newcommand{\N}{\mathbb{N}} % natuerliche Zahlen
\newcommand{\Z}{\mathbb{Z}} % ganze Zahlen
\newcommand{\C}{\mathbb{C}} % komplexe Zahlen
\newcommand{\diam}{\textnormal{diam}} % diam ...
\newcommand{\supp}{\textnormal{supp}} % supp ...
\renewcommand{\phi}{\varphi}

\newcommand{\bS}{\mathbb{S}}

\newcommand{\cB}{{\mathcal B}}

\newcommand{\cE}{{\mathcal E}}

\newcommand{\cH}{{\mathcal H}}

\newcommand{\cL}{{\mathcal L}}

\newcommand{\cX}{{\mathcal X}}

\newcommand{\weakto}{\rightharpoonup}

\newcommand{\eps}{\varepsilon}

\newtheorem{thm}{Theorem}[section]
\newtheorem{cor}[thm]{Corollary}
\newtheorem{prop}[thm]{Proposition}
\newtheorem{lemma}[thm]{Lemma}
\newtheorem*{theorem*}{Theorem}

\theoremstyle{definition}
\newtheorem{defi}[thm]{Definition}
\newtheorem{remark}[thm]{Remark}
\newtheorem*{remark*}{Remark}

\numberwithin{equation}{section}

\title{On a class of (non)local superposition operators of arbitrary order}

\author{
Serena Dipierro\footnote{Department of Mathematics and Statistics,
The University of Western Australia, 35 Stirling Highway, Crawley, Perth, WA 6009, Australia, serena.dipierro@uwa.edu.au}, 
Sven Jarohs\footnote{Institut f\"ur Mathematik, Goethe-Universit\"at, Frankfurt, Robert-Mayer-Stra\ss e 10, D-60629 Frankfurt, jarohs@math.uni-frankfurt.de.},  and 
Enrico Valdinoci\footnote{Department of Mathematics and Statistics, The University of Western Australia, 35 Stirling Highway, Crawley, Perth, WA 6009, Australia,
enrico.valdinoci@uwa.edu.au}
}

\date{\today}

\begin{document}

\maketitle

\begin{abstract}
In this paper we introduce a very general setting dealing with the superposition of
operators of any positive order and provide a systematic study of them.

We also provide examples and counterexamples, as well as characterizing properties of the
measures and the functional spaces under consideration.

Moreover, we present some applications regarding the existence theory for a class of nonlinear problems involving superposition operators of arbitrary (possibly fractional) order. 
\end{abstract}

{\footnotesize
\begin{center}
\textit{Keywords.} Mixed-order operators $\cdot$ nonlinear analysis $\cdot$ hypersingular kernels $\cdot$ higher-order operators $\cdot$ L\'evy-type operators\\[2ex]

\textit{Mathematics Subject Classification:} %(2020)
35A15, %Variational methods applied to PDEs
35R09, %Integro-partial differential equations
35J40. %Boundary value problems for higher-order elliptic equations
\end{center}
}

\section{Introduction}

\subsection{Some motivations and goals}
Classical models in mathematical biology often account for population dynamics involving random dispersal. Traditionally, these models assume Brownian motion as the underlying mechanism, but the recent literature has expanded this framework to include anomalous diffusion, which is governed by ``fat-tailed'' L\'evy-type distributions.
This line of theoretical research is motivated by the strong experimental evidence
confirming the presence of nonstandard diffusion in several situations of biological
interest, including animal foraging and cell motility.

In this context, a natural development is to consider biological populations in which different individuals are subject to different types of dispersal mechanisms (see, e.g.,~\cite{MR4249816, MR4651677}). The corresponding mathematical models capture diffusion via the superposition of spatial operators of different (potentially fractional) orders. A general theoretical setting for such a model was introduced in~\cite{DPSV23}.

The literature concerning the superposition of differential operators of varying orders is extensive and rapidly growing, reflecting the versatility of these models across different theoretical frameworks and their relevance to a broad spectrum of applications, see e.g.~\cite{MR2095633, MR2542727, MR2653895, MR2911421, MR3485125, MR4313576, MR4469224, MR4477810, MR4391102, MR4821750, MR4882759, MR4942302, DPLSV-MN}.\medskip

While the above examples focus on the superposition of  pseudo-differential operators of order at most~$2$, the goal of this paper is to introduce a general setting,
capable of dealing with the superposition of operators of any
positive order. In particular, the present seemingly purely nonlocal approach includes \textit{classical} local problems such as the clamped plate problem and also allows to study problems involving arbitrary integer powers of the Laplacian, see e.g.~\cite{zbMATH00223815,zbMATH06715608,zbMATH06756312,zbMATH07109934} and the monograph~\cite{zbMATH05712793} on the theory of polyLaplacians. In our framework, we give some applications related to functional analysis, nonlinear analysis, and topological methods.

In the next sections we will provide the details of the mathematical setting that we work in and present our main results.

\subsection{A class of homogeneous integrodifferential operators of arbitrary order}
\label{operator-simple}

Let~$N\in \N$. For any~$s>0$ we consider measures~$\nu_s$ which are given in polar coordinates by 
\begin{equation} 
\label{definition measure}
\nu_s(U):= \int_{0}^{+\infty} \int_{\bS^{N-1}} \chi_U(r\theta) r^{-1-2s}\,\sigma_s(d\theta) \,dr\qquad\text{ for all~$U\in \cB(\R^N)$.}
\end{equation}
Here, as customary, $\chi_U$ denotes the indicator function
of a set~$U$ (equal to~$1$ in~$U$ and to~$0$ in the complement of~$U$) and~$\cB$ stands for the family of Borel sets. Also,
\begin{equation} 
\label{assumption basis}
\text{$\sigma_s$ is a probability measure on~$\bS^{N-1}$,}
\end{equation}
that is,~$\sigma_s(\bS^{N-1})=1$. We note that the above assumptions imply that~$\nu_s$ is finite on all compact subsets of~$\R^N\setminus\{0\}$.

Moreover, as a consequence of~\eqref{assumption basis},
\begin{equation}\label{simple ellipticitybasic}
\max_{e\in\bS^{N-1}}\int_{\bS^{N-1}}|e\cdot \theta|^{\alpha}\sigma_s(d\theta)
>0\quad\text{for all~$\alpha\geq0$,}
\end{equation} see Lemma~\ref{lemma:nuovo} for the proof of this implication.

The particular choice~$\alpha:=2s$ will be important for our analysis. For this, we introduce the function~$M_{s,\sigma_s}:\bS^{N-1}\to [0,1]$ defined, for every~$e\in\bS^{N-1}$, as
\begin{equation}\label{def ms}
M_{s,\sigma_s}(e):=\int_{\bS^{N-1}}|e\cdot\theta|^{2s}\sigma_s(d\theta).
\end{equation}
Since~$\sigma_s$ is a probability measure, 
we have that
\begin{equation}\label{usoduevoltefprse769043-0876}
M_{s,\sigma_s}(e)\le \int_{\bS^{N-1}} \sigma_s(d\theta)=1.
\end{equation}

We consider~$e_s\in\bS^{N-1}$ such that
\begin{equation}\label{bvncmx9e76954uy95}
M_{s,\sigma_s}(e_s)=\max_{e\in\bS^{N-1}}M_{s,\sigma_s}(e).\end{equation}

For~$m\in \N$ such that~$m>s$ and\footnote{The space~$C^{2m}_b(\R^N)$ consists of bounded functions whose derivatives up to order~$2m$ are continuous.} $u\in C^{2m}_b(\R^N)$ define, for all~$ x\in \R^N$,
\begin{equation}\label{defop}
L_{m,s}u(x):= \frac{c_{m,s}}{2} \int_{\R^N}\delta_mu(x,y) \,\nu_s(dy),
\end{equation}
where~$\delta_mu(x,y)$ denotes the difference operator of order~$2m$, that is,
\begin{equation}\label{defdelta}
\delta_mu(x,y):=\sum_{k=-m}^m(-1)^k\binom{2m}{m-k}u(x+ky).
\end{equation}
Here, $c_{m,s}$ is a constant depending on~$N$, $m$, and~$s$ (but we omit the explicit dependence on~$N$ as it does not play any role in our analysis)
satisfying
\begin{equation}\label{constcms00}
\frac{2}{c_{m,s}}=2^m\int_{\R^N}\big(1-\cos(e_s\cdot y)\big)^{m}\,\nu_s(dy),
\end{equation}
where~$e_s\in \bS^{N-1}$ is given as in~\eqref{bvncmx9e76954uy95}.

In place of assumption \cref{assumption basis} it makes also sense to assume instead that~$\sigma_s$ is a nontrivial, bounded measure on~$\bS^{N-1}$
(in this sense, the probability renormalization in~\eqref{assumption basis}
does not constitute a significant restriction, up to a change of notation
accounting for the total mass of~$\sigma_s$).
Here, let us emphasize that {\em if~$\sigma_s$ were not bounded, then there would exist~$u\in C^{\infty}_c(\R^N)$ such that~$L_{m,s}u(x)$ does not exist for some~$x\in \R^N$}, see \cref{lemma remark xyz} below, therefore assumption \cref{assumption basis}
is somewhat unavoidable.

A classical choice for~$\nu_s$ in~\eqref{defop}
is~$\nu_s(dz):=|z|^{-N-2s}\, dz$ (possibly renormalized by a factor~$\cH^{N-1}(\bS^{N-1})$). With this choice and for some~$m\in\N$, with~$m>s$, it follows that~$L_{m,s}u=(-\Delta)^su$
for~$u\in C^{\infty}_c(\R^N)$, see~\cite{SKM93,AJS18a}. We emphasize that this equality also holds for~$s\in \N$.
%In Section~\ref{sec:1}, we extend the basic ideas of~\cite{AJS18a} in detail to the more general choice using the measure~$\nu_s$ and the operator~$L_{m,s}$.
\medskip

We now introduce two further assumptions (namely~\eqref{simple ellipticityBIS} and~\eqref{ellipticity} below)
that are uniform in the range of parameters~$s$ that we take into account. These new assumptions
will be relevant in the forthcoming Section~\ref{vfcbieuwr678265r8ot30987654}, where we deal with superposition operators.

In this setting, a useful condition requires an {\em ellipticity} assumption which is uniform for large~$s$ and reads as follows:
\begin{equation}\label{simple ellipticityBIS}
\liminf_{s\to+\infty} \int_{\bS^{N-1}}|e_s\cdot\theta|^{2s}\sigma_s(d\theta)>0.
\end{equation}
We stress that~\eqref{simple ellipticityBIS} is equivalent to
\begin{equation}
\label{simple ellipticity}
\lambda:=
\inf_{s>0}\int_{\bS^{N-1}}|e_s\cdot\theta|^{2s}\sigma_s(d\theta)>0,
\end{equation}
see Lemma~\ref{lem:equi00} for a proof of this fact.

We emphasize, moreover, that~\cref{simple ellipticity} implies a kind of ellipticity in the direction~$e_s$. This property is sometimes
strengthened with the following \textit{strong ellipticity} assumption
\begin{equation}\label{ellipticity}
\lambda_0:=\inf_{{\eta\in\bS^{N-1}}\atop{s>0}} \int_{\bS^{N-1}}|\eta\cdot \theta|^{2s}\,\sigma_s(d\theta)>0.
\end{equation}
In this work, we do not rely on this strong ellipticity assumption, and this allows us to include in our analysis the case in which~$\sigma_s=\delta_{e_s}$, being~$\delta_x$ the Dirac measure at~$x\in \R^N$ (this is an interesting example, comprising the case
of the sum of one-dimensional fractional Laplacians, possibly of different orders, in the coordinate directions, see~\eqref{9dihf3iu4tumy9b5} for more details).

\subsection{Superposition operators of arbitrary order}\label{vfcbieuwr678265r8ot30987654}

Building on the setting presented in Section~\ref{operator-simple},
we now introduce a class of superposition operators of arbitrary order.

This idea is inspired by the framework introduced in~\cite{DPSV23} 
for order up to~$2$
and goes as follows.
Let~$\mu^+$ and~$\mu^-$ be two (nonnegative) finite Borel measures on~$[0,+\infty)$ and consider the signed measure~$\mu:=\mu^+-\mu^-$. We assume that there exists~$s_{\ast}>0$ such that
\begin{equation}\label{mu-assumption1}
\mu^+([s_{\ast},+\infty))>0\quad\text{and} \quad \mu^-([s_{\ast},+\infty))=0.
\end{equation}
Moreover, we assume that there exists some~$\gamma\in[0,1)$ such that
\begin{equation}\label{mu-assumption2}
\mu^-([0,s_{\ast}))\leq \gamma \mu^+([s_{\ast},+\infty)).
\end{equation}

We then consider the following operator
\begin{equation}\label{defi:superposition op}
Lu:=\int_{[0+\infty)}L_{s_1,s}u\,\mu(ds)=\int_{[0+\infty)}L_{s_1,s}u\,\mu^+(ds)-\int_{[0,s_{\ast})}L_{s_1,s}u\,\mu^-(ds),
\end{equation}
where~$s_1:= \min\{n\in\Z\;:\; n>s\}$
and~$L_{s_1,s}$ is defined in~\eqref{defop}. Here above, we use the notation~$L_{1,0}:=\textnormal{id}$, being~$\textnormal{id}$
the identity operator; see Proposition~\ref{limit s to 0} for a proof of the fact that this is consistent with a limit procedure.

The measures~$\nu_s$ and~$\sigma_s$ are assumed to satisfy~\eqref{definition measure}, \eqref{assumption basis},
and~\eqref{simple ellipticityBIS} for~$s>0$, and we set
$$\sigma_0(d\theta):= \frac{1}{{\mathcal{H}}^{N-1}(\bS^{N-1})}d\cH^{N-1}_{\theta},$$
which in view of Lemma~\ref{lem:limms} below is again consistent with the limit procedure.

We stress that condition~\eqref{simple ellipticityBIS} is actually in place only when the support of~$\mu$ is unbounded, and can be dropped otherwise
(indeed, if the support of~$\mu$ is bounded, 
one can just modify~$\sigma_s$ for large~$s$ without affecting the problem).

We assume moreover that
\begin{equation}
\label{Mass} \begin{split}
&{\mbox{for every~$A\in \cB(\R^N)$, the map~$[0,\infty)\ni s\mapsto \nu_s(A)$ is measurable}}\\&{\mbox{with respect to~$\mu^+$ and~$\mu^-$.}}\end{split}\end{equation} and that
\begin{equation}\label{Eass}\begin{split}
&{\mbox{there exists~$\tilde{s}\in[s_\ast,+\infty)$ such that~$\mu^+([s_\ast,\tilde{s}])>0$ and}}\\&
\inf_{e\in\bS^{N-1}}
\int_{\bS^{N-1}} |e\cdot\theta|^{2s_\ast}\,\tilde{\sigma}_{\tilde{s}}
>0,\qquad 
{\mbox{where}}\qquad
\tilde{\sigma}_{\tilde{s}}(A):=\frac1{\mu^+([s_\ast,\tilde{s}])}\int_{[s_\ast,\tilde{s}]}\sigma_s(A)\,\mu^+(ds).
\end{split}\end{equation}

We present some remarks on the above assumptions.
\begin{enumerate}
\item Since~$\sigma_s$ are probability measures for all~$s>0$, the measurability assumption~\eqref{Mass} implies that~$Lu(x)$ is well-defined for~$u\in C^{\infty}_c(\R^N)$ as soon as~$\mu([m,+\infty))=0$ for some~$m>0$, see~\cref{lem:evaluation}.
\item Notice that the existence of~$\tilde{s}\in[s_\ast,+\infty)$ such that~$\mu^+([s_\ast,\tilde{s}])>0$ is guaranteed by the assumption on~$\mu$
in~\eqref{mu-assumption1}. The additional ellipticity assumption in~\eqref{Eass} is crucial to setup a variational framework in a Hilbert space which contains the fractional Sobolev space of order~$2s_\ast$, see Lemma~\ref{prop of x omega part1} below.
\end{enumerate}

We will provide in Section~\ref{examples} several explicit examples of superposition
operators comprised in our setting. To name a few interesting particular cases, let us point out the following examples of superposition operators covered in the above framework (see Section~\ref{examples} for details):
\begin{enumerate}
	\item The choices~$\mu:=\delta_{2}+a\delta_1$ with~$a\leq 1$ and~$\nu_s(dz):=|z|^{-N-2s}\,dz$ lead to the operator $L=\Delta^2+a\Delta$, which is associated with the clamped plate problem with tension.
	\item With~$\mu:=\delta_{a}+\delta_b$ for real
	numbers~$a$, $b>0$ and~$\nu_s(dz):=|z|^{-N-2s}\,dz$, the associated operator is given by~$L=(-\Delta)^a+(-\Delta)^b$, which can be either purely local (for~$a$, $b\in \N$), or purely nonlocal (for~$a$, $b\notin\N)$, or may represent a local-nonlocal interaction (for~$a\in\N$ and~$b\notin\N$). 
	\item Let $s_1,\ldots,s_N>0$, $\sigma_{s_k}:=\delta_{e_k}$ for~$k=1,\ldots,N$, and~$\mu:=\delta_{s_1}+\ldots \delta_{s_N}$. Then~$L=(-\partial_{x_1}^2)^{s_1}+\ldots+ (-\partial_{x_N}^2)^{s_N}$ is a sum of operators of different integer or fractional order, which only act on the respective axis.
\end{enumerate}

In this spirit,
we point out that, besides being able to deal with operators of arbitrary order,
our setting is more general than the one in~\cite{DPSV23} even in the case~$s\in(0,1]$,
since we do not deal only with the case of fractional Laplacians
and several nonlocal operators can be treated simultaneously
within the framework presented here.\medskip

We recall that studying fractional problems in a high generality accounting for anisotropies and inhomogeneity of the substratum is of great importance not only for real-world models, but also in view of significant challenges in the development of the mathematical theories. These difficulties arise, for instance, in the lack of extension methods, which are in turn a fundamental tool to obtain monotonicity formulas (see e.g.~\cite{MARV} and the references therein). In this spirit, while the general setting considered here complicates several technical aspects of the statements and the proofs presented here, we think that it is advantageous to find arguments that are robust enough to treat all these cases simultaneously and detect specific features of the measures and the functional spaces which impact the main structure of the problem (as in
Lemma~\ref{lemma remark xyz} and Theorems~\ref{special construction}, \ref{lem:strano09876} and~\ref{0qiwdofr09my5062pySgqweSth} below).

Moreover, in terms of probabilistic interpretation of nonlocal problems, the study of nonlocal operators subject to general measures is important to describe {L}\'{e}vy operators, which are the generators of strongly continuous semigroups of translation-invariant Markov operators (see~\cite{MR4149690, MR4490672, MR4884561} and the references therein).

\subsection{Existence theory for superposition operators of arbitrary order}

We now develop an existence theory for a class of nonlinear problems involving
superposition operators of arbitrary (possibly fractional) order. 

As in~\cite{DPSV23}, by~\eqref{mu-assumption1} we can find~$s_{\#}\in[s_{\ast},+\infty)$ such that 
\begin{equation}\label{mu-assumption1b}
	\mu^+([s_{\#},+\infty))>0
\end{equation}
which will play the role for our critical exponent. For this, we assume throughout this section that
\begin{equation}\label{final assumption}
\text{$\nu_{s_{\#}}$ satisfies~\eqref{Eass} (with~$s_{\#}$ in place of~${s}_\ast$).}
\end{equation}

Let~$\Omega\subset \R^N$ be an open set and
\begin{equation}\label{2 sharp}
2^{\ast}:=\left\{\begin{aligned} &\frac{2N}{N-2s_{\#}}&&\text{if~$N>2s_{\#}$,}\\ 
&p && \text{if~$N\leq 2s_{\#}$, where~$p\in(2,+\infty)$ is arbitrary large but fixed.}
\end{aligned}\right.
\end{equation} 

Furthermore, we assume that
\begin{equation}\label{nonempty xomega}
\text{there exists a nontrivial function in~$\cX(\Omega)$,}
\end{equation}
where~$\cX(\Omega)$ is the energy space given in Definition~\ref{def:chiomega}.
Of course, \eqref{nonempty xomega} is
trivially satisfied if~$\mu^+([t,+\infty))=0$ for some~$t>0$, since then~$C^{\infty}_c(\Omega)\subset \cX(\Omega)$. 

We stress that condition~\eqref{nonempty xomega} cannot be removed, since, in our general setting, {\em there exist finite measures~$\mu^+$ on~$[0,+\infty)$ for which the only smooth function with compact support contained in~$\cX(\Omega)$ is the one that vanishes identically}, as we will show in Theorem~\ref{lem:strano09876}.

A thorough discussion about the energy space and assumption~\eqref{nonempty xomega} will be provided in Section~\ref{defeber859604}. In particular,
we will show that there exist finite measures~$\mu^+$ for which {\em some, but not all, smooth functions with compact support possess finite norm}, see Theorem~\ref{special construction}.

In fact, {\em
the existence of smooth and compactly supported functions
with infinite norm is a characterizing property of finite measures~$\mu^+$
with unbounded support}: as showcased in Theorem~\ref{0qiwdofr09my5062pySgqweSth},
whenever the support of~$\mu^+$ is unbounded, one can construct
a smooth function with compact support and infinite norm,
and this does not even require the strong ellipticity condition~\eqref{ellipticity}.

We now discuss two different situations to find nontrivial solutions to the problem
\begin{eqnarray}\label{prob-gen}
	\left\{\begin{aligned}
		Lu&=f(x,u)&&\text{in~$\Omega$,}\\
		u&=0 &&\text{in~$\R^N\setminus \Omega$.}
	\end{aligned}\right.
\end{eqnarray}

Following~\cite{SV12}, we first suppose that~$f:\Omega\times \R\to\R$ is a Carath\'eodory function satisfying the following:
\begin{equation}\label{nonlinearprob-assumptions}
\begin{split}
&\text{there exist~$a_1$, $a_2>0$ and~$q\in(2,2^{\ast})$ such that, for a.e.~$x\in \Omega$ and~$t\in \R$,}\\
&\qquad |f(x,t)|\leq a_1+a_2|t|^{q-1};\\
& \lim_{|t|\to0}\frac{f(x,t)}{|t|}=0\quad \text{uniformly in~$x\in \Omega$;}\\
&\text{there exist~$\mu>2$ and~$r>0$ such that, for a.e.~$x\in \Omega$ and~$t\in \R$ with~$|t|\geq r$,}\\
&\qquad 0<\mu F(x,t)\leq tf(x,t),
\end{split}
\end{equation}
where the function~$F$ is the primitive of~$f$ with respect to the second variable, that is
$$
F(x,t)=\int_0^t f(x,\tau)\,d\tau.
$$
In this setting, we have the following result.

\begin{thm}\label{mountain pass solution}
Let~$\Omega\subset \R^N$ be an open bounded set. 
Let~$\mu$ satisfy~\eqref{mu-assumption1},
\eqref{mu-assumption2} and~\eqref{mu-assumption1b}
and~$L$ be as in~\eqref{defi:superposition op}.
Let~$f:\Omega\times \R\to\R$ be a Charath\'eodory function satisfying~\eqref{nonlinearprob-assumptions}.
Suppose that~\eqref{nonempty xomega} holds true.

Then, if~$\gamma$ is sufficently small, \eqref{prob-gen} admits a nontrivial mountain pass  solution.
\end{thm}

We think that it is an interesting open problem to determine
whether or not one can obtain mountain pass solutions with
a prescribed sign. In particular,
in~\cite[Corollary~13]{SV12} both positive and negative solutions
are constructed, but the technique used there is not immediately
applicable to the general case treated in this paper.
Specifically, the proof of~\cite[Corollary~13]{SV12}
uses the fact that the positive and negative parts of, say, smooth functions
belong to the energy space. However, 
when~$s>\frac32$ the maps~$u\mapsto u^+$ and~$u\mapsto u^-$ are not well-defined in~$H^s(\R^N)$ (see e.g.~\cite{M89,MN19}),
making this type of arguments unavailable in the generality considered here.

The second type of right-hand side that we consider is the one with jumping nonlinearities, which reads as
\begin{eqnarray}\label{nonlinearprob}
\left\{\begin{aligned}
Lu&=bu^+-au^-+|u|^{2^\ast-2}u&&\text{in~$\Omega$,}\\
u&=0 &&\text{in~$\R^N\setminus \Omega$,}
\end{aligned}\right.
\end{eqnarray}
where~$2^{\ast}$ is as in~\eqref{2 sharp}. 

We point out that when~$a=b$, the right-hand side in
the equation in~\eqref{nonlinearprob} boils down to a Br\'ezis-Nirenberg
nonlinearity. But when~$a\neq b$ some care is needed in the analysis of~\eqref{nonlinearprob}, due to the lack
of regularity of the source term. 

We mention that the study of
nondifferentiable nonlinearities has a very consolidated tradition in terms of mathematical theory, see
e.g.~\cite{MR499709, MR1181350, MR1215262, MR1725568}, and also in light of
concrete problems, such as singular perturbations of classical nonlinearities, plasma problems, mathematical biology, etc., see~\cite{MR1303035}.

Note that our superposition operator has a {\em discrete spectrum} consisting
of (Dirichlet) eigenvalues 
$$
0<\lambda_1<\ldots < \lambda_{n}< \ldots\to +\infty\quad\text{for~$n\to+\infty$},
$$
thanks to the forthcoming \cref{prop of x omega part2}, where each eigenvalue has a finite multiplicity. The goal now is to choose~$a$ and~$b$ such that~$Lu-bu^++au^-$ has a nontrivial solution. This leads to the \textit{Dancer-Fu\u{c}\'ik spectrum}, see Section~4
in~\cite{PS13}, Section~2
in~\cite{PS23}, and the introduction in~\cite{DPSV23}.

The latter spectrum is a closed subset of~$\R^2$ and it satisfies the following properties.
Let~$\{\lambda_k\}_{k\in \N}$ be the eigenvalues of~$L$ and let~$l\in \N$ with~$l\geq 2$. There are two continuous strictly decreasing functions~$A_{l-1}$ and~$B_l$ such that:
\begin{itemize}
\item it holds~$A_{l-1}\leq B_{l}$ in~$(\lambda_{l-1},\lambda_{l+1})$,
\item $A_{l-1}(\lambda_l)=\lambda_l=B_{l}(\lambda_l)$,
\item for all~$a\in(\lambda_{l-1},\lambda_{l+1})$ both~$(a,A_{l-1}(a))$ and~$(a,B_l(a))$ belong to the Dancer-Fu\u{c}\'ik spectrum,
\item if~$a$, $b\in(\lambda_{l-1},\lambda_{l+1})$ with either~$b<A_{l-1}(a)$ or~$b>B_l(a)$, then~$(a,b)$ does not belong to the Dancer-Fu\u{c}\'ik spectrum.
\end{itemize}
Analogously to~\cite{DPSV23}, we let~$l\in\N$ with~$l\geq 2$ and assume that
\begin{equation}\label{D-F-ab}
b<A_{l-1}(a).
\end{equation}
The part in~$Q_{l}:=(\lambda_{l-1},\lambda_{l+1})^2$ where the couple~$(a,b)$ satisfies~\eqref{D-F-ab} can be visualized through the exemplification in~\cite[Figure~2]{DPSV23}.

In this setting, we have the following result.

\begin{thm}\label{critical case}
Let~$\Omega\subset \R^N$ be an open bounded set.
Let~$\mu$ satisfy~\eqref{mu-assumption1},
\eqref{mu-assumption2} and~\eqref{mu-assumption1b}.

Let~$L$ be as in~\eqref{defi:superposition op}.
Let~$l\in\N$ with~$l\geq 2$, and let~$(a,b)\in Q_l:=(\lambda_{l-1},\lambda_{l+1})^2$. Assume that~\eqref{D-F-ab} holds true.

Then, if~$\gamma$ is sufficiently small, %there exists~$\gamma_0>0$, depending only on~$N$, $\Omega$, $s_{\#}$, $a$, and~$b$, such that if~$\gamma\in[0,\gamma_0]$, then 
problem~\eqref{nonlinearprob} admits a nontrivial solution.
\end{thm}

\begin{remark} In Theorems~\ref{mountain pass solution}
and~\ref{critical case} the smallness of~$\gamma$ 
guarantees that the contribution from~$\mu^-$ is ``suitably small'' compared to the one
from~$\mu^+$, thus allowing energy methods to be implemented.
\end{remark}

\begin{remark}
We emphasize that problems~\eqref{prob-gen} and~\eqref{nonlinearprob} are combined with \textit{Dirichlet boundary conditions}, which are encoded in the function space~$\cX(\Omega)$, see also \cref{dirichlet prob} in this regard. In particular, if~$\mu$ has positive support in~$(1,\infty)$, the fact that the functions are in $\cX(\Omega)$ does not only give that $u=0$ in $\R^N\setminus \Omega$, but that also some derivatives in the normal direction have vanishing traces depending on $\mu$ and $\nu_s$. For the classical cases involving powers of the Laplacian, see e.g.~\cite{zbMATH05712793} for polylaplacians and~\cite{zbMATH06994026, zbMATH07024011} for higher-order fractional Laplacians.
\end{remark}

\subsection{Organization of the paper}
Section~\ref{bcnowiur895u456789} contains some basic facts about the probability measures~$\sigma_s$.

In Section~\ref{sec:1} we provide a detailed analysis of the operator~$L_{m,s}$
and in Section~\ref{sec:var23} we discuss the functional setting.

Section~\ref{examples} contains several explicit examples of
superposition operators and some characterizations of
the measures and the functional spaces under consideration.

Section~\ref{sec:45} presents the functional setting associated with general superposition operators of mixed order.

Section~\ref{sec:67} contains the proofs of Theorems~\ref{mountain pass solution}
and~\ref{critical case}.

%The paper ends with some appendices, where we collect some technical results.

\section{On the measures~\texorpdfstring{$\sigma_s$}{sigma-s}}\label{bcnowiur895u456789}

Here we provide some basics results on the measures~$\sigma_s$.

We start with the following observation:

\begin{lemma}\label{lemma:nuovo}
Let~$\sigma$ be a probability measure on~$\bS^{N-1}$. Then, for all~$\alpha\ge0$,
\begin{equation}\label{nelladimsimpleellipticitybasic}
\max_{e\in\bS^{N-1}}\int_{\bS^{N-1}}|e\cdot \theta|^{\alpha}\,\sigma(d\theta)>0.
\end{equation} \end{lemma}

\begin{proof}
Suppose by contradiction that, for all~$e\in\bS^{N-1}$,
$$ \int_{\bS^{N-1}}|e\cdot \theta|^{\alpha}\,\sigma(d\theta)=0.$$

We consider the surface (Hausdorff) measure~${\mathcal{H}}^{N-1}$ on~$\bS^{N-1}$ 
and we use Fubini's Theorem (for finite measures, see e.g.~\cite[Theorem~1.22]{MR3409135}) to see that
\begin{eqnarray*}&&
0=\int_{\bS^{N-1}} \left( \;\int_{\bS^{N-1}}|e\cdot \theta|^{\alpha}\,\sigma(d\theta)\right)\,d{\mathcal{H}}^{N-1}_e=
\int_{\bS^{N-1}} \left(\; \int_{\bS^{N-1}}|e\cdot \theta|^{\alpha}\,
d{\mathcal{H}}^{N-1}_e\right)\,\sigma(d\theta)\\&&\qquad
=\int_{\bS^{N-1}} \left(\; \int_{\bS^{N-1}}|e_1|^{\alpha}\,
d{\mathcal{H}}^{N-1}_e\right)\,\sigma(d\theta)= \int_{\bS^{N-1}}|e_1|^{\alpha}\,
d{\mathcal{H}}^{N-1}_e
.
\end{eqnarray*}
This establishes the desired contradiction and completes the proof.
\end{proof}

We now show that
the assumptions~\eqref{simple ellipticityBIS} and~\eqref{simple ellipticity} are actually equivalent, and this justifies the fact that~\eqref{simple ellipticity} is assumed throughout the paper.

\begin{lemma}\label{lem:equi00}
For all~$s>0$, let~$\sigma_s$ be a probability measure and
let~$e_s\in\bS^{N-1}$ be as in~\eqref{bvncmx9e76954uy95}.

Then, the statements in~\eqref{simple ellipticityBIS} and~\eqref{simple ellipticity} are equivalent.
\end{lemma}

\begin{proof} The desired result will follow once we prove that~\eqref{simple ellipticityBIS} implies~\eqref{simple ellipticity},
being the other implication obvious.

For this, suppose, for the sake of contradiction, that there exists a sequence~$s_j$ such that
\begin{equation}\label{xm9058Ln0-9}
\sup_{e\in\bS^{n-1}}
\int_{\bS^{n-1}}|e\cdot\theta|^{2s_j}\sigma_{s_j}(d\theta)\le\frac1j.\end{equation}
We distinguish two cases, either~$s_j$ is a bounded sequence or not.
First, if~$s_j$ is bounded, we suppose, up to a subsequence, that~$s_j$ converges to some~$s_\infty\in[0,+\infty)$ as~$j\to+\infty$. 
Without loss of generality, we can assume also that~$|s_j-s_\infty|\le1$.

We pick~$\epsilon>0$ and, for all~$e\in \bS^{n-1}$, let
$$S_{\epsilon,e}:=\big\{ \theta\in\bS^{n-1}{\mbox{ s.t. }}|e\cdot\theta|\le\epsilon\big\}.$$
We observe that, when~$e\in \bS^{n-1}$ and~$\theta\in\bS^{n-1}\setminus S_{\epsilon,e}$,
\begin{eqnarray*}&&
\big| |e\cdot\theta|^{2s_j}-|e\cdot\theta|^{2s_\infty}\big|
=|e\cdot\theta|^{2s_\infty}\big||e\cdot\theta|^{2(s_j-s_\infty)}-1\big|
\le
\big||e\cdot\theta|^{2(s_j-s_\infty)}-1\big|
\\&&\qquad=
\big|\exp\big(2(s_j-s_\infty)\log |e\cdot\theta|\big)-1\big|\\&&\qquad=\left|\;
2(s_j-s_\infty)\log |e\cdot\theta| \int_0^1\exp\big(2(s_j-s_\infty)\log |e\cdot\theta|t\big)\,dt\right|
\le C_\epsilon\,|s_j-s_\infty|,
\end{eqnarray*} with~$C_\epsilon:=2\epsilon^{-2}|\log\epsilon|$.

Therefore, since~$\sigma_{s_j}$ is a probability measure,
\begin{eqnarray*}&&
\left|\;\int_{\bS^{n-1}\setminus S_{\epsilon,e}}|e\cdot\theta|^{2s_j}\sigma_{s_j}(d\theta)-
\int_{\bS^{n-1}\setminus S_{\epsilon,e}}|e\cdot\theta|^{2s_\infty}\sigma_{s_j}(d\theta)\right|\\&&\qquad\quad\le
C_\epsilon\,|s_j-s_\infty|\,\int_{\bS^{n-1}\setminus S_{\epsilon,e}}\sigma_{s_j}(d\theta)\le
C_\epsilon\,|s_j-s_\infty|.\end{eqnarray*}
This, in tandem with~\eqref{xm9058Ln0-9}, returns that
\begin{equation}\label{9mc8v5buMTFS-7n0}\begin{split}
0&\ge\lim_{j\to+\infty}\sup_{e\in\bS^{n-1}}\int_{\bS^{n-1}\setminus S_{\epsilon,e}}|e\cdot\theta|^{2s_j}\sigma_{s_j}(d\theta)\\&\ge
\lim_{j\to+\infty}\left(\sup_{e\in\bS^{n-1}}\int_{\bS^{n-1}\setminus S_{\epsilon,e}}|e\cdot\theta|^{2s_\infty}\sigma_{s_j}(d\theta)-C_\epsilon\,|s_j-s_\infty|\right)\\&=\lim_{j\to+\infty}\sup_{e\in\bS^{n-1}}\int_{\bS^{n-1}\setminus S_{\epsilon,e}}|e\cdot\theta|^{2s_\infty}\sigma_{s_j}(d\theta)
.\end{split}\end{equation}

Also, by the weak compactness of probability measures (see e.g.~\cite[Theorems~6.1 and~6.4, Chapter~2]{MR226684}),
up to a subsequence we may assume that~$\sigma_{s_j}$ converges weakly
to a probability measure~$\sigma_\infty$ and therefore, for all~$e\in\bS^{n-1}$,
$$\lim_{j\to+\infty}\int_{\bS^{n-1}\setminus S_{\epsilon,e}}|e\cdot\theta|^{2s_\infty}\sigma_{s_j}(d\theta)=\int_{\bS^{n-1}\setminus S_{\epsilon,e}}|e\cdot\theta|^{2s_\infty}\sigma_\infty(d\theta).$$

{F}rom this and~\eqref{9mc8v5buMTFS-7n0} we deduce that, for all~$e\in\bS^{n-1}$,
\begin{equation*}\begin{split}
0&=\int_{\bS^{n-1}\setminus S_{\epsilon,e}}|e\cdot\theta|^{2s_\infty}\sigma_\infty(d\theta)\\&=
\int_{\bS^{n-1}}|e\cdot\theta|^{2s_\infty}\sigma_\infty(d\theta)
-\int_{S_{\epsilon,e}}|e\cdot\theta|^{2s_\infty}\sigma_\infty(d\theta)\\&\ge
\int_{\bS^{n-1}}|e\cdot\theta|^{2s_\infty}\sigma_\infty(d\theta)
-\int_{S_{\epsilon,e}}\epsilon^{2s_\infty}\sigma_\infty(d\theta)\\&\ge
\int_{\bS^{n-1}}|e\cdot\theta|^{2s_\infty}\sigma_\infty(d\theta)
-\epsilon^{2s_\infty}.
\end{split}\end{equation*}
Hence, since~$\epsilon$ is arbitrary,
$$\int_{\bS^{n-1}}|e\cdot\theta|^{2s_\infty}\sigma_\infty(d\theta)=0$$
and this holds for all~$e\in\bS^{n-1}$, which is in contradiction with
the statement in Lemma~\ref{lemma:nuovo}.

Now we suppose that the sequence~$s_j$ is unbounded. Up to a subsequence, we have that~$s_j$ is divergent. This and~\eqref{xm9058Ln0-9}
yield that
$$\liminf_{s\to+\infty}\sup_{e\in\bS^{n-1}}\int_{\bS^{n-1}}|e\cdot\theta|^{2s}\sigma_s(d\theta)=0,$$
in contradiction with~\eqref{simple ellipticityBIS}.
\end{proof}

\section{On the operator~\texorpdfstring{$L_{m,s}$}{Lms}}\label{sec:1}

In this section we provide a systematic study of the operator~$L_{m,s}$. 

\subsection{Some notation and preliminary results on~\texorpdfstring{$L_{m,s}$}{Lms}}
If not stated otherwise, we assume throughout this section that~$\nu_s$ is given with respect to a probability measure~$\sigma_s$ in~$\bS^{N-1}$ as in~\eqref{definition measure}, \eqref{assumption basis} and~\eqref{simple ellipticity}.

In the following, for a fixed number~$s>0$ we let
\begin{equation}\label{sis2s3s4}\begin{split}
&s_0:=\max\{n\in \Z\;:\; n<s\},\qquad s_1:=\min\{n\in\Z\;:\; n>s\}\\
&\text{and}\quad \tau:=s-s_0\in(0,1].\end{split}
\end{equation}

Moreover, let~$\cL^1_s$ be the space of
locally integrable functions~$u :\R^N\to\R$ such that
\begin{equation}\label{dewfhweut765432}
\int_0^{+\infty} \int_{\bS^{N-1}} |u(r\theta)| \min\{1,r^{-1-2s}\} \,\sigma_s(d\theta)\, dr <+\infty.
\end{equation}

Furthermore, given~$x_0\in\R^N$, $R>0$ and~$\epsilon>0$,
we say that~$u\in C^{2s+\epsilon}(B_R(x_0))$ if
\begin{equation*}
\begin{split}
&{\mbox{$u\in C^{2s_0,2\tau+\epsilon}(B_R(x_0))$ if~$\tau< \frac12$,}}\\
&{\mbox{$u\in C^{2s_0+1,2\tau-1+\epsilon}(B_R(x_0))$ if~$\frac12\leq\tau<1$,}}\\
&{\mbox{or~$u\in C^{2s_0+2,\epsilon}(B_R(x_0))$ if~$\tau=1$.}}
\end{split}\end{equation*}

With this notation, we have the following observation:

\begin{lemma}\label{lem:evaluation}
Let~$x_0\in \R^N$, $R>0$ and~$\epsilon>0$.
Let~$u\in C^{2s+\epsilon}(B_R(x_0))\cap \cL^1_s$
and let~$x\in B_R(x_0)$.

Then, we have that~$|L_{m,s}u(x)|\le C$, for some~$C>0$.
\end{lemma}

\begin{proof}
The regularity assumptions on~$u$ give that, for all~$x\in B_R(x_0)$ and~$y\in B_\eta$
(for some~$\eta\in(0,1)$ sufficiently small),
$$
|\delta_mu(x,y)|\leq C|y|^{2s+\epsilon},
$$ 
for some~$C>0$.

As a consequence,
\begin{eqnarray*}
&&\left|\;\int_{B_\eta}\delta_mu(x,y) \,\nu_s(dy)\right|
\le \int_{B_\eta}|\delta_mu(x,y)| \,\nu_s(dy)
\le C\int_{B_\eta}|y|^{2s+\epsilon} \,\nu_s(dy)\\
&&\qquad = C\int_{0}^{\eta} \int_{\bS^{N-1}}r^{2s+\epsilon}r^{-1-2s}\,\sigma_s(d\theta) \,dr= \frac{C \eta^\epsilon}{\epsilon}.
\end{eqnarray*}

In addition,
\begin{eqnarray*}
&&\left|\;\int_{\R^N\setminus B_\eta}\delta_mu(x,y) \,\nu_s(dy)\right|
\le \int_{\R^N\setminus B_\eta}|\delta_mu(x,y)| \,\nu_s(dy)
\\&&\qquad \le C\left(\;|u(x)|\int_{\R^N\setminus B_\eta}\nu_s(dy)+
\int_{\R^N\setminus B_\eta}|u(y)| \,\nu_s(dy)\right) \\
&&\qquad 
\le C\left(\;\|u\|_{L^\infty(B_R(x_0))}\int_{\eta}^{+\infty} \int_{\bS^{N-1}}
r^{-1-2s}\,\sigma_s(d\theta) \,dr+\int_{\eta}^{+\infty} \int_{\bS^{N-1}}
|u(r\theta)|r^{-1-2s} \,\sigma_s(d\theta) \,dr\right) \\
&&\qquad=
C\left(\;\frac{ \|u\|_{L^\infty(B_R(x_0))}}{2s\,\eta^{2s}}
+\int_{\eta}^{+\infty} \int_{\bS^{N-1}}
|u(r\theta)|r^{-1-2s} \,\sigma_s(d\theta) \,dr\right).
\end{eqnarray*}

Gathering these pieces of information and recalling also~\eqref{dewfhweut765432}, we conclude that
\begin{eqnarray*}
&&|L_{m,s}u(x)|= \frac{c_{m,s}}{2} \left|\;\int_{\R^N}\delta_mu(x,y) \,\nu_s(dy)\right|\\
&&\qquad \le C\left(\;\frac{ \eta^\epsilon}{\epsilon}
+\frac{ \|u\|_{L^\infty(B_R(x_0))}}{2s\,\eta^{2s}}
+\int_{\eta}^{+\infty} \int_{\bS^{N-1}}
|u(r\theta)|r^{-1-2s} \,\sigma_s(d\theta) \,dr\right)
\\&&\qquad \le C,\end{eqnarray*}
as desired.
\end{proof}

See also the Introduction and Appendix~A in~\cite{GH23} on the tail space~$L^1(\R^N,\nu_s(x)dx)$ in this context.

As mentioned in the Introduction, we next show that if~$\sigma_s$ is not a probability measure but an unbounded nontrivial measure, then~$L_{m,s}u(x)$ may not be well-defined.

\begin{lemma}\label{lemma remark xyz}
Let~$\nu_s$ be as in \cref{definition measure}, where~$\sigma_s$ is a measure such that~$\sigma_s(\bS^{N-1})=+\infty$.

Then, there exists~$u\in C^{\infty}_c(\R^N)$ such that~$L_{m,s}u(x)$ is not defined for some~$x\in \R^N$.
\end{lemma}

\begin{proof}
Let~$\xi\in C^{\infty}([0,+\infty), [0,1])$ be nonincreasing and such that~$\xi(t)=1$ for all~$t\in[0,1]$ and~$\xi(t)=0$ for all~$t\geq 2$. Let~$u(x)=\xi(|x|)$. 
Then~$u\in C^{\infty}_c(\R^N, [0,1])$ with~$\supp\,u\subset B_2$
and~$u(x)=1$ for all~$x\in B_1$.

Also, $u$ is radially nonincreasing. As a result,
\begin{eqnarray*}
\delta_mu(0,y)=\sum_{k=-m}^m(-1)^k\binom{2m}{m-k}u(ky)\ge0.
\end{eqnarray*}

We also observe that if~$y\in\R^N\setminus B_4$ 
then~$u(ky)=0$ for all~$k\in\Z\setminus\{0\}$. Therefore,
by~\eqref{defdelta} we deduce that
\begin{equation*}
\delta_mu(0,y)=\binom{2m}{m}u(0)=\binom{2m}{m}.
\end{equation*}

Thus,
\begin{eqnarray*}&&
\frac{2}{c_{m,s}}L_{m,s}u(0)=\int_{B_4}\delta_mu(0,y)\,\nu_s(dy)+\int_{\R^N\setminus B_4}\binom{2m}{m}\nu_s(dy)
\\&&\qquad
\geq \binom{2m}{m} \int_2^{+\infty}t^{-1-2s}\sigma_s(\bS^{N-1})\, dt=+\infty,
\end{eqnarray*}
which implies the desired result.
\end{proof}

The next result aims at showing the variability of~$m$ as long as~$m>s$. For this, we need the following lemma.

\begin{lemma}\label{lem:constant}
Let~$n>m>s>0$ and let
\begin{equation}\label{padef00}
P_a:=\sum_{k=1}^a(-1)^k\binom{2a}{a-k}k^{2s}\qquad\text{ for~$a>s$.}
\end{equation}

Then, we have that~$c_{n,s}P_n=c_{m,s}P_m$.
\end{lemma}

\begin{proof}
We first observe that
\begin{eqnarray*}
P_a&=&\frac12\sum_{k=1}^a(-1)^k\binom{2a}{a-k}k^{2s}+
\frac12\sum_{k=1}^a(-1)^k\binom{2a}{a-k}k^{2s}\\
& =& \frac12\sum_{k=1}^a(-1)^k\binom{2a}{a-k}k^{2s}+
\frac12\sum_{k=-a}^{-1}(-1)^{-k}\binom{2a}{a+k}(-k)^{2s}\\
& = &\frac12\sum_{k=1}^a(-1)^k\binom{2a}{a-k}k^{2s}+
\frac12\sum_{k=-a}^{-1}(-1)^{k}\binom{2a}{a-k}k^{2s}\\
&=&\frac12\sum_{k=-a}^a(-1)^k\binom{2a}{a-k}k^{2s}.
\end{eqnarray*}

Now, let~$f(t):=\exp(it)$ for~$t\in \R$. Then, we have that
\begin{equation}\label{agaapa0}
\delta_mf(0,t)=2^m(1-\cos t)^m.\end{equation}
For the sake of completeness, we provide the proof of~\eqref{agaapa0}
in Appendix~\ref{app:agaapa0}.

Hence, recalling~\eqref{constcms00} and performing the changes of variable~$z:=y/k$ and~$w:=jz$,
\begin{align*}
	\frac{2P_n}{c_{m,s}}&=P_n\int_{\R^N}\delta_mf(0,y_1)\,\nu_s(dy)\\
	&=\frac12\sum_{k=-n}^n(-1)^k\binom{2n}{n-k}k^{2s}
	\int_{\R^N}\delta_mf(0,y_1)\,\nu_s(dy)\\
	&=\frac12\sum_{k=-n}^n(-1)^k\binom{2n}{n-k}\int_{\R^N}\delta_mf(0,kz_1)\,\nu_s(dz)\\
	&=\frac12\int_{\R^N}\sum_{k=-n}^n\sum_{j=-m}^{m}(-1)^k\binom{2n}{n-k}(-1)^j\binom{2m}{m-j}f(0,kjz_1)\,\nu_s(dz)\\
	&=\frac12\sum_{j=-m}^{m}(-1)^j\binom{2m}{m-j}\int_{\R^N}\delta_nf(0,jz_1)\,\nu_s(dy)\\
	&=\frac12\sum_{j=-m}^{m}(-1)^j\binom{2m}{m-j}j^{2s}
	\int_{\R^N}\delta_nf(0,w_1)\,\nu_s(dw)\\
	&=P_m\int_{\R^N}\delta_nf(0,w_1)\,\nu_s(dw)\\
	&=\frac{2P_m}{c_{n,s}}.\qedhere
\end{align*}
\end{proof}

\begin{lemma}\label{lem:independence of m}
Let~$n>m>s>0$, $x_0\in \R^N$, and~$R$, $\epsilon>0$.
Let~$u\in C^{2s+\epsilon}(B_R(x_0))\cap \cL^1_s$.

Then, for all~$x\in B_R(x_0)$,
$$
L_{n,s}u(x)=L_{m,s}u(x).
$$	
\end{lemma}

\begin{proof}
In light of Lemma~\ref{lem:evaluation}, since~$u\in C^{2s+\epsilon}(B_R(x_0))\cap \cL^1_s$, we have that~$L_{n,s}u(x)$ is well defined for all~$x\in B_R(x_0)$.

Now, let~$P_a$ be as in~\eqref{padef00}. Then similarly as in the proof of Lemma~\ref{lem:constant} we have
	\begin{align*}
		2P_n&\int_{\R^N}\delta_mu(x,y)\,\nu_s(dy)=\frac12\sum_{j=-n}^n(-1)^j\binom{2n}{n-j}\int_{\R^N}\delta_mu(x,jy)\,\nu_s(dy)\\
		&=\frac12\int_{\R^N} \sum_{j=-n}^n\sum_{k=-m}^m(-1)^k\binom{2m}{m-k}(-1)^j\binom{2n}{n-j}u(x+kjy)\,\nu_s(dy)\\
	&=\frac12\sum_{k=-m}^m(-1)^k\binom{2m}{m-k}\int_{\R^N}\delta_nu(x,ky)\,\nu_s(dy)=2P_m\int_{\R^N}\delta_nu(x,y)\,\nu_s(dy).
	\end{align*}
Thus the claim follows from Lemma~\ref{lem:constant}.
\end{proof}

\subsection{Limits as~\texorpdfstring{$s\to 0^+$ and as~$s\to m^-$}{s to 0+ and as s to m-}}
The following is to understand the limit as~$s\to 0^+$. 
For this, it turns out that the limit as~$s\to 0^+$ of
the quantity~$M_{s,\sigma_s}(e_s)$ given in~\eqref{bvncmx9e76954uy95}
plays a crucial role,
and we therefore provide the following result:

\begin{lemma}\label{lem:limms}
We have that
$$\lim_{s\to0^+}\max_{e\in\bS^{N-1}}\int_{\bS^{N-1}}|e\cdot\theta|^{2s}\,\sigma_s(d\theta)=1.$$
\end{lemma}

\begin{proof}
Suppose not. Then, recalling also~\eqref{usoduevoltefprse769043-0876}, there exist~$a\in[0,1)$ and an infinitesimal sequence~$s_j$ such that
$$\max_{e\in\bS^{N-1}}\int_{\bS^{N-1}}|e\cdot\theta|^{2s_j}\,\sigma_{s_j}(d\theta)\le a.$$
By the weak compactness of probability measures (see e.g.~\cite[Theorems~6.1 and~6.4, Chapter~2]{MR226684}), we can select a weakly convergent subsequence~$\sigma_{s_{j_k}}$ and denote by~$\sigma_*$ its weak limit, which is still a probability measure.

Also, given~$\epsilon>0$, we can suppose that~$s_{j_k}\in[0,\epsilon]$. Since, for all~$e$, $\theta\in\bS^{N-1}$, we have that~$|e\cdot\theta|\le1$, we infer that~$|e\cdot\theta|^{2s_{j_k}}\ge|e\cdot\theta|^{2\epsilon}$.

Accordingly, for all~$e\in\bS^{N-1}$,
$$\int_{\bS^{N-1}}|e\cdot\theta|^{2\epsilon}\,\sigma_{s_{j_k}}(d\theta)\le
\int_{\bS^{N-1}}|e\cdot\theta|^{2s_{j_k}}\,\sigma_{s_{j_k}}(d\theta)\le a.$$
Passing to the limit as~$k\to+\infty$ and using the weak convergence of the probability measures, we find that
$$\int_{\bS^{N-1}}|e\cdot\theta|^{2\epsilon}\,\sigma_*(d\theta)\le a,$$
and so
\begin{equation}\label{pk-dfe4560}
\int_{\{e\cdot\theta\ne0\}}|e\cdot\theta|^{2\epsilon}\,\sigma_*(d\theta)\le a.\end{equation}

We can now take the limit as~$\epsilon\to0^+$. Since, when~$e\cdot\theta\ne0$,
$$\lim_{\epsilon \to 0^+}|e\cdot\theta|^{2\epsilon}=1,$$
we deduce from~\eqref{pk-dfe4560} and Fatou's Lemma that
\begin{equation}\label{ALkl0-erfPMVbS40i}
\int_{\{e\cdot\theta\ne0\}}\,\sigma_*(d\theta)\le a.
\end{equation}

Now we consider the surface (Hausdorff) measure~${\mathcal{H}}^{N-1}$ on~$\bS^{N-1}$ and the product measure~$\sigma_* \times {\mathcal{H}}^{N-1}$ on~$(\bS^{N-1})^2$. We define
$$ \Sigma:=\big\{
{\mbox{$(\theta,e)\in(\bS^{N-1})^2$ s.t. $e\cdot\theta\ne0$}}
\big\}$$
and we use Fubini's Theorem (for finite measures, see e.g.~\cite[Theorem~1.22]{MR3409135}) to see that
\begin{eqnarray*}
\int_{\bS^{N-1}} \left( \,\int_{\{e\cdot\theta\ne0\}}\,\sigma_*(d\theta)\right)\,d{\mathcal{H}}^{N-1}_e=
\iint_\Sigma \big(\sigma_*(d\theta), \, d{\mathcal{H}}^{N-1}_e\big)=
\int_{\bS^{N-1}} \left(\, \int_{\{e\cdot\theta\ne0\}}\,
d{\mathcal{H}}^{N-1}_e\right)\,\sigma_*(d\theta).
\end{eqnarray*}
Also, the spherical equator has null surface measure, giving that
$$\int_{\{e\cdot\theta\ne0\}}\,d
{\mathcal{H}}^{N-1}_e=\int_{\bS^{N-1}}\,d
{\mathcal{H}}^{N-1}_e={\mathcal{H}}^{N-1}(\bS^{N-1}).$$
{F}rom these considerations, we arrive at
$$\int_{\bS^{N-1}} \left(\, \int_{\{e\cdot\theta\ne0\}}\,\sigma_*(d\theta)\right)\,d{\mathcal{H}}^{N-1}_e=
\int_{\bS^{N-1}} {\mathcal{H}}^{N-1}(\bS^{N-1})
\,\sigma_*(d\theta)={\mathcal{H}}^{N-1}(\bS^{N-1}).$$
This in tandem with~\eqref{ALkl0-erfPMVbS40i} returns that
$$ a{\mathcal{H}}^{N-1}(\bS^{N-1})=\int_{\bS^{N-1}} a\,d{\mathcal{H}}^{N-1}_e\ge{\mathcal{H}}^{N-1}(\bS^{N-1}),$$
which provides the desired contradiction.\end{proof}

We now check the behavior of the constant~$c_{1,s}$ as~$s\to 0^+$.

\begin{lemma}\label{rep of cms}
Let~$m\in \N$ and~$s\in(0,m)$. Then,
$$
c_{m,s}=2^{1-m}\left(\,M_{s,\sigma_s}(e_s)\int_0^{+\infty}\big(1- \cos\tau\big)^{m} \tau^{-1-2s}\,d\tau\right)^{-1}.
$$
\end{lemma}

\begin{proof}
From the definition of~$c_{m,s}$, we have, by using polar coordinates
and the change of variable~$\tau:=|e_s\cdot \theta| t$, that
\begin{align*}
\frac{2^{1-m}}{c_{m,s}}&= \int_{\R^N}\big(1-\cos(e_s\cdot y)\big)^{m}\,\nu_{s}(dy)\\
&=\int_0^{+\infty} t^{-1-2s}\int_{\bS^{N-1}} \big(1-\cos(e_s\cdot \theta t)\big)^m\,\sigma_s(d\theta)\, dt\\
&= M_{s,\sigma_s}(e_{s})\int_{0}^{+\infty}\big(1-\cos\tau\big)^{m}\tau^{-1-2s}\,d\tau.
\end{align*}
Rearranging this identity gives the claim.
\end{proof}

\begin{lemma}\label{lem:constant estimate}
For~$m\in \N$ it holds that
$$
\lim_{s\to0^+}\frac{c_{m,s}}{s}M_{s,\sigma_s}(e_s)=\frac{2}{-P_m},
$$ where~$P_m$ is given by formula~\eqref{padef00}. 
\end{lemma}

\begin{proof}
Using the expression of~$c_{m,s}$ as represented in \cref{rep of cms}, we see that, when~$m=1$,
$$ c_{1,s}^{-1}=M_{s,\sigma_s}(e_s)\int_{0}^{+\infty}\big(1-\cos\tau\big)\tau^{-1-2s}\,d\tau.$$
Moreover, we know (see e.g. formula~(2.15) in~\cite{MR4500878}) that
$$
\int_0^{+\infty} \big(1-\cos\tau\big){\tau^{-1-2s}}\, d\tau=\frac{\cos(\pi s)\,\Gamma(2-2s)}{2s (1-2s)}.
$$
Thus, we conclude that
\begin{eqnarray*}
c_{1,s}^{-1}=\frac{\cos(\pi s)\,\Gamma(2-2s)M_{s,\sigma_s}(e_s)}{2s (1-2s)}
\end{eqnarray*}
and therefore
\begin{eqnarray*}
\lim_{s\to0^+}\frac{c_{1,s}}{s}M_{s,\sigma_s}(e_s)=
\lim_{s\to0^+}\frac{2(1-2s)}{\cos(\pi s)\,\Gamma(2-2s)}=2.
\end{eqnarray*}
Since~$P_1=-1$, this gives the desired result when~$m=1$.

The claim for any~$m>1$ then follows from \cref{lem:constant}, noting that,
for~$m\geq 1>s>0$,
\begin{equation*}
c_{m,s}=\frac{c_{1,s}P_1}{P_m}=\frac{c_{1,s}}{-P_m}.
\qedhere \end{equation*}
\end{proof}

As a consequence of Lemma~\ref{lem:constant estimate}
we have the following result:

\begin{prop}\label{limit s to 0}
For any~$\beta\in(0,1)$ and~$u\in C^{\beta}_c(\R^N)$, it holds that
$$
 \lim_{s\to0^+}L_{m,s}u(x)=u(x)\quad\text{for all~$x\in \R^N$.}
$$
\end{prop}

\begin{proof}
Without loss of generality, we can suppose that~$x=0$. 
Let~$R>0$ be such that~$\supp\,u\subset B_{R}$ and let~$\rho:=R+1$.
Since~$s\to 0^+$, we may also assume that~$2s<\beta<1$. In particular, 
by Lemma~\ref{lem:independence of m},
$$
L_{m,s}u(0)=L_{1,s}u(0)=\frac{c_{1,s}}{2}u(0)\int_{\R^N\setminus B_{\rho}}\,\nu_s(dy) + \frac{c_{1,s}}{2}\int_{B_{\rho}}\big(u(0)-u(y)-u(-y)\big)\,\nu_s(dy).
$$

Now, using polar coordinates, we have that
\begin{eqnarray*}&&
\left|\;\int_{B_{\rho}}\big(u(0)-u(y)-u(-y)\big)\,\nu_s(dy)\right|=\left|\;
\int_0^{\rho}\int_{\bS^{N-1}}\frac{u(0)-u(r\theta)-u(-r\theta)}{r^{1+2s}}\,\sigma_s(d\theta)\, dr\right|\\
&&\qquad\leq 2\|u\|_{C^{\beta}(\R^N)}\int_0^{\rho}\int_{\bS^{N-1}} r^{\beta-1-2s}\,\sigma_s(d\theta)\, dr
=\frac{2\|u\|_{C^{\beta}(\R^N)} \rho^{\beta-2s}}{\beta-2s}.
\end{eqnarray*}
Thus,
\begin{equation}\label{duweigfuwtg43itg43ui}
\big|L_{m,s}u(0)- u(0)\big|\leq
|u(0)|\,\left|c_{1,s}\int_{\R^N\setminus B_{\rho}}\,\nu_s(dy)-1\right|
+ \frac{c_{1,s}\|u\|_{C^{\beta}(\R^N)} \rho^{\beta-2s}}{\beta-2s}.
\end{equation}

Now, by Lemmata~\ref{lem:limms} and~\ref{lem:constant estimate}, it follows that
\begin{equation}\label{duweigfuwtg43itg43ui2}\begin{split}
&\lim_{s\to0^+} \frac{c_{1,s}\|u\|_{C^{\beta}(\R^N)} \rho^{\beta-2s}}{\beta-2s}
=\lim_{s\to0^+} 
\frac{c_{1,s} M_{s,\sigma_s}(e_s)}{s} \,
\frac{s\|u\|_{C^{\beta}(\R^N)} \rho^{\beta-2s}}{(\beta-2s)M_{s,\sigma_s}(e_s)}\\
&\qquad=2 \lim_{s\to0^+} 
\frac{s\|u\|_{C^{\beta}(\R^N)} \rho^{\beta-2s}}{ \beta-2s}
=0.
\end{split}\end{equation}

Furthermore, since~$\sigma_s$ is a probability measure, we see that
\begin{equation}\label{gtru76680489hkthfewlj}
\int_{\R^N\setminus B_{\rho}}\,\nu_s(dy)= \int_{\rho}^{+\infty}\int_{\bS^{N-1}} 
r^{-1-2s}\,\sigma_s(d\theta) \,dr=\frac{1 }{2s\rho^{2s}},
\end{equation}
and therefore, exploiting again Lemmata~\ref{lem:limms} and~\ref{lem:constant estimate},
\begin{eqnarray*}
\lim_{s\to0^+}c_{1,s}\int_{\R^N\setminus B_{\rho}}\,\nu_s(dy)
=\lim_{s\to0^+} \, \frac{c_{1,s} }{2s\rho^{2s}}
=\lim_{s\to0^+} \frac{c_{1,s}M_{s,\sigma_s}(e_s)}s\, \frac{1}{2\rho^{2s}M_{s,\sigma_s}(e_s) }
=1.
\end{eqnarray*}
{F}rom this, \eqref{duweigfuwtg43itg43ui}
and~\eqref{duweigfuwtg43itg43ui2}, the desired claim follows.
\end{proof}

We also remark that when~$s$ is a positive integer the operator~$L_{m,s}$
defined in~\eqref{defop} boils down to a local operator of order~$2s$.

To show this, we first check the behavior of the constant~$c_{m,s}$ as~$s\to m^-$:

\begin{lemma}\label{lemma:aism}
For any~$n$, $m\in \N$ with~$n\geq m$ it holds that
$$ \lim_{s\to m^-} \frac{c_{n,s}}{m-s}M_{s,\sigma_s}(e_s)= \frac{4 P_m}{P_n},$$
where~$P_n$, $P_m$ are given by formula~\eqref{padef00}.
\end{lemma}

\begin{proof}
By \cref{lem:constant}, we see that
\begin{eqnarray*}
\frac{c_{n,s}}{m-s}= \frac{c_{n,s}P_n}{(m-s)P_n}= \frac{c_{m,s}P_m}{(m-s)P_n},
\end{eqnarray*}
hence it is enough to prove the desired limit in the case~$n=m\in \N$. 

Now, by \cref{rep of cms} we have
$$
2^{1-m}c_{m,s}^{-1}=M_{s,\sigma_s}(e_s)\int_0^{+\infty}\big(1- \cos\tau\big)^{m} \tau^{-1-2s}\,d\tau.
$$
As a consequence,
\begin{equation}\label{mnbvcxasdfgh234567}\begin{split}
&\lim_{s\to m^-}\frac{m-s}{c_{m,s} M_{s,\sigma_s}(e_s)}
=2^{m-1}\lim_{s\to m^-}(m-s)\int_0^{+\infty}\big(1-\cos\tau\big)^{m}\tau^{-1-2s}\,d\tau\\&\qquad\qquad
= 2^{m-1}\lim_{\eta\to 0^+}\eta\int_0^{+\infty}\big(1-\cos\tau\big)^{m}\tau^{-1-2m+2\eta}\,d\tau.
\end{split}\end{equation}

Notice that
\begin{eqnarray*}&&
\left|\int_1^{+\infty}\big(1-\cos\tau\big)^{m}\tau^{-1-2m+2\eta}\,d\tau\right|\le
2^m\int_1^{+\infty}\tau^{-1-2m+2\eta}\,d\tau=\frac{2^m}{2m-2\eta}
\end{eqnarray*}
and therefore
$$\lim_{\eta\to 0^+}\eta\int_1^{+\infty}\big(1-\cos\tau\big)^{m}\tau^{-1-2m+2\eta}\,d\tau=0.$$
Plugging this information into~\eqref{mnbvcxasdfgh234567}, we obtain that
\begin{eqnarray*}
\lim_{s\to m^-}\frac{m-s}{c_{m,s}M_{s,\sigma_s}(e_s)} &
=& 2^{m-1}\lim_{\eta\to 0^+}\eta\int_0^{1}\big(1-\cos\tau\big)^{m}\tau^{-1-2m+2\eta}\,d\tau\\&=&2^{m-1}\lim_{\eta\to 0^+}\eta\int_0^{1}\left(\frac{\tau^2}2+O(\tau^4)\right)^{m}\tau^{-1-2m+2\eta}\,d\tau
\\&=&\frac12\lim_{\eta\to 0^+}\eta\int_0^{1}\Big(1+O(\tau^2)\Big)^{m}\tau^{-1+2\eta}\,d\tau
\\&=&\frac12\lim_{\eta\to 0^+}\eta\int_0^{1}\Big(1+O(\tau^2)\Big)\tau^{-1+2\eta}\,d\tau
\\&=&\frac12\lim_{\eta\to 0^+}\eta\frac1{2\eta}\\&=&\frac14,
\end{eqnarray*}
which entails the desired result.
\end{proof}

With this we can now establish the following:

\begin{prop}\label{limit s to integers}
Let~$m\in\N\setminus\{0\}$, $U\subset\R^N$ and~$x\in U$.

Suppose that,
for all~$\alpha\in\N^N$ with~$|\alpha|=2m$,
\begin{equation}\label{limiszeroBIS2}\lim_{s\to m^-}
\frac{\displaystyle\int_{\bS^{N-1}}\theta^\alpha \,\sigma_s(d\theta)}{\displaystyle
M_{s,\sigma_s}(e_s)}= c_\alpha,
\end{equation}
for some~$ c_\alpha\in\R$. 

Then, for all~$u\in C^{2m}(U)\cap L^\infty(\R^N)$,
\begin{equation}\label{limintnnpiuuno}
\lim_{s\to m^-}L_{m,s}u(x)=
\sum_{|\alpha|=2m}\frac{(-1)^m (2m)!}{\alpha!}\,c_\alpha\,\partial^\alpha u(x).
\end{equation}

In particular, if~$\nu_s(dz)=\nu_n(dz):=|z|^{-N-2n}/\cH^{N-1}(\bS^{N-1})\, dz$, then 
\begin{equation}\label{limintnnpiuunoBIS}
\lim_{s\to m^-}L_{m,s}u(x)=(-\Delta)^mu(x).\end{equation}
\end{prop}

\begin{proof}
We exploit Lemma~2.4 in~\cite{AJS18a} to obtain that for every~$\varepsilon>0$
there exists~$\rho>0$, depending on~$U$, $m$, $u$ and~$\varepsilon$,
such that, for all~$y\in B_\rho$,
\begin{equation}\label{qifbewk093725t9753hje}
\delta_m u(x,y)= R(x,y)+\sum_{|\alpha|=2m}\frac{(-1)^m (2m)!}{\alpha!}\partial^\alpha u(x) y^\alpha,
\end{equation}
where~$|R(x,y)|\le C\varepsilon |y|^{2m}$, for some~$C>0$, depending only on~$N$ and~$m$.

Also, recalling~\eqref{defdelta} and~\eqref{gtru76680489hkthfewlj},
\begin{eqnarray*}&&\left|\;\int_{\R^N\setminus B_\rho}\delta_mu(x,y) \,\nu_s(dy)\right|
=\left|\;\int_{\R^N\setminus B_\rho}\sum_{k=-m}^m(-1)^k\binom{2m}{m-k}u(ky)\,\nu_s(dy)\right|
\\&&\qquad\le \|u\|_{L^\infty(\R^N)}\sum_{k=-m}^m(-1)^k\binom{2m}{m-k}\int_{\R^N\setminus B_\rho}\nu_s(dy)\\&&\qquad
=\|u\|_{L^\infty(\R^N)}\frac{1}{2s\rho^{2s}}
\sum_{k=-m}^m(-1)^k\binom{2m}{m-k}.
\end{eqnarray*}
Therefore, in light of Lemma~\ref{lemma:aism} and~\eqref{simple ellipticity},
\begin{equation*}\begin{split}&
\lim_{s\to m^-}\left|\frac{c_{m,s}}{2} \int_{\R^N\setminus B_\rho}\delta_mu(x,y) \,\nu_s(dy)\right|\\
\le\;& \frac{\|u\|_{L^\infty(\R^N)}}{4}\sum_{k=-m}^m(-1)^k\binom{2m}{m-k}
\lim_{s\to m^-} \frac{c_{m,s} }{s\rho^{2s}}\\
=\;&\frac{\|u\|_{L^\infty(\R^N)}}{4m\rho^{2m}}\sum_{k=-m}^m(-1)^k\binom{2m}{m-k}
\lim_{s\to m^-} \frac{c_{m,s}M_{s,\sigma_s}(e_s)}{m-s}
\frac{(m-s) }{M_{s,\sigma_s}(e_s)}\\ \le \;&
\frac{\|u\|_{L^\infty(\R^N)}}{\lambda\,m\rho^{2m}}\sum_{k=-m}^m(-1)^k\binom{2m}{m-k}
\lim_{s\to m^-} (m-s)\\=\;&0.
\end{split}
\end{equation*}

As a consequence of this, and recalling~\eqref{defop} and~\eqref{qifbewk093725t9753hje},
\begin{equation}\label{pwq12345qwert}\begin{split}
& \lim_{s\to m^-}\left|L_{m,s} u(x)-\sum_{|\alpha|=2m}\frac{(-1)^m (2m)!}{\alpha!}\,c_\alpha\,\partial^\alpha u(x)\right|\\=\;&\lim_{s\to m^-}\left|
\frac{c_{m,s}}{2} \int_{\R^N}\delta_mu(x,y) \,\nu_s(dy)-\sum_{|\alpha|=2m}\frac{(-1)^m (2m)!}{\alpha!}\,c_\alpha\,\partial^\alpha u(x)\right| \\
\le\;& \lim_{s\to m^-}
\left|\frac{c_{m,s}}{2} \int_{B_\rho}\delta_mu(x,y) \,\nu_s(dy)-\sum_{|\alpha|=2m}\frac{(-1)^m (2m)!}{\alpha!}\,c_\alpha\,\partial^\alpha u(x)\right|
\\ \le\;& \lim_{s\to m^-}\left[
\frac{c_{m,s}}{2} \int_{B_\rho}|R(x,y)|\,\nu_s(dy)\right.
\\&\qquad\qquad\left.+
\left|\sum_{|\alpha|=2m}\frac{(-1)^m (2m)!}{\alpha!}\partial^\alpha u(x) \left(\frac{c_{m,s}}{2} \int_{B_\rho}y^\alpha\,\nu_s(dy)-c_\alpha\right)\right|\;\;\right].
\end{split}\end{equation}

Now, we point out that
\begin{equation*}\begin{split}
& \int_{B_\rho}|R(x,y)|\,\nu_s(dy)\le
C\varepsilon\int_{B_\rho} |y|^{2m}\,\nu_s(dy) =
C\varepsilon\int_{0}^{\rho}
\int_{\bS^{N-1}}r^{-1-2s+2m}\,\sigma_s(d\theta) \,dr
=\frac{C\varepsilon\rho^{2m-2s}}{2(m-s)}.
\end{split}
\end{equation*}
This, together with Lemma~\ref{lemma:aism} and~\eqref{simple ellipticity}, entails that
\begin{eqnarray*} \lim_{s\to m^-}
\frac{c_{m,s}}{2} \int_{B_\rho}|R(x,y)|\,\nu_s(dy)\le
\lim_{s\to m^-}
\frac{c_{m,s}C\varepsilon\rho^{2m-2s} }{4(m-s)}=\lim_{s\to m^-}
\frac{C\varepsilon }{M_{s,\sigma_s}(e_s)}=\frac{C\varepsilon}{\lambda}.
\end{eqnarray*}

Hence, plugging this information into~\eqref{pwq12345qwert},
and making use of~\eqref{limiszeroBIS2}, we infer that
\begin{equation*}\begin{split}
& \lim_{s\to m^-}\left|L_{m,s} u(x)-\sum_{|\alpha|=2m}\frac{(-1)^m (2m)!}{\alpha!}\,c_\alpha\,\partial^\alpha u(x)\right|\\
&\le \frac{C\varepsilon}{\lambda}+\lim_{s\to m^-}
\left|\sum_{|\alpha|=2m}\frac{(-1)^m (2m)!}{\alpha!}\partial^\alpha u(x)\left(
\frac{c_{m,s}}{2} \int_{B_\rho} y^\alpha\,\nu_s(dy)- c_\alpha\right)\right|\\
&=\frac{C\varepsilon}{\lambda}+\lim_{s\to m^-}
\left|\sum_{|\alpha|=2m}\frac{(-1)^m (2m)!}{\alpha!}\partial^\alpha u(x) \left(
\frac{c_{m,s}\rho^{2m-2s}}{4(m-s)}\int_{\bS^{N-1}}\theta^\alpha \,\sigma_s(d\theta)- c_\alpha\right)\right|\\
&=\frac{C\varepsilon}{\lambda}+\lim_{s\to m^-}
\left|\sum_{|\alpha|=2m}\frac{(-1)^m (2m)!}{\alpha!}\partial^\alpha u(x)\cdot\right.\\
&\qquad\qquad \qquad \qquad  \cdot \left.\left(
\frac{c_{m,s}}{4(m-s)}M_{s,\sigma_s}(e_s)
\rho^{2m-2s}\ \frac{\displaystyle\int_{\bS^{N-1}}\theta^\alpha \,\sigma_s(d\theta)}{M_{s,\sigma_s}(e_s)}- c_\alpha\right)\right|\\
&=\frac{C\varepsilon}{\lambda}
.\end{split}\end{equation*}
Sending~$\varepsilon\to 0$, we obtain the desired result in~\eqref{limintnnpiuuno}.

In order to check~\eqref{limintnnpiuunoBIS}, we exploit formula~(3.4) in Lemma~3.3
of~\cite{AJS18a} to see that, for all~$\beta\in\N^N$ with~$|\beta|=m$,
\begin{eqnarray*}
&&\int_{\bS^{N-1}}\theta^{2\beta}\,d{\mathcal{H}}^{N-1}_\theta=2(m-s)\int_0^1
r^{-1-2s+2|\beta|}\,dr\int_{\bS^{N-1}}\theta^{2\beta}\,d{\mathcal{H}}^{N-1}_\theta\\&&\qquad\qquad=
\int_{B_1}\frac{y^{2\beta}}{|y|^{N+2s}}\,dy=\frac{(2\beta)!}{\beta!}\,
\frac{\pi^{\frac{N}2}}{\displaystyle 2^{2m-1}\Gamma\left(\frac{N}2+m\right)}.
\end{eqnarray*}
In particular, taking~$\beta:=m e_1=(m, 0,\dots,0)$, we have that
\begin{eqnarray*}
\int_{\bS^{N-1}}|\theta_1|^{2m}\,d{\mathcal{H}}^{N-1}_\theta=
\int_{\bS^{N-1}}\theta_1^{2m}\,d{\mathcal{H}}^{N-1}_\theta=
\int_{\bS^{N-1}}\theta^{2\beta}\,d{\mathcal{H}}^{N-1}_\theta
=\frac{(2m)!}{m!}\,
\frac{\pi^{\frac{N}2}}{\displaystyle 2^{2m-1}\Gamma\left(\frac{N}2+m\right)}
.\end{eqnarray*}
As a result, taking~$\beta:=\alpha/2$, we deduce from~\eqref{limiszeroBIS2} that
\begin{eqnarray*}c_\alpha&=&
\lim_{s\to m^-}
\frac{\displaystyle\int_{\bS^{N-1}}\theta^\alpha \,\sigma_s(d\theta)}{M_{s,\sigma_s}(e_1)}=
\lim_{s\to m^-}
\frac{\displaystyle\int_{\bS^{N-1}}\theta^{2\beta} \,d{\mathcal{H}}^{N-1}_\theta}{\displaystyle
\int_{\bS^{N-1}}|\theta_1|^{2s}\,d{\mathcal{H}}^{N-1}_\theta} =
\frac{(2\beta)!}{\beta!}\,\frac{m!}{(2m)!}.
\end{eqnarray*}
We point out that here above we have used the fact that we can choose~$e_s:=e_1$ (and actually any unit vector could be picked in this case).

Accordingly, formula~\eqref{limintnnpiuuno} and the Multinomial Theorem give in this case that
\begin{eqnarray*}&&\lim_{s\to m^-}L_{m,s}u(x)=
\sum_{|\beta|=m}\frac{(-1)^m (2m)!}{(2\beta)!}\,\frac{(2\beta)!}{\beta!}\,\frac{m!}{(2m)!}\,\partial^{2\beta} u(x)\\&&\qquad\quad=
\sum_{|\beta|=m}\frac{(-1)^m m! }{\beta!}\,\partial^{2\beta} u(x)
=(-\Delta)^m u(x),
\end{eqnarray*}
which completes the proof of~\eqref{limintnnpiuunoBIS}.
\end{proof}

We point out that the assumption in~\eqref{limiszeroBIS2} is not necessarly satisfied, as pointed out by the following example:

\begin{lemma}
For all~$k\in\N$, let
$$ I_k:=\left[ 1-\frac1{2^k}, 1-\frac1{2^{k+1}}\right).$$
For all~$s\in(1,2)$, define
$$ \sigma_s:=\begin{cases}\displaystyle
\delta_{e_1} \quad{\mbox{ if $s\in I_k$ with~$k$ even}},\\
\displaystyle
\frac{ {\mathcal{H}}^1}{2\pi} \quad{\mbox{ if $s\in I_k$ with~$k$ odd}}.
\end{cases}$$

Then, the limit 
\begin{equation*}\lim_{s\to 2^-}
\frac{\displaystyle\int_{\bS^{1}}\theta_1^2\theta_2^2 \,\sigma_s(d\theta)}{\displaystyle
M_{s,\sigma_s}(e_s)}
\end{equation*} does not exist.
\end{lemma}

\begin{proof} We point out that, for all~$s\in(1,2)$, we can choose~$e_s:=e_1$. 

Now, if~$s\in I_k$ with~$k$ even, we have that
\begin{eqnarray*}&&
M_{s,\sigma_s}(e_s)= \int_{\bS^{1}}|\theta_1|^{2s}\,\sigma_s(d\theta)=1\qquad
{\mbox{ and }} \qquad
\int_{\bS^{1}}\theta_1^2\theta_2^2 \,\sigma_s(d\theta)=0,
\end{eqnarray*} and therefore
\begin{equation}\label{bvnckcewur984376y8648y3o9876}
\lim_{{I_k\ni s\to 2^-}\atop{k{\tiny{\mbox{ even}}}}}
\frac{\displaystyle\int_{\bS^{1}}\theta_1^2\theta_2^2 \,\sigma_s(d\theta)}{\displaystyle
M_{s,\sigma_s}(e_s)}=0.\end{equation}

If instead~$s\in I_k$ with~$k$ odd,
\begin{eqnarray*}&&
M_{s,\sigma_s}(e_s)=\frac1{2\pi} \int_{\bS^{N-1}}|\theta_1|^{2s}\,d{\mathcal{H}}^{N-1}_\theta\\ {\mbox{and }}
&&\int_{\bS^{1}}\theta_1^2\theta_2^2 \,\sigma_s(d\theta)=\frac1{2\pi}\int_{\bS^{1}}\theta_1^2\theta_2^2 \,d{\mathcal{H}}^{N-1}_\theta.
\end{eqnarray*}
As a result,
\begin{eqnarray*}&&\lim_{{I_k\ni s\to 2^-}\atop{k{\tiny{\mbox{ odd}}}}}
\frac{\displaystyle\int_{\bS^{1}}\theta_1^2\theta_2^2 \,\sigma_s(d\theta)}{\displaystyle
M_{s,\sigma_s}(e_s)}=\lim_{{I_k\ni s\to 2^-}\atop{k{\tiny{\mbox{ odd}}}}}
\frac{\displaystyle\int_{\bS^{1}}\theta_1^2\theta_2^2 \,d{\mathcal{H}}^{N-1}_\theta}{\displaystyle\int_{\bS^{N-1}}|\theta_1|^{2s}\,d{\mathcal{H}}^{N-1}_\theta}
=\frac{\displaystyle\int_{\bS^{1}}\theta_1^2\theta_2^2 \,d{\mathcal{H}}^{N-1}_\theta}{\displaystyle\int_{\bS^{N-1}}|\theta_1|^{4}\,d{\mathcal{H}}^{N-1}_\theta}.
\end{eqnarray*}
We now use formula~(A.1) in~\cite[Appendix~A]{JACK} and we conclude that
\begin{eqnarray*}
\lim_{{I_k\ni s\to 2^-}\atop{k{\tiny{\mbox{ odd}}}}}
\frac{\displaystyle\int_{\bS^{1}}\theta_1^2\theta_2^2 \,\sigma_s(d\theta)}{\displaystyle
M_{s,\sigma_s}(e_s)}=\frac13.
\end{eqnarray*}
This and~\eqref{bvnckcewur984376y8648y3o9876} entail the desired claim.
\end{proof}

We close this section with the following remark on the constant~$c_{m,s}$.

\begin{cor}\label{cor:bounds on cms}
Let~$m>0$ and~$\{\sigma_s\}_{s\in(0,m)}$ be a family of probability measures. 
Then, there exists~$C_m>0$, depending on~$m$ but
independent of~$s$, such that
$$
2^{2-m}\leq c_{m,s} \left(\frac{1}{s}+\frac{1}{m-s}\right) M_{s,\sigma_s}(e_s)\leq C\quad\text{for all~$s\in(0,m)$.}
$$
\end{cor}
\begin{proof}
From \cref{rep of cms} we know that
$$
c_{m,s}^{-1}=2^{m-1}M_{s,\sigma_s}(e_s)\int_0^{+\infty} \big(1-\cos\tau\big)^m \tau^{-1-2s}\, d\tau.
$$
Now we note that
$$
\int_0^{+\infty} \big(1-\cos\tau\big)^m \tau^{-1-2s}\, d\tau\leq \int_0^{+\infty} \min\{1,\tau^{2m}\}\tau^{-1-2s}\, d\tau\leq \frac{1}{2(m-s)}+\frac{1}{2s}= \frac{m}{2s(m-s)}.
$$
Hence,
$$
c_{m,s} \left(\frac{1}{s}+\frac{1}{m-s}\right) M_{s,\sigma_s}(e_s) \geq 2^{2-m}\frac{s(m-s)}{m} \left(\frac{1}{s}+\frac{1}{m-s}\right)=2^{2-m}
$$
and this entails the lower bound.

Furthermore, Lemmata~\ref{lem:limms} and~\ref{lem:constant estimate} give that
$$
\lim_{s\to0^+}c_{m,s} \left(\frac{1}{s}+\frac{1}{m-s}\right)M_{s,\sigma_s}(e_s)=\frac{2}{-P_m},
$$
while \cref{lemma:aism} implies that
$$
\lim_{s\to m^-}c_{m,s} \left(\frac{1}{s}+\frac{1}{m-s}\right) M_{s,\sigma_s}(e_s)=4.
$$
Hence, there exists~$\epsilon>0$, independent of~$s$, such that, for all~$s\in(0,\epsilon]\cap [m-\epsilon,m)$,
\begin{equation}\label{0987654bchekw1qsax}
c_{m,s} \left(\frac{1}{s}+\frac{1}{m-s}\right) M_{s,\sigma_s}(e_s)\leq \max\left\{5, \frac{3}{-P_m}\right\}.
\end{equation}

Moreover, since for any~$s\in(0,m)$ we have
$$
\int_0^{+\infty} \big(1-\cos\tau\big)^m \tau^{-1-2s}\, d\tau\geq \frac{2-\sqrt{2}}2
\int_{\frac{\pi}4}^{\frac{\pi}{2}}\tau^{-1-2s}\, d\tau= \frac{2^{2s}(2-\sqrt{2})}{4s \pi^{2s}}\big(2^{2s}-1\big),
$$
it follows that, for any~$s\in(\epsilon,m-\epsilon)$,
\begin{eqnarray*}&&
c_{m,s} \left(\frac{1}{s}+\frac{1}{m-s}\right) M_{s,\sigma_s}(e_s)=
2^{1-m}\left(\frac{1}{s}+\frac{1}{m-s}\right)\left(\,\int_0^{+\infty} \big(1-\cos\tau\big)^m \tau^{-1-2s}\, d\tau\right)^{-1}\\
&&\qquad\le 
2^{1-m}\left(\frac{1}{s}+\frac{1}{m-s}\right) \frac{4s \pi^{2s}}{2^{2s}(2-\sqrt{2})(2^{2s}-1)}
\le \frac{2^{1-m} m}{\epsilon^2}\frac{4 m \pi^{2m}}{ (2-\sqrt{2})(2^{2\epsilon}-1)}.
\end{eqnarray*}
This and~\eqref{0987654bchekw1qsax} entail the upper bound.
\end{proof}

\section{Variational definitions and properties}\label{sec:var23}

Here, we introduce the variational framework that we work in.
For this, for any~$s\in(0,2m)$, with~$m\in \N$, and~$u$, $v\in C^{\infty}_c(\R^N)$,
we define the bilinear form
\begin{equation}\label{098765099werfgthkmnbgfvd876}
\cE_{2m,s}(u,v):=\frac{c_{2m,s}}{2}\iint_{\R^{2N}}\delta_mu(x,y)\delta_mv(x,y) \,dx\,\nu_s(dy).
\end{equation}
Then, we have:

\begin{prop}\label{prop:bilinear}
For~$s>0$ and~$n$, $m\in \N$ such that~$s<n\leq 2m$, it holds that,
for all~$u, v\in C^{\infty}_c(\R^N)$,
$$
\int_{\R^N}L_{n,s}u(x)v(x)\, dx=\cE_{2m,s}(u,v).
$$
\end{prop}

\begin{proof}
In light of Lemma~\ref{lem:independence of m},
for all~$u, v\in C^{\infty}_c(\R^N)$, we have that
\begin{eqnarray*}&&
\int_{\R^N}L_{n,s}u(x)v(x)\, dx=
\int_{\R^N}L_{2m,s}u(x)v(x)\, dx
=\frac{c_{2m,s}}{2}\iint_{\R^{2N}}\delta_{2m}u(x,y) v(x) \,dx\,\nu_s(dy).
\end{eqnarray*}
Therefore, recalling the definition of~$\delta_{2m}$ in~\eqref{defdelta},
\begin{equation}\label{ewuigiuty45ou745u56}
\int_{\R^N}L_{n,s}u(x)v(x)\, dx=
\frac{c_{2m,s}}{2}\sum_{k=-2m}^{2m}(-1)^k
\binom{4m}{2m-k}
\iint_{\R^{2N}}u(x+ky) v(x) \,dx\,\nu_s(dy).
\end{equation}

On the other hand, changing variable~$z:=x+jy$ and then changing index~$h:=k-j$,
\begin{equation}\label{ewuigiuty45ou745u562}\begin{split}
&\cE_{2m,s}(u,v)=
\frac{c_{2m,s}}{2}\iint_{\R^{2N}}\delta_mu(x,y)\delta_mv(x,y) \,dx\,\nu_s(dy)
\\&\qquad =
\frac{c_{2m,s}}{2}
\sum_{{-m\le k\le m}\atop{-m\le j\le m}} (-1)^{k+j}\binom{2m}{m-k}
\binom{2m}{m-j}
\iint_{\R^{2N}}u(x+ky)v(x+jy) \,dx\,\nu_s(dy)\\
&\qquad=
\frac{c_{2m,s}}{2}
\sum_{{-m\le k\le m}\atop{-m\le j\le m}} (-1)^{k+j}\binom{2m}{m-k}
\binom{2m}{m-j}
\iint_{\R^{2N}}u\big(z+(k-j)y\big)v(z) \,dz\,\nu_s(dy)\\
&\qquad=
\frac{c_{2m,s}}{2}
\sum_{{-m\le k\le m}\atop{k-m\le h\le k+m}} (-1)^{h}\binom{2m}{m-k}
\binom{2m}{m-k+h}
\iint_{\R^{2N}}u\big(z+hy\big)v(z) \,dz\,\nu_s(dy).
\end{split}\end{equation}

We now claim that
\begin{equation}\label{ewuigiuty45ou745u563}
\sum_{{-m\le k\le m}\atop{k-m\le h\le k+m}}\binom{2m}{m-k}
\binom{2m}{m-k+h} =\sum_{h=-2m}^{2m}\binom{4m}{2m-h}.
\end{equation}
Indeed, for any~$a\in\N$ and~$b\in\Z$, we set
\begin{equation}\label{conventp607u}
\binom{a}{b}:=\begin{cases}\displaystyle
\frac{a!}{b!(a-b)!} &{\mbox{ if }}
0\le b \le a,\\
0 &{\mbox{ otherwise.}}
\end{cases}\end{equation}
With this notation, we write the Chu-Vandermonde Identity,
for~$a$, $b\in\N$ and~$r\in\Z$, as
$$  \binom{a+b}{r} =\sum_{\ell\in\Z}\binom{a}{\ell}\binom{b}{r-\ell}.
$$
Hence, for all~$h\in\{-2m,\dots,2m\}$, taking~$a:=2m$, $b:=2m$ and~$r:=2m-h$,
we obtain that
$$  \binom{4m}{2m-h} =\sum_{\ell\in\Z}\binom{2m}{\ell}\binom{2m}{2m-h-\ell}.
$$
Setting~$k:=m-\ell$, we get that
$$  \binom{4m}{2m-h} =\sum_{k\in\Z}\binom{2m}{m-k}
\binom{2m}{m+k-h}=\sum_{k\in\Z}\binom{2m}{m-k}
\binom{2m}{m-k+h}.
$$
Accordingly, recalling also the convention in~\eqref{conventp607u},
\begin{eqnarray*}
\sum_{h=-2m}^{2m}\binom{4m}{2m-h}
&=&\sum_{h=-2m}^{2m}\sum_{k\in\Z}\binom{2m}{m-k}
\binom{2m}{m-k+h}\\
&=&\sum_{h=-2m}^{2m}\;\sum_{k=\max\{-m,h-m\}}^{\min\{m,m+h\}}\binom{2m}{m-k}\binom{2m}{m-k+h}\\
&=&\sum_{k=-m}^{m}\;\sum_{h=k-m}^{k+m}\binom{2m}{m-k}\binom{2m}{m-k+h},
\end{eqnarray*}
which establishes~\eqref{ewuigiuty45ou745u563}.

{F}rom~\eqref{ewuigiuty45ou745u56},
\eqref{ewuigiuty45ou745u562} and~\eqref{ewuigiuty45ou745u563},
we obtain the desired result.
\end{proof}

Furthermore, let 
\begin{equation*}
\cX^s:=\cX^s(\R^N):=\Big\{u\in L^2(\R^N)\;:\; \cE_{2n,s}(u,u)<+\infty\text{ for some~$n\in \N$ with~$2n>s$}\Big\}.
\end{equation*}
The space~$\cX^s$ is equipped with the scalar product
$$
\langle u,v\rangle_{\cX^s}:=\int_{\R^N}u(x)v(x)\, dx+\cE_{2n,s}(u,v).
$$
The norm associated with this scalar product is given by
\begin{equation}\label{normdef098765}
\|u\|_{\cX^s}:=
\sqrt{\langle u,u\rangle_{\cX^s}}=
\sqrt{\;\int_{\R^N}|u(x)|^2\, dx+\cE_{2n,s}(u,u)}.
\end{equation}

When~$s=0$, we make the ansatz that~$\cX^s=L^2(\R^N)$ (notice that this is compatible with the limit procedure
in Propositions~\ref{limit s to 0} and~\ref{prop:bilinear}).

We emphasize that the definition of~$\cX^s$ does not depend on the choice of~$n$, thanks to \cref{lem:independence of m} and \cref{prop:bilinear}. 

In this setting, we have:

\begin{prop}\label{hilbert-space}
$\cX^s$ is a Hilbert space
with the scalar product~$\langle \cdot, \cdot\rangle_{\cX^s}$.
\end{prop}

\begin{proof}
Clearly, $\langle \cdot, \cdot\rangle_{\cX^s}$ is a scalar product. To finish the proof, note that for~$s\in(0,1)$ this is stated in~\cite[Lemma~2.3]{FKV15} or~\cite[Proposition~2.2]{JW16}. The proof presented there easily extends to the general case here. 

Indeed, let~$(u_n)_n$ be a Cauchy sequence, then since~$\cX^s\subset L^2(\R^N)$, there exists~$u\in L^2(\R^N)$ such that~$u_n\to u$ in~$L^2(\R^N)$. Passing to a subsequence, we have that~$u_{n_k}\to u$ a.e. in~$\R^N$ as~$k\to+\infty$. Therefore, Fatou's Lemma implies that
$$
\cE_{2n,s}(u,u)\leq \liminf_{k\to+\infty}\cE_{2n,s}(u_{n_k},u_{n_k})\leq \sup_{n\in \N} \cE_{2n,s}(u_n,u_n)<+\infty. 
$$
This says that~$u\in \cX^s$.

Then, for~$k\in \N$, by applying again Fatou's Lemma, we find that
$$
\cE_{2n,s}(u_{n_k}-u,u_{n_k}-u)\leq \liminf_{j\to+\infty}\cE_{2n,s}(u_{n_k}-u_{n_j},u_{n_k}-u_{n_j})\leq \sup_{j\geq k} \cE_{2n,s}(u_{n_k}-u_{n_j},u_{n_k}-u_{n_j}). 
$$
Since~$(u_n)_n$ is a Cauchy sequence with respect to~$\langle \cdot,\cdot \rangle _{\cX^s}$, it follows that~$u_{n_k}\to u$ in~$\cX^s$ as~$k\to+\infty$. The completeness of~$\cX^s$ follows.
\end{proof}

\begin{prop}\label{unif-convex}
$\cX^s$ is a uniformly convex space.
\end{prop}

\begin{proof}
We need to prove that for every~$\eps \in(0, 2]$ there exists~$\delta>0$ such that if~$u$, $v\in \cX^s$ are such that~$\|u\|_{\cX^s}=\|v\|_{\cX^s}
=1$ and~$\|u-v\|_{\cX^s} >\eps$, then~$\|u+v\|_{\cX^s}\le 2-\delta$.

To this aim, we recall that, for all~$a$, $b\in\R$,
$$ |a+b|^2+|a-b|^2\le 2\big(|a|^2+|b|^2\big).$$ 
Thus, we use this inequality with~$a:=u(x)$ and~$b:=v(x)$, and also
with~$a:=\delta_nu(x,y)$ and~$b:=\delta_nv(x,y)$,
and we see that
\begin{eqnarray*}
\|u+v\|_{\cX^s}^2&=&\int_{\R^N}|u(x)+v(x)|^2\, dx+
\frac{c_{2n,s}}{2}\iint_{\R^{2N}}\Big(\delta_nu(x,y)+\delta_nv(x,y)\Big)^2 \,dx\,\nu_s(dy)\\
&\le& 
2\int_{\R^N}\Big(|u(x)|^2+|v(x)|^2\Big)\, dx-
\int_{\R^N}\Big(|u(x)-v(x)|^2\Big)\, dx\\
&&\qquad
+{c_{2n,s}}\iint_{\R^{2N}}\Big[\Big(\delta_nu(x,y)\Big)^2+\Big(\delta_nv(x,y)\Big)^2\Big] \,dx\,\nu_s(dy)
\\&&\qquad-
\frac{c_{2n,s}}{2}\iint_{\R^{2N}}\Big(\delta_nu(x,y)-\delta_nv(x,y)\Big)^2 \,dx\,\nu_s(dy)\\
&=&2 \Big( \|u\|_{\cX^s}^2+\|v\|_{\cX^s}^2\Big)- \|u-v\|_{\cX^s}^2
\\&=& 4-\eps^2
\\&=& \big(2-\delta\big)^2,
\end{eqnarray*}
with
$$ \delta:=\delta(\eps):=2-\big(4-\eps^2\big)^{\frac12}.
$$
Notice that~$\delta$ is strictly increasing in~$\eps\in(0,2]$, and therefore~$\delta>\delta(0)=0$. This completes the proof.
\end{proof}

\begin{prop}\label{density in rn}
The space~$C^{\infty}_c(\R^N)$ is dense in~$\cX^s$ with respect to the norm~$\|\cdot\|_{\cX^s}$.
\end{prop}

\begin{proof}
We observe that~$C^{\infty}_c(\R^N)\subset \cX^s$ by \cref{lem:evaluation} and Proposition~\ref{prop:bilinear}.

We will first show that
\begin{equation}\label{firstcinfdensxp}
{\mbox{$C^{\infty}(\R^N)\cap \cX^s$ is dense in~$\cX^s$
with respect to the norm~$\|\cdot\|_{\cX^s}$.}}\end{equation}
For this, let~$(\rho_{\eps})_{\eps>0}$ be a Dirac sequence, that is~$\rho_{\eps}(x)=\eps^{-N}\rho(x/\eps)$ for a smooth nonnegative function~$\rho:\R^N\to\R$ with~$\|\rho\|_{L^1(\R^N)}=1$ and~$\supp\,\rho\subset B_1(0)$. Then, following~\cite[Proposition~4.1]{JW19}, given~$u\in\cX^s$,
we set~$u_{\eps}:=u\ast \rho_{\eps}$ and we have that
\begin{eqnarray*}
&&\cE_{2n,s}(u_{\eps},u_{\eps})\\
&=&
\frac{c_{2n,s}}{2}\iint_{\R^{2N}}\Big(\delta_nu_\eps(x,y)\Big)^2 \,dx\,\nu_s(dy)\\
&=&\frac{c_{2n,s}}{2}
\iiiint_{\R^{4N}}\rho_{\eps}(z)\rho_{\eps}(z')\delta_nu(x-z,y)\delta_nu(x-z',y)\,dz\,dz'\,dx\,\nu_s(dy)\\
&=&\frac{c_{2n,s}}{2}
\iiiint_{\R^{4N}}\rho_{\eps}(z)\rho_{\eps}(z')\delta_nu(x+z'-z,y)\delta_nu(x,y)\,dz\,dz'\,dx\,\nu_s(dy)\\
&=& \frac{c_{2n,s}}{2}
\iiiint_{\R^{4N}}\rho_{\eps}(z'-\widetilde{z})\rho_{\eps}(z')\delta_nu(x+\widetilde{z},y)\delta_nu(x,y)\, dx\,d\widetilde{z}\,dz'\,\nu_s(dy)\\
&=& \frac{c_{2n,s}}{2}
\iint_{\R^{2N}} \delta_nu(x,y)\left(\;\int_{\R^N}
(\rho_{\eps}\ast\rho_{\eps})(\widetilde z)\delta_nu(x+\widetilde z,y)\,d\widetilde z\right)\,dx\,\nu_s(dy).
\end{eqnarray*}
Therefore, by H\"older Inequality and Jensen Inequality,
\begin{equation}\label{oiuytrew098765432mnbvcxzjhgfds}\begin{split}
&\cE_{2n,s}(u_{\eps},u_{\eps})\\
\leq\;&\frac{c_{2n,s}}{2}
\left(\;\iint_{\R^{2N}} \Big(\delta_nu(x,y)\Big)^2\,dx\,\nu_s(dy)\right)^{\frac12}\\&\qquad \times
\left(\; \iint_{\R^{2N}}\left(\;\int_{\R^N}
(\rho_{\eps}\ast\rho_{\eps})(\widetilde z)\delta_nu(x+\widetilde z,y)\,d\widetilde z\right)^2\,dx\,\nu_s(dy)\right)^{\frac12} \\ \le\;&
\frac{c_{2n,s}}{2}
\left(\;\iint_{\R^{2N}} \Big(\delta_nu(x,y)\Big)^2\,dx\,\nu_s(dy)\right)^{\frac12}\\&\qquad\times
\left(\; \int_{\R^{N}} (\rho_{\eps}\ast\rho_{\eps})(\widetilde z)
\left(\;\iint_{\R^{2N}}
\Big(\delta_nu(x+\widetilde z,y)\Big)^2\,dx\,\nu_s(dy)
\right)\,d\widetilde z\right)^{\frac12} \\=\;&
\frac{c_{2n,s}}{2}
\iint_{\R^{2N}} \Big(\delta_nu(x,y)\Big)^2\,dx\,\nu_s(dy)
\left(\; \int_{\R^{N}} (\rho_{\eps}\ast\rho_{\eps})(\widetilde z)\,d\widetilde z\right)^{\frac12} \\=\;&
\cE_{2n,s}(u,u)
\left(\; \int_{\R^{N}} (\rho_{\eps}\ast\rho_{\eps})(\widetilde z)\,d\widetilde z\right)^{\frac12}.
\end{split}\end{equation}
Now we point out that
\begin{eqnarray*}
&&\int_{\R^N}(\rho_{\eps}\ast\rho_{\eps})(z)\, dz
= \iint_{\R^{2N}}\rho_{\eps}(w) \rho_{\eps}(w-z)\, dz\,dw
=\|\rho_{\eps}\|_{L^1(\R^N)}^2=1,
\end{eqnarray*}
and thus, plugging this information into~\eqref{oiuytrew098765432mnbvcxzjhgfds}, we conclude that
\begin{equation}\label{fruey7658655p45436548}
\cE_{2n,s}(u_{\eps},u_{\eps})\le\cE_{2n,s}(u,u).\end{equation}
This implies that~$u_\eps\in C^{\infty}(\R^N)\cap \cX^s$.

Hence, to complete the proof of~\eqref{firstcinfdensxp}, it remains
to check that
\begin{equation}\label{fhiwo4ty34y345py35y34y3ghjghjkegw}
\lim_{\eps\to0} \|u-u_\eps\|_{\cX^s}=0.
\end{equation}
To this end, we recall that
\begin{equation}\label{lduewte76790}
{\mbox{$u_\eps$ converges to~$u$ in~$L^2(\R^N)$
and a.e. in~$\R^N$ as~$\epsilon\to0$}}\end{equation}
(see e.g.~\cite[Section~4.2.1]{MR3409135}).

Accordingly, we can use Fatou's Lemma and~\eqref{fruey7658655p45436548}
to see that
\begin{eqnarray*}
\cE_{2n,s}(u,u)&=&\frac{c_{2n,s}}{2}\iint_{\R^{2N}}\Big(\delta_nu(x,y)\Big)^2 \,dx\,\nu_s(dy)\\
&\le&\liminf_{\eps\to0}\frac{c_{2n,s}}{2}\iint_{\R^{2N}}\Big(\delta_nu_\eps(x,y)\Big)^2 \,dx\,\nu_s(dy)
\\&\le&\limsup_{\eps\to0}\frac{c_{2n,s}}{2}\iint_{\R^{2N}}\Big(\delta_nu_\eps(x,y)\Big)^2 \,dx\,\nu_s(dy)\\&
\le& \frac{c_{2n,s}}{2}\iint_{\R^{2N}}\Big(\delta_nu(x,y)\Big)^2 \,dx\,\nu_s(dy)
\\&=&\cE_{2n,s}(u,u).
\end{eqnarray*}
This says that
$$
\lim_{\eps\to0}\cE_{2n,s}(u_\eps,u_\eps)=\cE_{2n,s}(u,u),$$
and therefore, recalling also~\eqref{lduewte76790},
\begin{equation}\label{fhiwo4ty34y345py35y34y3ghjghjkegw1}
\lim_{\eps\to0}\|u_\eps\|_{\cX^s}=\|u\|_{\cX^s}.
\end{equation}

We now show that, for all~$v\in \cX^s$,
\begin{equation}\label{fhiwo4ty34y345py35y34y3ghjghjkegw2}
\lim_{\eps\to0} \cE_{2n,s}(u-u_{\eps},v)= 0.
\end{equation}
For this, we take~$v\in \cX^s$ and we observe that, by H\"older Inequality
and~\eqref{fruey7658655p45436548},
\begin{eqnarray*}
&&\frac{c_{2n,s}}{2}\iint_{\R^{2N}}\delta_n(u-u_\eps)(x,y)\delta_nv(x,y) \,dx\,\nu_s(dy)\\&=&\frac{c_{2n,s}}{2}
\iint_{\R^{2N}}\delta_nu(x,y)\delta_nv(x,y) \,dx\,\nu_s(dy)-\frac{c_{2n,s}}{2}
\iint_{\R^{2N}}\delta_n u_\eps(x,y)\delta_nv(x,y) \,dx\,\nu_s(dy)
\\&\le& \Big(\cE_{2n,s}(u,u)\Big)^{\frac12}
\Big(\cE_{2n,s}(v,v)\Big)^{\frac12}
+\Big(\cE_{2n,s}(u_\eps,u_\eps)\Big)^{\frac12}
\Big(\cE_{2n,s}(v,v)\Big)^{\frac12}\\
&\le& 2\Big(\cE_{2n,s}(u,u)\Big)^{\frac12}
\Big(\cE_{2n,s}(v,v)\Big)^{\frac12}.
\end{eqnarray*}
This and the Dominated Convergence Theorem give~\eqref{fhiwo4ty34y345py35y34y3ghjghjkegw2}.

Thus, \eqref{fhiwo4ty34y345py35y34y3ghjghjkegw}
follows from~\eqref{lduewte76790}, \eqref{fhiwo4ty34y345py35y34y3ghjghjkegw1},
\eqref{fhiwo4ty34y345py35y34y3ghjghjkegw2}
and Proposition~\ref{unif-convex}.
This, in turn, completes the proof of~\eqref{firstcinfdensxp}.

Having established~\eqref{firstcinfdensxp},
the desired claim of the proposition now follows in a standard way by considering for~$t>0$ a smooth radial cut-off~$\phi_t:\R^N\to\R$ satisfying~$\phi_t=1$ in~$B_t$ and~$\phi_t=0$ in~$\R^N\setminus B_{t+1}$. Then~$\rho_{\eps}\ast(\phi_{1/\eps}u)\in C^{\infty}_c(\R^N)$ is a sequence which approximates~$u$ in~$\cX^s$. 
\end{proof}

Due to the Hilbert space structure, $\cX^s$ has a representation via Fourier Transform. This representation will be helpful when comparing spaces~$\cX^s$ with respect to different measures~$\sigma_s$ or varying expontents~$s$.

\begin{prop}\label{fourier rep}
	Let~$n\in \N$ and~$s\in(0,n)$. Then, for any~$u\in \cX^s$,
	$$
	\cE_{2n,s}(u,u)=(M_{s,\sigma_s}(e_s))^{-1}\int_{\R^N} M_{s,\sigma_s}\left(\frac{\xi}{|\xi|}\right)|\xi|^{2s}|\widehat{u}(\xi)|^2\, d\xi
	$$
	where~$M_{s,\sigma_s}$ is as in~\eqref{def ms}.
\end{prop}

\begin{proof}
The idea is based on the computations in~\cite[Proposition~6.1]{DK20} ---since the proof there is based on the Fourier Transform and the~$m$-th order difference operator translates in a similar way (see also~\cite[Theorem~1.7]{AJS18a}). We now provide the details of the proof.

We recall that the function~$f(t):=\exp(it)$ satisfies~$\delta_nf(0,t)=2^n(1-\cos t)^n$
(see formula~\eqref{agaapa0}). Thus,
by the Plancherel Formula in the~$x$-variable and Fubini's Theorem, we have that
\begin{equation}\label{step1-equal spaces}\begin{split}
		&\frac{2}{c_{2n,s}}\cE_{2n,s}(u,u)=\iint_{\R^{2N}}\Big(\delta_{n} {u}(x,y)\Big)^2\, dx\, \nu_s(dy) \\
		&\quad=\iint_{\R^{2N}}\Big(\widehat{\delta_{n}{u}}(\xi,y)\Big)^2\, d\xi\, \nu_s(dy) \\
		&\quad=\iint_{\R^{2N}}\Big(\delta_{n}f(0,y\cdot\xi)\Big)^2|\widehat{u}(\xi)|^2\, d\xi\, \nu_s(dy)
		\\
		&\quad=\iint_{\R^{2N}}4^{n}\Big(1-\cos(y\cdot \xi)\Big)^{2n} |\widehat{u}(\xi)|^2  \,\nu_s(dy)\, d\xi\\
		&\quad=4^{n}\int_{\R^N}\left(\; \int_0^{+\infty} r^{-1-2s}\int_{\bS^{N-1}}\Big(1-\cos(r \theta\cdot\xi)\Big)^{2n}\,\sigma_s(d\theta)\,dr\right)|\widehat{u}(\xi)|^2\, d\xi\\
		&\quad=4^{n}\int_{\R^N} \left(\;
		\int_0^{+\infty} t^{-1-2s}\int_{\bS^{N-1}}\left(\;1-\cos\left(\,\frac{t\theta\cdot\xi}{|\xi|}\right)\right)^{2n}\,\sigma_s(d\theta)\,dt\right)|\xi|^{2s}|\widehat{u}(\xi)|^2\, d\xi\\
		&\quad=4^{n}\int_{\R^N} \left(\;
		\int_{\R^N} \left(\;1-\cos\left(\,y\cdot \frac{\xi}{|\xi|}\right)\right)^{2n}\,\nu_s(dy)\right) |\xi|^{2s}|\widehat{u}(\xi)|^2\, d\xi.
\end{split}\end{equation}

Now, we can rewrite the integral in~$\nu_s$ as follows. Given~$h\in \bS^{N-1}$, we have by the substitution~$\tau := r|\theta  \cdot h|$ that
\begin{eqnarray*}
	&& \int_{\R^N} \Big(1-\cos(y\cdot h)\Big)^{2n}\,\nu_s(dy)
	= \int_0^{+\infty} r^{-1-2s}\int_{\bS^{N-1}}\Big(1- \cos( r\theta \cdot h)\Big)^{2n}\sigma_s(d\theta)\,dr\\
	&&\qquad= M_{s,\sigma_s}(h)\int_0^{+\infty}\big(1- \cos\tau\big)^{2n} \tau^{-1-2s}\,d\tau.
\end{eqnarray*}
Also, using \cref{rep of cms} we find that
$$
\frac{4^{n}c_{2n,s}}{2}\int_0^{+\infty} \big(1- \cos\tau\big)^{2n}\tau^{-1-2s}\,d\tau=(M_{s,\sigma_s}(e_s))^{-1},
$$
which entails that
$$\int_{\R^N} \Big(1-\cos(y\cdot h)\Big)^{2n}\,\nu_s(dy)=\frac{2}{4^{n}c_{2n,s}}M_{s,\sigma_s}(h)(M_{s,\sigma_s}(e_s))^{-1}.
$$

Plugging this information into~\eqref{step1-equal spaces}, we obtain the desired result.
\end{proof}

\subsection{Some useful embeddings}\label{sec:embeddings}

As discussed in~\cite[Theorem~1.7]{AJS18a}, with the choice~$\nu_s(dz):=|z|^{-N-2s}\, dz$ (with a possible renormalization by a factor~$\cH^{N-1}(\bS^{N-1})$)
it follows that the space~$\cX^s$ coincides with the Sobolev space~$H^s(\R^N)=W^{s,2}(\R^N)$ for~$s>0$ and we have
$$
\cE_s(u,u)=\cE_{2s_1,s}(u,u),
$$
where~$\cE_s$ denotes the bilinear form of the classical fractional Sobolev space~$H^s(\R^N)$ given by
\begin{equation}\label{defesse097}
	\cE_s(u,v):=\int_{\R^N} (-\Delta)^{s/2}u(x) (-\Delta)^{s/2}v(x)\, dx.
	\end{equation}
Our aim is now to understand the relation between the spaces~$\cX^s$ and~$H^s(\R^N)$ in our general setting.

\begin{prop}\label{relation to hs}
Let~$s>0$ and~$n\in \N$ with~$n>s$. Then, for all~$u\in H^s(\R^N)$,
	$$
	\cE_{2n,s}(u,u)\leq (M_{s,\sigma_s}(e_s))^{-1} \cE_s(u,u).
	$$
In particular, $H^s(\R^N)\subset \cX^s$.
	
	If, in addition, the strong ellipticity assumption~\eqref{ellipticity} is satisfied for some~$\lambda_0>0$, then
	\begin{equation}\label{0qwdoj320949b65m7uRM2} \cE_s(u,u)
\le  \frac{M_{s,\sigma_s}(e_s)}{\lambda_0}  \,\cE_{2n,s}(u,u),\end{equation}
	and~$H^s(\R^N)=\cX^s$.
\end{prop}

\begin{proof}
From \cref{fourier rep} we have that
$$
\cE_{2n,s}(u,u)=(M_{s,\sigma_s}(e_s))^{-1}\int_{\R^N} M_{s,\sigma_s}\left(\frac{\xi}{|\xi|}\right)|\xi|^{2s}|\widehat{u}(\xi)|^2\, d\xi,
$$
while from~\cite[Theorems~1.8 and~1.9]{AJS18a} we know that
$$
\cE_{s}(u,u)=\int_{\R^N}  |\xi|^{2s}|\widehat{u}(\xi)|^2\, d\xi.
$$
This entails that
$$
(M_{s,\sigma_s}(e_s))^{-1}\inf_{h\in \bS^{N-1}} M_{s,\sigma_s}(h)  \cE_s(u,u)\leq \cE_{2n,s}(u,u)\leq (M_{s,\sigma_s}(e_s))^{-1}\sup_{h\in \bS^{N-1}}M_{s,\sigma_s}(h)\cE_s(u,u)
$$
from which the claim follows using~\eqref{usoduevoltefprse769043-0876} and~\eqref{ellipticity}.
\end{proof}

In the case where~$\cX^s=H^s(\R^N)$, we have in particular the classical embedding theorems,
see e.g.~\cite[Theorem~1.4.4.1]{zbMATH05960425} or~\cite[Theorem 2.8.1 and eq. (19)]{T78}. Hence, if~\eqref{ellipticity} is satisfied, \cref{relation to hs} gives immediately the following.

\begin{cor}\label{cor:sobolev embedding}
Let~$\nu_s$ satisfy~\eqref{ellipticity}. Then, the following embeddings hold:
\begin{enumerate}
\item If~$s<\frac{N}{2}$, then~$\cX^s\subset L^{2_{s}^{\ast}}(\R^N)$ with~$2_s^{\ast}:=\frac{2N}{N-2s}$.
\item If~$s\geq \frac{N}{2}$ and~$s-\frac{N}{2}\notin\N_0$, then~$\cX^s\subset C^{s-\frac{N}{2}}(\R^N)$.
\end{enumerate}
\end{cor}

Now we check the scaling properties of the energy functional~$\cE_{2s_1,s}$.

\begin{lemma}\label{scaling}
Let~$u\in \cX^s$. For~$t>0$, let~$u_\rho(x):=u(x/\rho)$ for all~$x\in\R^N$.

Then, $u_\rho\in \cX^s$ and
$$
\cE_{2s_1,s}(u,u)=\rho^{2s-N}\cE_{2s_1,s}(u_\rho,u_\rho).
$$ 
\end{lemma}

\begin{proof}
The changes of variable~$\widetilde x:=x/\rho$ and~$\widetilde r:=r/\rho$ give that
\begin{eqnarray*} \frac{2}{c_{2s_1,s}}
\cE_{2s_1,s}(u_\rho,u_\rho)&=&
\iint_{\R^{2N}}\Big(\delta_{s_1}u_\rho(x,y)\Big)^2\,dx\,\nu_s(dy) \\
&=& \iint_{\R^{2N}}\left(\delta_{s_1}u\left(\frac{x}{\rho},\frac{y}{\rho}\right)\right)^2\,dx\,\nu_s(dy) \\
&=&\int_{\R^{N}} \int_{0}^{+\infty}\int_{\bS^{N-1}} r^{-1-2s}
\left(\delta_{s_1}u\left(\frac{x}{\rho},\frac{r\theta}{\rho}\right)\right)^2\,dx\,dr\,
\sigma_s(d\theta)
\\&=& \rho^{N-2s}\int_{\R^N}\int_{0}^{+\infty}\int_{\bS^{N-1}}
\widetilde{r}^{-1-2s}\Big(\delta_{s_1}u(\widetilde{x}, \theta \widetilde{r})\Big)^2\,d\widetilde{x}\,d\widetilde{r}\,\sigma_s(d\theta)
\\&=&\rho^{N-2s}\iint_{\R^{2N}}\Big(\delta_{s_1}u(x,y)\Big)^2\,dx\,\nu_s(dy) \\
&=& \rho^{N-2s}\frac{2}{c_{2s_1,s}}
\cE_{2s_1,s}(u,u),
\end{eqnarray*}
as desired.
\end{proof}

Now, given an open set~$\Omega\subset \R^N$ and~$s>0$,
we define
\begin{equation}\label{vbncmxe8ty849et93476643807}
{\mbox{$\cX^s(\Omega)$ as the closure of the space~$C^\infty_c(\Omega)$
with respect to the norm in~\eqref{normdef098765}.}}\end{equation}
When~$s=0$, we make the ansatz that~$\cX^s(\Omega)=L^2(\Omega)$ (notice that this is compatible with the limit procedure
in Propositions~\ref{limit s to 0} and~\ref{prop:bilinear}). In view of \cref{density in rn}, this definition is consistent for $\Omega=\R^N$ and it follows that $\cX^s(\Omega)$ is a Hilbert space with the same norm as on $\cX^s$.

\begin{prop}\label{poincare}
Let~$s\ge0$ and~$\Omega\subset \R^N$ be open and bounded.

Then, for all~$u\in \cX^s(\Omega)$,
\begin{equation}\label{51IS}
\|u\|_{L^2(\R^N)}^2\leq C\cE_{2s_1,s}(u,u),
\end{equation}
with~$s_1$ defined in~\eqref{sis2s3s4} and
\begin{equation}\label{explicitconstant of poincare}
C:=2^{4s_1+1}\binom{2s_1}{s_1}^{-2}(\diam(\Omega))^{2s}.
\end{equation}

In particular,~$\cE_{2s_1,s}$ is a scalar product which induces an equivalent norm on $\cX^s(\Omega)$.
\end{prop}

\begin{proof}
We first deal with the case~$s=0$. 
Thanks to Propositions~\ref{limit s to 0} and~\ref{prop:bilinear} we have that, for all~$u\in C^\infty_c(\Omega)$,
\begin{eqnarray*}
\lim_{s\to0^+}\cE_{2s_1,s}(u,u)=\lim_{s\to0^+}\int_{\R^N}L_{s_1,s}u(x)u(x)\, dx=\int_\Omega |u(x)|^2\,dx=\|u\|_{L^2(\Omega)}^2.
\end{eqnarray*}
Therefore, the desired inequality holds true for all~$u\in C^\infty_c(\Omega)$, and then for all~$u\in  \cX^s(\Omega)$ by density.

Now we suppose that~$s>0$ and we recall that, for any~$R>0$,
\begin{equation}\label{jiewoytokjhgfdsnbvcxqwertyusdfge45678}
\nu_s\big(\R^N\setminus B_R\big)= 
\int_R^{+\infty}r^{-1-2s}\, dr = \frac{1}{2sR^{2s}}.
\end{equation}
We choose~$R:=2 \diam(\Omega)$: in this way, we have that~$u(x+ y)=0$ for all~$x\in \Omega$ and~$y\in \R^N\setminus B_R$. 
Note that also~$u(x\pm k y)=0$ for all~$x\in \Omega$,
$y\in \R^N\setminus B_R$ and~$ k\in\{1,\ldots,s_1\}$. Thus,
\begin{eqnarray*}&&
\iint_{\R^{2N}}\Big(\delta_{s_1}u(x,y)\Big)^2\,\nu_s(dy)\,dx 
\\&\geq& \int_{\Omega}\int_{\R^N\setminus B_R} \Big(\delta_{s_1}u(x,y)\Big)^2\,\nu_s(dy)\,dx
\\&=&\sum_{{-s_1\le k\le s_1}\atop{-s_1\le j\le s_1}} (-1)^{k+j}
\binom{2s_1}{s_1-k}\binom{2s_1}{s_1-j}\int_{\Omega}\int_{\R^N\setminus B_R}
u(x+ky)u(x+jy)\,\nu_s(dy)\,dx
\\&=& \binom{2s_1}{s_1}^2\|u\|_{L^2(\Omega)}^2\int_{\R^N\setminus B_R}\,\nu_s(dy)\\
&=&\binom{2s_1}{s_1}^2\frac{1}{2sR^{2s}}
\|u\|_{L^2(\Omega)}^2,
\end{eqnarray*}
where we have also used~\eqref{jiewoytokjhgfdsnbvcxqwertyusdfge45678}
in the last line.

Therefore, the inequality in~\eqref{51IS} holds true with the constant
$$
C_0:=\left(\;\frac{c_{2s_1,s}}2\binom{2s_1}{s_1}^2\frac{1}{2^{1+2s}s (\diam(\Omega))^{2s}}\right)^{-1}.
$$
Now, by \cref{cor:bounds on cms} and~\eqref{usoduevoltefprse769043-0876},
and recalling that~$s\in(0,s_1)$, we have
$$
\frac{c_{2s_1,s}}{s}=\frac{c_{2s_1,s}(2s_1-s)}{2s_1}\left(\frac{1}{s}+\frac{1}{2s_1-s}\right)
\geq\frac{ 2^{2-2s_1}(2s_1-s)}{2s_1\,M_{s,\sigma_s}(e_s)}\geq 2^{1-2s_1}.
$$
Hence, we have
$$
C_0\leq 2^{2s_1-1}\left(\; \binom{2s_1}{s_1}^2\frac{1}{2^{2+2s}(\diam(\Omega))^{2s}}\right)^{-1}\leq 2^{4s_1+1}\binom{2s_1}{s_1}^{-2}(\diam(\Omega))^{2s},
$$ and this gives the desired inequality
recalling~\eqref{explicitconstant of poincare}.

{F}rom the inequality in~\eqref{51IS} and Proposition~\ref{hilbert-space}
it follows that~$\cX^s(\Omega)$ is a Hilbert space with scalar
product~$\cE_{2s_1,s}$.
\end{proof}

The following result can be seen as the counterpart of~\cite[Lemma~2.1]{DPSV23} (that deals with Gagliardo seminorms), with additional difficulties due to the fact that the fractional exponent is any arbitrary nonnegative number and the energy involves a very general measure~$\nu_s$.

\begin{prop}\label{prop:spaces ordered}
Let~$\bar{s}>0$. Let~$0<t\le s \le\bar{s}$ and let~$\nu_s$ and~$\nu_t$ satisfy \cref{definition measure}, \cref{assumption basis} and~\cref{simple ellipticity} for~$s$ and~$t$ respectively. Suppose that there exists~$\Lambda\ge\lambda$, possibly depending on~$\bar{s}$, such that
	\begin{equation}
	\label{suppcontster897654} 
	M_{t,\sigma_t} \le \Lambda M_{s,\sigma_s}\quad\text{ on~$\bS^{N-1}$}.
	\end{equation}
	
Then, for all functions~$u:\R^N\to\R$ we have
	\begin{equation}\label{firstclaim}
	\cE_{2t_1,t}(u,u)\leq \frac{\Lambda}{\lambda}\Big(\|u\|_{L^2(\R^N)}^2+\cE_{2s_1,s}(u,u)\Big).
	\end{equation}
In particular, $\cX^s(\Omega)\subset \cX^t(\Omega)$ for all open sets $\Omega\subset \R^N$.
	
If, moreover, $\Omega\subset\R^N$ is an open and bounded set, then, for all~$u\in \cX^s(\Omega)$,
	\begin{equation}\label{firstclaim2}
	\cE_{2t_1,t}(u,u)\leq \frac{\Lambda (1+C)}{\lambda} \cE_{2s_1,s}(u,u),
\end{equation}
	where~$C$ is given in~\eqref{explicitconstant of poincare} and is independent of~$t$.% In particular, also~$\cX^s(\Omega)\subset \cX^t(\Omega)$.
\end{prop}

\begin{proof}
By \cref{fourier rep} we have that, for all~$u\in	\cX^s$,
\begin{eqnarray*}
\cE_{2s_1,s}(u,u)=(M_{s,\sigma_s}(e_s))^{-1}\int_{\R^N} M_{s,\sigma_s}\left(\frac{\xi}{|\xi|}\right)|\xi|^{2s}|\widehat{u}(\xi)|^2\, d\xi.
\end{eqnarray*}
Therefore, recalling also~\eqref{suppcontster897654},
\begin{eqnarray*}&&
\|u\|_{L^2(\R^N)}^2+\cE_{2s_1,s}(u,u)=\int_{\R^N}|\widehat{u}(\xi)|^2\, d\xi+
(M_{s,\sigma_s}(e_s))^{-1}\int_{\R^N} M_{s,\sigma_s}\left(\frac{\xi}{|\xi|}\right)|\xi|^{2s}|\widehat{u}(\xi)|^2\, d\xi
\\&&\quad
\ge (M_{t,\sigma_t}(e_t))^{-1}\int_{\R^N}M_{t,\sigma_t}\left(\frac{\xi}{|\xi|}\right)|\widehat{u}(\xi)|^2\, d\xi+
(\Lambda\, M_{s,\sigma_s}(e_s))^{-1}\int_{\R^N} M_{t,\sigma_t}\left(\frac{\xi}{|\xi|}\right)|\xi|^{2s}|\widehat{u}(\xi)|^2\, d\xi
\\&&\quad
\ge (\Lambda\,M_{s,\sigma_s}(e_t))^{-1}\int_{\R^N}M_{t,\sigma_t}\left(\frac{\xi}{|\xi|}\right)|\widehat{u}(\xi)|^2\, d\xi+
(\Lambda\, M_{s,\sigma_s}(e_s))^{-1}\int_{\R^N} M_{t,\sigma_t}\left(\frac{\xi}{|\xi|}\right)|\xi|^{2s}|\widehat{u}(\xi)|^2\, d\xi
\\&&\quad
\ge (\Lambda\,M_{s,\sigma_s}(e_s))^{-1}\int_{\R^N}M_{t,\sigma_t}\left(\frac{\xi}{|\xi|}\right)\big(1+|\xi|^{2s}\big)|\widehat{u}(\xi)|^2\, d\xi
\\&&\quad
\ge (\Lambda\,M_{s,\sigma_s}(e_s))^{-1}\int_{\R^N}M_{t,\sigma_t}\left(\frac{\xi}{|\xi|}\right)|\xi|^{2t} |\widehat{u}(\xi)|^2\, d\xi.
\end{eqnarray*}

Thus, exploiting again \cref{fourier rep} (and recalling \cref{prop:bilinear}),
\begin{eqnarray*}
\|u\|_{L^2(\R^N)}^2+\cE_{2s_1,s}(u,u)\ge (\Lambda\,M_{s,\sigma_s}(e_s))^{-1}M_{t,\sigma_t}(e_t)\,\cE_{2t_1,t}(u,u)
.\end{eqnarray*}
The inequality in~\eqref{firstclaim} then follows from~\cref{usoduevoltefprse769043-0876} and~\cref{simple ellipticity}.

{F}rom~\eqref{firstclaim} and \cref{poincare}, we also obtain~\eqref{firstclaim2}.
\end{proof}

We close this section with the following compactness statement.

\begin{thm}\label{compact}
	Let~$s>0$ and~$\Omega\subset \R^N$ be an open bounded set. Then~$\cX^s(\Omega)\hookrightarrow L^2(\Omega)$ is compact.
\end{thm}

For the proof we proceed similarly to~\cite[Section~2]{JW20}. However, we need to modify the statements slightly, as we are dealing with measures.

In the following, a continuous operator~$T:E\to L^2(\R^N)$ on a normed vector space is called \textit{locally compact}, if~$R_KT:E\to L^2(\R^N)$ is compact for every compact set~$K\subset \R^N$, where~$R_K:L^2(\R^N)\to L^2(\R^N)$ defined as~$R_Ku=\chi_Ku$ is a kind of cut-off operator represented by the multiplication with the characteristic function~$\chi_K$.

A key step is to isolate the singularity of the measure at zero and use it to conclude the total boundedness in~$L^2(K)$ of a bounded subset~$M$ of~$\cX^s$ for some compact set~$K\subset \R^N$. For this, we define for~$\delta>0$ the following family of \textit{modified measures}
\begin{equation}\label{modified-compact}
\nu_s^{\delta}(U)=\int_{\delta}^{+\infty}\int_{\mathbb{S}^{N-1}}\chi_U(r\theta)r^{-1-2s}\sigma_s(d\theta)\,dr\quad\text{for~$U\in \cB(\R^N)$}
\end{equation}
and we show the following modifications of~\cite[Lemmata~2.1 and~2.2]{JW20}.

\begin{lemma}\label{compact-step1}
For~$\delta>0$ the operator~$
T_{\delta}:L^2(\R^N)\to L^2(\R^N)$ defined as
$$ T_{\delta}u(x):=\int_{\R^N} (u(x-y)+u(x+y))\nu_s^{\delta}(dy)
$$
is locally compact.
\end{lemma}

\begin{proof}
Since
\begin{align*}
\int_{\R^N}|T_{\delta}u(x)|^2\,dx&\leq \int_{\R^N}\nu_s^{\delta}(\R^N)\int_{\R^N}(u(x-y)+u(x+y))^2\nu_s^{\delta}(dy)\, dx\\
&\leq 2\nu_s^{\delta}(\R^N)\Bigg(\ \int_{\R^N}\int_{\R^N}u^2(x-y)\, dx\,\nu_s^{\delta}(dy)+\int_{\R^N}\int_{\R^N}u^2(x+y)\, dx\,\nu_s^{\delta}(dy)\Bigg)\\
&\leq 4\Big(\nu_s^{\delta}(\R^N)\Big)^2\|u\|_{L^2(\R^N)}^2
\end{align*}
it follows that~$T_{\delta}$ is a well-defined bounded, linear operator.

Next, let~$q_n\in C^{\infty}_c(\R^N,[0,+\infty))$ for~$n\in \N$ be such that 
$$\int_{\R^N}q_n(x)\,dx=1\qquad{\mbox{and}}\qquad
\lim_{n\to+\infty}\int_{\R^N\setminus B_{\eps}}q_n(x)\,dx=0.$$ Consider 
$$
w_{\delta,n}(z):=\int_{\R^N}q_n(y-z)\nu_s^{\delta}(dy)\quad \text{for~$n\in \N$.}
$$
Then~$w_{\delta,n}\in L^1(\R^N)$ and the operator~$L^2(\R^N)\to L^2(\R^N)$
given by~$u\mapsto w_{\delta,n}\ast u$ is a locally compact operator by~\cite[Lemma~2.1]{JW20} for every~$n\in \N$. Thus, also the operator
$$
T_{\delta,n}:L^2(\R^N)\to L^2(\R^N),\quad T_{\delta,n}u(x):=\int_{\R^N}\big(u(x-y)+u(x+y)\big)w_{\delta,n}(y)\,dy
$$
is locally compact for every~$n\in \N$.

Note that for any measurable~$U\subset \R^N$ and almost all~$y\in \R^N$
it holds 
$$\lim_{n\to+\infty}\int_U  q_n(y-z)\, dz= \chi_U(y).$$
Thus, the Dominated Convergence Theorem implies
$$ \lim_{n\to+\infty}
\int_U w_{\delta,n}(z)\,dz=\lim_{n\to+\infty}\int_{\R^N}\int_U  q_n(y-z)\, dz\,\nu_s^{\delta}(dy)= \int_{\R^N}\chi_U\nu_s^{\delta}(dy)=\nu_s^{\delta}(U).
$$
This entails that~$w_{\delta,n}\,{\mathcal{H}}^N$ converges weakly (in the sense of measures) to~$\nu_s^{\delta}$, where~${\mathcal{H}}^N$ denotes the Lebesgue measure. 

Finally, note that~$\mu_n=w_{\delta,n}\,{\mathcal{H}}^N-\nu_s^{\delta}$ defines a signed measure. Then, for~$K\subset \R^N$ compact subset,
\begin{align*}
\|\big(R_KT_{\delta}-R_KT_{\delta,n}\big)u\|_{L^2(\R^N)}^2&\leq \int_{K}\Bigg(\ \int_{\R^N}|u(x-y)+u(x+y)| |\mu_n|(dy)\Bigg)^2\,dx\\
&\leq 2|\mu_n|(\R^N)\int_{\R^N}\int_{\R^N}u^2(x-y)+u^2(x+y)|\mu_n|(dy)\\
&\leq 4\|u\|_{L^2(\R^N)}^2|\mu_n|(\R^N).
\end{align*}
Since~$\mu_n$ converges weakly to~$0$ and~$R_KT_{\delta,n}$ is compact for every~$n\in\N$, it follows that also~$R_KT_{\delta}$ is compact as claimed.
\end{proof}

\begin{lemma}\label{compact-step2}
Let~$\delta>0$ and~$T_{\delta}$ be as in \cref{compact-step1}.

Then, for all~$u\in \cX^s(\R^N)$,
$$
\left\|u-\frac{1}{2\nu_s^{\delta}(\R^N)}T_{\delta}u\right\|_{L^2(\R^N)}\leq \frac{1}{2\sqrt{\nu_s^{\delta}(\R^N)}}\,\|u\|_{\cX^s}.
$$
\end{lemma}

\begin{proof}
Let~$u\in \cX^s$. Then by H\"older's inequality
\begin{align*}
\left\|u-\frac{1}{2\nu_s^{\delta}(\R^N)}T_{\delta}u\right\|_{L^2(\R^N)}^2&=\int_{\R^N}\Bigg|u(x)-\frac{1}{2\nu_s^{\delta}(\R^N)}\int_{\R^N}(u(x+y)+u(x-y))\,\nu_s^{\delta}(dy)\Bigg|^2\,dx\\
&=\frac{1}{4(\nu_s^{\delta}(\R^N))^2} \int_{\R^N}\Bigg|\int_{\R^N}(2u(x)-u(x+y)-u(x-y))\,\nu_s^{\delta}(dy)\Bigg|^2\,dx\\
&\leq \frac{1}{4 \nu_s^{\delta}(\R^N)} \int_{\R^N} \int_{\R^N}(2u(x)-u(x+y)-u(x-y))^2\,\nu_s^{\delta}(dy) \,dx\\
&\leq \frac{1}{4\nu_s^{\delta}(\R^N)}\|u\|_{\cX^s}
\end{align*}
by the definition of~$\nu_s^{\delta}$. The claim follows.
\end{proof}

\begin{proof}[Proof of \cref{compact}]
The claim follows once we show that the embedding~$\cX^s\hookrightarrow L^2(\R^N)$ is locally compact. 

Moreover, if~$s\ge \frac12$, we may consider the measure 
$$
\nu_{t}(U)=\int_0^{+\infty}\int_{\mathbb{S}^{N-1}}\chi_U(r\theta)r^{-1-2t}\sigma_s(d\theta)\,dr\quad\text{for~$U\in \cB(\R^N)$}
$$
for some~$t\in(0,\frac12)$. By \cref{prop:spaces ordered} it holds that~$\cX^s\subset \cX^t$ (with~$\sigma_t=\sigma_s$) and a bounded subset~$M\subset \cX^s$ is also bounded in~$\cX^t$. In particular, the local compactness of~$\cX^s\hookrightarrow L^2(\R^N)$ follows once we show that~$\cX^t\hookrightarrow L^2(\R^N)$ is locally compact. Hence, without loss of generality, we assume in the following that~$s\in(0,\frac12)$.

We follow the estimates presented in~\cite[Section~2]{JW20}. Let~$M\subset \cX^s(\R^N)$ be a bounded set and let~$K\subset \R^N$ be compact. Let~$C:=\sup_{u\in M}\|u\|_{\cX^s}$ and let~$\epsilon>0$. Since~$\nu_s(\R^N)=+\infty$, we can choose~$\delta>0$ small enough such that
$$
\nu_s^{\delta}(\R^N)\geq \frac{C^2}{4\epsilon^2}.
$$
Note that, by \cref{compact-step1}, 
$$\widetilde{M}:=\left[\frac1{2\nu_s^{\delta}(\R^N)} R_KT_{\delta}\right](M)$$ is relatively compact in~$L^2(\R^N)$. Moreover, for~$u\in M$, \cref{compact-step2} entails that
$$
\left\|R_Ku-\left[\frac1{2\nu_s^{\delta}(\R^N)} R_KT_{\delta}\right]u\right\|_{L^2(\R^N)}\leq \left\|u-\left[\frac1{2\nu_s^{\delta}(\R^N)} T_{\delta}\right]u\right\|_{L^2(\R^N)}\leq \frac{\|u\|_{\cX^s}}{2\sqrt{\nu_s^{\delta}(\R^N)}}\leq   \epsilon.
$$
Therefore, $R_K(M)$ is contained in the~$\epsilon$-neighborhood of~$\widetilde{M}$. Since~$\epsilon>0$ is chosen arbitrary, it follows that~$R_K(M)$ is totally bounded and thus relatively compact in~$L^2(K)$ as claimed.
\end{proof}

\begin{remark}
As in~\cite{JW20}, it is possible to extend \cref{compact} to possibly unbounded open sets~$\Omega\subset \R^N$ provided that~$|\Omega|<+\infty$. The proof is analogous to the one presented in~\cite[Section~3]{JW20}.
\end{remark}

\section{Some remarks and examples of superposition operators}\label{examples}

\subsection{Definitions and observations on the energy space}\label{defeber859604}
We point out that, though the integration in~\eqref{defi:superposition op} seems to be also in~$s_1$, due to Lemma~\ref{lem:independence of m} we may write the operator as an infinite sum
$$
Lu=\sum_{k=1}^{+\infty}\,\int_{[k-1,k)} L_{k,s}u\,\mu(ds).
$$

\begin{defi}\label{def:chiomega}
For~$\Omega\subset \R^N$ open, let~$\cX(\Omega)$ be
the closure of the space
\begin{equation}\label{mnbvc0987rkgjujrfy}
C^\infty_c(\Omega)\cap\left\{u:\R^N\to\R\,{\mbox{ s.t. }}\, \int_{[0,+\infty)}\cE_{2s_1,s}(u,u)\,\mu^+(ds)<+\infty\right\}
\end{equation}
with respect to the norm
\begin{eqnarray*}
 \|u\|_{\cX(\Omega)}&:=&\left(\|u\|_{L^2(\R^N)}^2+
\int_{[0,+\infty)}\cE_{2s_1,s}(u,u)\,\mu^+(ds)\right)^{\frac12}\\
&=&\left(\|u\|_{L^2(\R^N)}^2+
\sum_{k=1}^{+\infty}\,\int_{[k-1,k)}\cE_{2k,s}(u,u)\,\mu^+(ds)\right)^{\frac12}
.\end{eqnarray*}
\end{defi}

The presence of the intersection in formula~\eqref{mnbvc0987rkgjujrfy} of Definition~\ref{def:chiomega} is motivated by the fact that there exist significant examples of measures~$\mu^+$ for which some, but not all, smooth functions with compact support possess finite norm. 
We present one of these examples in the following result, where
we consider the case~$\nu_s(dz)=\nu_n(dz):=|z|^{-N-2n}\, dz$
(we omit here the normalization factor~$\cH^{N-1}(\bS^{N-1})$ since it does not play any role).

\begin{thm}\label{special construction}
Let~$\varphi\in C^\infty_c\left(\left(-\frac12,\frac12\right)\right)$ and not vanishing identically. Let~$\psi(x):=\varphi(2x)$ for all~$x\in\R$.
Let
$$ c_k:=\| D^k\varphi\|_{L^2(\R)}\qquad{\mbox{and}}\qquad
\mu^+:=\sum_{k=1}^{+\infty} \frac{\delta_k}{c_k\,k^2}.$$

Then,
\begin{eqnarray*}&&
\| \varphi\|_{L^2(\R)}+\sum_{k=1}^{+\infty} \frac{\| D^k \varphi\|_{L^2(\R)}}{c_k\,k^2}<+\infty \\{\mbox{and }} &&
\| \psi\|_{L^2(\R)}+\sum_{k=1}^{+\infty} \frac{\| D^k \psi\|_{L^2(\R)}}{c_k\,k^2}=+\infty
.\end{eqnarray*}
\end{thm}

\begin{proof}
We point out that, for all~$k\in\N$,
\begin{equation}\label{02ei:01-2r3fpeEC}
\| D^k\varphi\|_{L^2(\R)}\ge\| \varphi\|_{L^2(\R)}.
\end{equation}
We check this by induction. The claim is true for~$k=0$.
If~$k=1$ we have that
\begin{eqnarray*}&&
\|\varphi\|_{L^2(\R)}^2=\int_{-1/2}^{1/2}\left|\varphi\left(\frac12\right)-\varphi(x)\right|^2\,dx\le
\int_{-1/2}^{1/2}\left(\int_x^{1/2}|\nabla\varphi(y)|\,dy\right)^2\,dx
\\&&\qquad
\le \int_{-1/2}^{1/2}
\left(\,\int_{-1/2}^{1/2}|\nabla\varphi(y)|^2\,dy
\right)\,dx=\|\nabla\varphi\|^2_{L^2(\R^N)}.
\end{eqnarray*}

We now suppose the statement true for an index~$k\ge1$ and prove it for~$k+1$.
In order to achieve this goal, we point out that there exists~$x_k\in\left(-\frac12,\frac12\right)$ such that~$|D^k\varphi(x_k)|\ge \| \varphi\|_{L^2(\R)}$,
otherwise
$$ \| D^k \varphi\|_{L^2(\R)}^2<\int_{-1/2}^{1/2}\| \varphi\|_{L^2(\R)}^2\,dx
=\| \varphi\|_{L^2(\R)}^2,$$
against our recursive assumption.

On that account,
\begin{eqnarray*}
&&\| \varphi\|_{L^2(\R)}\le|D^k\varphi(x_k)|=
\left|D^k \varphi\left(\frac12\right)-D^k\varphi(x_k)\right|\le
\int_{x_k}^{1/2} |D^{k+1}\varphi(x)|\,dx\\&&\qquad
\le \int_{-1/2}^{1/2} |D^{k+1}\varphi(x)|\,dx
\le \sqrt{
\int_{-1/2}^{1/2} |D^{k+1}\varphi(x)|^2\,dx
}=\| D^{k+1}\varphi\|_{L^2(\R)}.
\end{eqnarray*}
The proof of~\eqref{02ei:01-2r3fpeEC} is thereby complete.

Now we stress that the measure~$\mu^+$ is finite, because, in light of~\eqref{02ei:01-2r3fpeEC},
$$\mu^+\left([0,+\infty)\right)=\sum_{k=1}^{+\infty}\frac1{c_k\,k^2}\le \frac1{\|\varphi\|_{L^2(\R)}}
\sum_{k=1}^{+\infty}\frac1{k^2}<+\infty.$$

Now, recalling Propositions~\ref{limit s to integers}
and~\ref{prop:bilinear},
we see that the norm in Definition~\ref{def:chiomega} boils down to
\begin{equation}\label{Sqw120oVJBAF}\| u\|_{L^2(\R)}+\sum_{k=1}^{+\infty} \frac{\| D^k u\|_{L^2(\R)}}{c_k\,k^2},\end{equation}
which, for~$\varphi$, returns
\begin{equation*}\| \varphi\|_{L^2(\R)}+\sum_{k=1}^{+\infty} \frac{\| D^k \varphi\|_{L^2(\R)}}{\| D^k \varphi\|_{L^2(\R)}
\,k^2}
=\| \varphi\|_{L^2(\R)}+\sum_{k=1}^{+\infty} \frac{1}{
k^2}<+\infty.
\end{equation*}

Yet, while the norm of~$\varphi$ is finite, this is not the case for all smooth, compactly supported functions. Indeed, taking~$\psi(x)=\varphi(2x)$, we have that~$\psi\in C^\infty_c(B_{1/4})$ and
$$ \| D^k\psi\|_{L^2(\R)}^2=2^{2k}\int_{\R} |D^k\varphi(2x)|^2\,dx=
2^{2k-1}\int_{\R} |D^k\varphi(y)|^2\,dy=2^{2k-1} c_k^2.$$
That being the case, we infer from~\eqref{Sqw120oVJBAF} that the norm of~$\psi$ is
\begin{equation*}
\frac1{\sqrt2}\left(
\| \varphi\|_{L^2(\R)}+\sum_{k=1}^{+\infty} \frac{2^k\| D^k \varphi\|_{L^2(\R)}}{c_k\,k^2}\right)
=\frac1{\sqrt2}\left(
\| \varphi\|_{L^2(\R)}+\sum_{k=1}^{+\infty} \frac{2^k}{
k^2}\right)=
+\infty.\qedhere \end{equation*}
\end{proof}

We also stress that condition~\eqref{nonempty xomega} cannot be removed, since there exist finite measures~$\mu^+$ on~$[0,+\infty)$ for which the only smooth function with compact support contained in~$\cX(\Omega)$ is the one that vanishes identically:

\begin{thm}\label{lem:strano09876} Let
$$\mu^+:=\sum_{k=1}^{+\infty}\frac{\delta_k}{k^2}.$$
Then, if~$\phi\in C^\infty_c(\R)$ and 
\begin{equation}\label{vbhut7438t4y7gbuewk}
\sum_{k=1}^{+\infty}\frac{\|D^k\varphi\|_{L^2(\R)}^2}{k^2}<+\infty,
\end{equation}
we have that~$\phi$ necessarily vanishes identically. 

In particular, for any open set~$\Omega\subset \R$ we have~$\cX(\Omega)=\{0\}$ for this choice of~$\mu$.
\end{thm}

\begin{proof}
{F}rom Propositions~\ref{limit s to integers} and~\ref{prop:bilinear}
we see that, for~$k\in \N_0$,
\begin{eqnarray*}
\cE_{2(k+1),k}(\phi,\phi)=\int_{\R} L_{k+1,k}\phi(x)\phi(x)\, dx=
\int_{\R} (-\Delta)^{k}\phi(x)\phi(x)\, dx=
\int_{\R}|D^k\phi(x)|^2\, dx.
\end{eqnarray*}
This entails that
$$
\|\phi\|_{\cX(\Omega)}^2=\|\phi\|_{L^2(\R)}^2+\sum_{k=1}^{+\infty}\frac{\|D^k\phi\|_{L^2(\R)}^2}{k^2}.
$$

Now we prove the first statement of Theorem~\ref{lem:strano09876}.
By contradiction,
we suppose $\phi\not\equiv 0$. Then there exists a sequence~$k_j\to+\infty$
and points~$x_{k_j}\in\R$ such that
$$ |D^{k_j}\phi(x_{k_j})|\ge k_j!,$$
otherwise~$\phi$ would be real analytic (in contradiction with the fact that its support is bounded),
see e.g.~\cite[Corollary~1.2.9]{MR1916029}.

Let also~$R>0$ be such that the support of~$\phi$ is contained in~$(-R,R)$. In this way, it follows that~$x_k\in(-R,R)$ and~$D^k\phi(R)=0$ for $k\in \N_0$.
On this account, we have that
\begin{eqnarray*}
&&k_j!\le|D^{k_j}\phi(R)-D^{k_j}\phi(x_{k_j})|\le
\int_{x_{k_j}}^R |D^{k_j+1}\phi(y)|\,dy\le\int_{-R}^R |D^{k_j+1}\phi(y)|\,dy
\\&&\qquad \le\sqrt{2R
\int_{-R}^R |D^{k_j+1}\phi(y)|^2\,dy}=
\sqrt{2R}\,\|D^{k_j+1}\phi\|_{L^2(\R)}.
\end{eqnarray*}
Consequently, the norm of~$\varphi$ is bounded from below by
$$\sum_{j=1}^{+\infty} \frac{\|D^{k_j+1}\phi\|_{L^2(\R)}^2}{(k_j+1)^2}\ge
\frac1{2R}
\sum_{j=1}^{+\infty} \frac{(k_j!)^2}{(k_j+1)^2}
=+\infty,$$
in contradiction with~\eqref{vbhut7438t4y7gbuewk}. 

The first claim of Theorem~\ref{lem:strano09876} is thereby established
and we now focus on the second claim.
For this, let~$\Omega\subset \R$ be open and bounded and let~$\phi\in \cX(\Omega)$. Then,
$$
\sum_{k=1}^{+\infty}\frac{\|D^k\phi\|_{L^2(\R)}^2}{k^2}
\le \|\phi\|_{\cX(\Omega)}^2<+\infty
$$
and the first part implies that~$\phi$ vanishes identically. Hence,~$\cX(\Omega)=\{0\}$, as claimed.
\end{proof}

Now  we show that the existence
of smooth and compactly supported functions with infinite norm is equivalent to the unboundedness
of the support of the measure~$\mu^+$.

\begin{thm}\label{0qiwdofr09my5062pySgqweSth}
If the support of~$\mu^+$ is unbounded, then there exists~$u\in C^\infty_c(B_1)$ such that
$$\int_{[0,+\infty)} {\mathcal{E}}_{2s_1,s}(u,u)\,\mu^+(ds)=+\infty.$$
\end{thm}

\begin{proof}
Suppose that~$\mu^+$ has unbounded support in~$[0,+\infty)$ (this support being, as usual, the set of all points in~$[0,+\infty)$ for which every open neighborhood has positive measure). Then, there exists a diverging sequence~$s_j$ in this support. Without loss of generality, we can suppose that~$s_{j+1}>s_j+1$. In this way, the open neighborhoods~$I_j:=\left(s_j-\frac14,s_j+\frac14\right)$ are disjoint and~$\mu_j:=\mu^+(I_j)>0$.
We also consider a diverging sequence~$m_j\in\N$ with
\begin{equation}\label{bcjqwry3847t673465t8934ty8fkwjas}
s_j\ge m_j+N+\frac14.\end{equation}

Let~$a_k$ be a sequence of
positive reals, to be specifically chosen here below
with~$a_0=\frac12$. Borel's Lemma (see e.g. page~30 in~\cite{MR346855}) gives that there exists a function~$g\in C^\infty(\R)$ such that~$g^{(k)}(0)=a_k$ for all~$k\in\N$. Since~$g(0)=a_0=\frac12$,
up to modifying~$g$ outside a small neighborhood of the origin, we can suppose that~$g\in C^\infty_c\left(\left(-\frac12,\frac12\right),[0,1]\right)$.
We define $$u(x):=g\left(\frac{|x|^2}2\right).$$ In this way, we see that~$u\in C^\infty_c(B_1)$
and it is a radial function.

Moreover,
\begin{equation}\label{01eu348m9vb6y96nm7Hy4v32} \|u\|_{L^2(\R^N)}^2=
\int_{B_{1}}\left|g\left(\frac{|x|^2}2\right)\right|^2\,dx\le {\mathcal{H}}^N(B_{1}),\end{equation}
where~${\mathcal{H}}^N$ denotes the Lebesgue measure.

Let~$e_i$ be the $i$th element of the Euclidean basis.
We claim that, for each~$m\in\N$ and~$i\in\{1,\dots,N\}$,
\begin{equation}\label{2-pif2521}
D^{2m e_i}u(0)\ge c_m\,a_{m},
\end{equation}
for a suitable constant~$c_m>0$, using the standard multi-index notation.

To prove this, we use the multivariate Fa\`a di Bruno Formula (see e.g.~\cite[Theorem~2.1]{MR1325915}) with the notation~$\gamma(x):=\frac{|x|^2}2$.

Hence, since
$$ D^\ell \gamma(0)=
\begin{cases}
1 &{\mbox{ if $\ell\in\{ 2e_1,\dots,2e_N\}$,}}\\
0&{\mbox{ otherwise,}}
\end{cases}
$$ 
we conclude that
\begin{eqnarray*}
D^{2m e_i}u(0)&=&\sum_{1\le h\le2m} g^{(h)}(0)
\sum_{q=1}^{2m}\sum_{\star\in p_q(h)} (2m)!\prod_{j=1}^q\frac{
\big(D^{\ell_j}\gamma(0)\big)^{k_j}}{k_j!\;(\ell_j!)^{k_j}}\\&\ge&
\sum_{1\le h\le2m}a_h
\sum_{\star\in p_1(h)} (2m)!\frac{
\big(D^{\ell_1}\gamma(0)\big)^{k_1}}{k_1!\;(\ell_1!)^{k_1}}\\&\ge&
a_{m}
\sum_{\star\in p_1(m)} (2m)!\frac{
\big(D^{\ell_1}\gamma(0)\big)^{k_1}}{k_1!\;(\ell_1!)^{k_1}}\\&=&
\frac{a_{m}\,(2m)!}{m!\;2^{m}}.
\end{eqnarray*}
Here above,
the ``$\star$'' in the summation index stands for~$
k_1,\dots,k_q,\ell_1,\dots,\ell_q\in
p_q(h)$, with~$k_1,\dots,k_q\in\N\setminus\{0\}$
and~$\ell_1,\dots,\ell_q\in\big(\N\setminus\{0\}\big)^N$,
for a suitable finite set of indices~$p_q(h)$
with the property that
$$ p_1(m)=\{k_1=m,\; \ell_1=2e_i\}.$$
This completes the proof of~\eqref{2-pif2521}.

We now use a tensorial notation for which~$D^{2m}u$ collects all partial derivatives of order~$m\in\N$
and we infer from~\eqref{2-pif2521} that
\begin{equation*}\|
D^{2m }u\|_{L^\infty(B_1)}\ge \widetilde c_m\,a_{m},
\end{equation*}
for a suitable constant~$\widetilde c_m>0$.

Hence, by the Sobolev Embedding Theorem
(namely, the embedding~$H^{2N}_0(B_1)$ into~$L^\infty(B_1)$) we conclude that
\begin{equation}\label{S3290utotu4m59mbmynm954o}\|
D^{2m +2N}u\|_{L^2(B_1)}\ge\kappa_m\,a_{m},
\end{equation}
for a suitable constant~$\kappa_m>0$.

Now, given~$e$, $\theta\in\bS^{N-1}$ and~$\omega\in\R^N$ with~$|\omega|\le\frac14$, we set~$\widetilde e:=
\frac{e+\omega}{|e+\omega|}$
and observe that
\begin{eqnarray*}&&
\big| |\theta\cdot e|-|\theta\cdot \widetilde e|\big|\le
\left| |\theta\cdot e|-\left| \frac{\theta\cdot e}{|e+\omega|}\right|\right|+
\left|\frac{\theta\cdot\omega}{|e+\omega|}\right|\\&&\qquad\le |\theta\cdot e|
\left| 1- \frac{1}{|e+\omega|}\right|+2|\omega|\le
2\big|\big|e+\omega|-1\big|+2|\omega|\le4|\omega|.
\end{eqnarray*}
As a result, for all~$s\ge\frac12$,
\begin{eqnarray*}&&
\big| |\theta\cdot e|^{2s}-|\theta\cdot \widetilde e|^{2s}\big|
=2s\left|\;
\int^{|\theta\cdot e|}_{|\theta\cdot \widetilde e|}
t^{2s-1}\,dt
\right|\le2s\big| |\theta\cdot e|-|\theta\cdot \widetilde e|\big|\le8s|\omega|
\end{eqnarray*}
and therefore, recalling~\eqref{def ms},
$$\big|M_{s,\sigma_s}(e)-
M_{s,\sigma_s}(\widetilde e)\big|\le
\int_{\bS^{N-1}}\big| |\theta\cdot e|^{2s}-|\theta\cdot \widetilde e|^{2s}\big|\sigma_s(d\theta).
\le8s|\omega|.$$
In particular, setting
$${\mathcal{S}}_s:=
\left\{\frac{e_s+\omega}{|e_s+\omega|},\;
\omega\in\R^N,\;|\omega|\le\min\left\{\frac14,\frac{\lambda}{16s}\right\}
\right\},$$by means of~\eqref{simple ellipticity}
we conclude that, for all~$\zeta\in{\mathcal{S}}_s$,
$$ M_{s,\sigma_s}(\zeta)\ge M_{s,\sigma_s}(e_s)-\frac\lambda2\ge\frac\lambda2.$$

Hence, by Proposition~\ref{fourier rep} and~\eqref{usoduevoltefprse769043-0876}, for all~$s\ge\frac12$,
	\begin{equation}\label{0qdpojfrprh56yth-3}\begin{split}
	\cE_{2s_1,s}(u,u)&=(M_{s,\sigma_s}(e_s))^{-1}\int_{\R^N} M_{s,\sigma_s}\left(\frac{\xi}{|\xi|}\right)|\xi|^{2s}|\widehat{u}(\xi)|^2\, d\xi\\&\ge\int_{\R^N} M_{s,\sigma_s}\left(\frac{\xi}{|\xi|}\right)|\xi|^{2s}|\widehat{u}(\xi)|^2\, d\xi\\&\ge\frac{\lambda}2
\int_{\frac{\xi}{|\xi|}\in{\mathcal{S}}_s} |\xi|^{2s}|\widehat{u}(\xi)|^2\, d\xi.\end{split}\end{equation}
	
We also observe that~$\widehat u$ is radial, since so is~$u$. Hence we can write~$\widehat u(\xi)=U(|\xi|)$ for some~$U:[0,+\infty)\to\C$. This, combined with~\eqref{0qdpojfrprh56yth-3}, gives that
\begin{equation}\label{0ojedftivmy9056iox}
\cE_{2s_1,s}(u,u)\ge
\frac{\lambda}2
\int_{\frac{\xi}{|\xi|}\in{\mathcal{S}}_s}|\xi|^{2s}|U(|\xi|)|^2\, d\xi.
\end{equation}

We now remark that, for each~$s\ge\frac12$, there exist finitely many rotations~$R_{s,1},\dots,R_{s,n_s}$,
with~$n_s$ nondecreasing in~$s$, such that~$\bS^{N-1}=( R_{s,1}{\mathcal{S}}_s)\cup\dots\cup(R_{s,n_s}{\mathcal{S}}_s)$. Consequently, using the change of variable~$\eta:=R_{s,i}^{-1}\,\xi$,
we see that
$$ \int_{\R^N}|\xi|^{2s}|U(|\xi|)|^2\, d\xi
\le\sum_{i=1}^{n_s}
\int_{\frac{\xi}{|\xi|}\in R_{s,i}({\mathcal{S}}_s)}|\xi|^{2s}|U(|\xi|)|^2\, d\xi
=n_s\int_{\frac{\eta}{|\eta|}\in{\mathcal{S}}_s}|\eta|^{2s}|U(|\eta|)|^2\, d\eta.$$
It follows from this observation and~\eqref{0ojedftivmy9056iox} that\footnote{We stress that~\eqref{0qwdoj320949b65m7uRM} can be seen as a refinement of~\eqref{0qwdoj320949b65m7uRM2}.
Specifically, \eqref{0qwdoj320949b65m7uRM2} is valid without any additional symmetry
assumption on the function~$u$, but it relies on the strong ellipticity hypothesis~\eqref{ellipticity}.
Instead, here the function~$u$ is supposed to be rotationally symmetric,
but we do not need to require the strong ellipticity hypothesis~\eqref{ellipticity}.}
\begin{equation}\label{0qwdoj320949b65m7uRM}
\cE_{2s_1,s}(u,u)\ge
\frac{\lambda}{2n_s}
\int_{\R^N}|\xi|^{2s}|U(|\xi|)|^2\, d\xi=\frac{\lambda}{2n_s}
\int_{\R^N}|\xi|^{2s}|\widehat u(\xi)|^2\, d\xi.
\end{equation}
On this account,
\begin{equation}\label{qsd0owvtn04b59ypom6uX7ik7i}
\begin{split}
\int_{[0,+\infty)} {\mathcal{E}}_{2s_1,s}(u,u)\,\mu^+(ds)&\ge
\sum_{j=2}^{+\infty}\int_{I_j} {\mathcal{E}}_{2s_1,s}(u,u)\,\mu^+(ds)\\& \ge\frac{\lambda}{2}
\sum_{j=2}^{+\infty} \int_{I_j} \left(\;
\int_{\R^N}\frac{|\xi|^{2s}|\widehat u(\xi)|^2}{n_s}\, d\xi
\right)\,\mu^+(ds)
\\& \ge\frac{\lambda}{2}
\sum_{j=2}^{+\infty}\frac{1}{\widetilde n_j} \int_{I_j} \left(\;
\int_{\R^N} |\xi|^{2s}|\widehat u(\xi)|^2\, d\xi
\right)\,\mu^+(ds),
\end{split}
\end{equation}for some~$\widetilde n_j$.

We also observe that if~$s\in I_j$, then~$s\ge s_j-\frac14\ge m_j+N$ (thanks to~\eqref{bcjqwry3847t673465t8934ty8fkwjas}) and therefore
\begin{eqnarray*}
&&\int_{\R^N}|\xi|^{2s}|\widehat u(\xi)|^2\, d\xi\ge
\int_{\{|\xi|\ge1\}}|\xi|^{2m_j+2N}|\widehat u(\xi)|^2\, d\xi\\&&\qquad=
\int_{\R^N}|\xi|^{2m_j+2N}|\widehat u(\xi)|^2\, d\xi-
\int_{\{|\xi|<1\}}|\xi|^{2m_j+2N}|\widehat u(\xi)|^2\, d\xi\\&&\qquad\ge
\int_{\R^N}|\xi|^{2m_j+2N}|\widehat u(\xi)|^2\, d\xi-
\int_{\{|\xi|<1\}}|\widehat u(\xi)|^2\, d\xi\\&&\qquad\ge
\int_{\R^N}|\xi|^{2m_j+2N}|\widehat u(\xi)|^2\, d\xi-
\int_{\R^N}|\widehat u(\xi)|^2\, d\xi\\&&\qquad\ge
b_j\int_{\R^N} \big|D^{2m_j+2N} u(x)\big|^2\, dx-
b\int_{\R^N}| u(x)|^2\, dx,
\end{eqnarray*}for suitable positive constants~$b_j$ and~$b$
(depending on the normalization of the Fourier Transform
and the corresponding Plancherel Theorem that we have used in the last step).

Hence, recalling~\eqref{01eu348m9vb6y96nm7Hy4v32} and~\eqref{S3290utotu4m59mbmynm954o},
\begin{eqnarray*}
&&\int_{\R^N}|\xi|^{2s}|\widehat u(\xi)|^2\, d\xi\ge
\widetilde b_j a_{m_j}-\widetilde b,
\end{eqnarray*}for suitable positive constants~$\widetilde b_j$ and~$\widetilde b$.

This, in tandem with~\eqref{qsd0owvtn04b59ypom6uX7ik7i}, returns that
\begin{equation*}
\begin{split}
&\int_{[0,+\infty)} {\mathcal{E}}_{2s_1,s}(u,u)\,\mu^+(ds)\ge\frac{\lambda}{2}
\sum_{j=2}^{+\infty}\frac{1}{\widetilde n_j} \int_{I_j} \left(
\widetilde b_j a_{m_j}-\widetilde b
\right)\,\mu^+(ds)=\frac\lambda{2}
\sum_{j=2}^{+\infty}\frac{\mu_j}{\widetilde n_j} \left(
\widetilde b_j a_{m_j}-\widetilde b
\right).
\end{split}
\end{equation*}It now suffices to choose~$ a_{m_j}:=\frac{\widetilde n_j}{\widetilde b_j \;\mu_j}+\frac{\widetilde b}{\widetilde b_j }
$ to obtain the desired result.
\end{proof}

\subsection{Some examples}
We emphasize that if there exists some~$\widehat{s}>s_{\ast}$ such that~$\mu^+((\widehat{s},+\infty))=0$, then we simply have
$$
Lu=\int_{[0,\widehat{s}]}L_{\widehat{s}_1,s}u\,\mu(ds)\qquad\text{and}\qquad \int_{[0,+\infty)}\cE_{2s_1,s}(u,v)\,\mu^+(ds)=\int_{[0,\widehat{s}]}\cE_{2\widehat{s}_1,s}(u,v)\,\mu^+(ds).
$$

Given~$m>0$, for all~$U\in \cB(\R^N)$ we set
$$
\tilde{\mu}(U):=\int_0^mc_{2m,s}\,\nu_s(U)\,\mu(ds)=\int_0^m c_{2m,s}\int_0^{+\infty}r^{-1-2s} \int_{\bS^{N-1}}\chi_{U}(r\theta)\,\sigma_s(d\theta)\,dr\,\mu(ds). 
$$
This definition is well-posed thanks to assumption~\eqref{Mass}. In particular, since the map~$s\mapsto c_{2m,s}r^{-1-2s}$ is continuous in~$(0,m)$, the measurability assumption~\eqref{Mass} affects the measurability of the family of measures~$s\mapsto \sigma_s$. 

With this notation, using the Fubini-Tonelli Theorem, we see that, for all~$m>0$,
\begin{eqnarray*}&&
\int_0^m\cE_{2m,s}(u,u)\,\mu(ds)=
\int_0^m\left(\;\frac{c_{2m,s}}2 \iint_{\R^{2N}}(\delta_{m}u(x,y))^2\,dx\,\nu_s(dy)\right)\,\mu(ds)\\&&\qquad
=
\frac12\iint_{\R^{2N}}(\delta_{m}u(x,y))^2\,\tilde{\mu}(dy)\,dx.
\end{eqnarray*}

{F}rom these remarks, it follows that~$Lu(x)$ is well-defined for~$u\in C^{\infty}_c(\R^N)$ and~$x\in \R^N$ if there exists~$m>0$ such that~$\mu((m,+\infty))=0$.

In this part we give some examples covered by our framework. This is followed by a short discussion on what kind of assumptions can be added, so that~$Lu(x)$ also exists if~$\mu((m,+\infty))\neq 0$ for all~$m>0$.

\paragraph{Example 1.} The choice 
$$
\nu_s(dz):=|z|^{-N-2s}\,dz\quad\text{for~$s>0$}
$$
leads to the operators~$L_{m,s}=(-\Delta)^s$, the~$s$-th power of the Laplacian, which for~$s\in(0,1)$ is called the fractional Laplacian and for~$s\in \N$ with~$s>1$ is also called the polylaplacian. Note that since~$d\sigma_s(\theta)=d\mathcal{H}^{N-1}_{\theta}$ coincides with the surface measure for~$\mathbb{S}^{N-1}$, the family~$\nu_s(dz)$ clearly satisfies~\eqref{assumption basis} (after normalization) and
the strong ellipticity assumption~\eqref{ellipticity}. 

With different choices of~$\mu$, the following operators can be represented by the superposition operator~$L$ such that~\eqref{mu-assumption1}, \eqref{mu-assumption2}, \eqref{Mass}, and~\eqref{Eass} are satisfied, see also~\cite{DPSV23}. We provide here below a list of possible choices:
\begin{itemize}
\item[(1)] With~$\mu=\delta_s$, the Dirac measure with weight in~$s>0$, $L$ reduces to~$(-\Delta)^s$.
\item[(2)] With~$\mu=\delta_{s}+\delta_{t}$ where~$s$, $t>0$, $L$ reduces to the sum of~$(-\Delta)^{s}+(-\Delta)^{t}$. In particular, with~$s\in \N$ and~$t\notin\N$, $L$ can be seen as a mixed local-nonlocal operator.
\item[(3)] With~$\mu=\delta_{s}-\alpha \delta_{t}$, where~$s>t>0$ and~$\alpha>0$ is small enough, we have~$L=(-\Delta)^{s}-\alpha(-\Delta)^{t}$. Thus~$L$ is a representative for a class of operators with competing trends.
\item[(4)] More generally, $\mu=\delta_{s}+\delta_{t}-\alpha(\delta_{a}+\delta_{b})$, where~$t$, $a$, $b>0$, $s>\max\{t,a,b\}$ and~$\alpha>0$ is small enough, $L=(-\Delta)^{s}+(-\Delta)^{t}-\alpha((-\Delta)^{a}+(-\Delta)^{b})$ describes a competing operator of possibly mixed local-nonlocal interactions.
\item[(5)] Let~$(c_k)_{k\in\N}\subset \R$ be a sequence such that
$$\sum_{k=0}^{+\infty}c_k\quad{\mbox{is convergent}} $$ and satisfies the following:
$$
\text{there exists~$\bar{k}\in \N$ with~$c_k>0$ for~$k\in\{1,\ldots \bar{k}\}$ and }\sum_{1}^{\bar{k}}c_k\leq \gamma \sum_{k=\bar{k}+1}^{+\infty}c_k
$$
for some~$\gamma\in[0,1)$. Then, the choice
$$\mu=\sum_{k=1}^{+\infty}c_k\delta_{s_k},$$ where~$(s_k)_k\subset(0,+\infty)$ is
strictly decreasing, leads to the operator
$$L=\sum_{k=1}^{+\infty}c_k(-\Delta)^{s_k}.$$
\item[(6)] More generally, for any measurable, nontrivial function~$f:[0,+\infty)\to\R$ such that there exists~$s_{\ast}>0$ with
$$
f\geq 0\quad\text{in~$[s^{\ast},+\infty)$ and }\quad \int_{0}^{s_{\ast}}f(s)\,ds\leq \gamma\int_{s_\ast}^{+\infty}f(s)\,ds
$$
for some~$\gamma\in[0,1)$ we may consider~$\mu=f(s)\,ds$. This leads to the operator
$$Lu=\int_0^{+\infty}(-\Delta)^suf(s)\,ds.$$ 
\end{itemize}

We make some observations on the example in~(6) above.
If~$f>0$ on~$(s_{\ast},+\infty)$, it is a rather delicate question if~$Lu$ exists for~$u\in C^{\infty}_c(\R^N)$ c.~f. the discussion in \cref{defeber859604}. To give a suitable assumption on $f$ so that~\eqref{nonempty xomega} is satisfied, we take a closer look at \cref{lem:evaluation} and its proof for~$u\in C^{\infty}_c(\R^N)$.

\begin{lemma}\label{evaluation test function}
Let~$u\in C^{\infty}(\R^N)$ with $|u|\leq 1$, $m\in \N$, and~$s\in(0,m)$. Then, for all~$x\in \R^N$,
$$
|L_{m,s}u(x)|\leq  \frac{C_m\,4^{m-1}\mathcal{H}^{N-1}(\mathbb{S}^{N-1})}{M_{s,\sigma_s}(e_s)}\max\big\{1,\|D^mu(x)\|\big\},
$$
where~$C_m$ is the constant appearing in
Corollary~\ref{cor:bounds on cms}, $D^mu(x)$ denotes the differential of order~$m$ of~$u$ at~$x$
and~$\|D^mu(x)\|$ its respective norm.
\end{lemma}

\begin{proof}
Let~$x$, $y\in \R^N$.  If~$|y|\leq 1$, then
\begin{align*}
|\delta_mu(x,y)|&\leq \|D^mu(x)\|\sum_{k=-m}^{m}\binom{2m}{m-k} |y|^{2m}=4^m \|D^mu(x)\|\, |y|^{2m}
\end{align*}
and, if~$|y|\geq 1$, then 
\begin{align*}
|\delta_mu(x,y)|&\leq \sum_{k=-m}^{m}\binom{2m}{m-k} = 4^m.
\end{align*}
Thus, for all~$x$, $y\in \R^N$,
$$
|\delta_mu(x,y)|\leq 4^m\max\big\{1,\|D^mu(x)\|\big\}\min\big\{1,|y|^{2m}\big\}.
$$
Therefore, using Corollary~\ref{cor:bounds on cms}, we have that
\begin{eqnarray*}
|L_{m,s}u(x)|&&\leq \frac{c_{m,s}4^m}{2}\mathcal{H}^{N-1}(\mathbb{S}^{N-1})\max\big\{1,\|D^mu(x)\|\big\}
\int_0^{+\infty}\min\big\{t^{2m},1\big\}t^{-1-2s}\,dt\\
&&= c_{m,s}4^{m-1}\mathcal{H}^{N-1}(\mathbb{S}^{N-1})\max\big\{1,\|D^mu(x)\|\big\}\left(\frac{1}{s}+\frac{1}{m-s}\right)\\
&&\leq \frac{C_m\,4^{m-1}\mathcal{H}^{N-1}(\mathbb{S}^{N-1})}{M_{s,\sigma_s}(e_s)}\max\big\{1,\|D^mu(x)\|\big\},
\end{eqnarray*} 
as desired.
\end{proof}

\begin{cor}\label{coro:05348953jfk}
Let $u\in C^\infty_b(\R^N)$ and let~$f:\R^N\to[0,+\infty)$ be a measurable function such that
\begin{equation}\label{vbjkwe23456789}
\sup_{x\in \R^N}\sum_{m=1}^{+\infty}C_m\,4^{m-1}\max\big\{1,\|D^mu(x)\|\big\}\int_{m-1}^{m}f(s)\, ds<+\infty,
\end{equation} where~$C_m$ is the constant appearing in
Corollary~\ref{cor:bounds on cms}.

If~\eqref{simple ellipticity} is satisfied, then 
$$
Lu(x)=\int_0^{+\infty}(-\Delta)^su(x)f(s)\,ds
$$
is finite for every~$x\in \R^N$.

If in addition~$u$ is compactly supported in an open bounded set~$\Omega\subset \R^N$, then~$\|u\|_{\cX(\Omega)}<+\infty$.
\end{cor}

\begin{proof}
In view of Lemma~\ref{evaluation test function} and the assumptions in~\eqref{simple ellipticity} and~\eqref{vbjkwe23456789}, we find that 
\begin{eqnarray*}
|Lu(x)|&\leq& \sum_{m=1}^{+\infty}\frac{C_m\,4^{m-1}\mathcal{H}^{N-1}(\mathbb{S}^{N-1})}{M_{s,\sigma_s}(e_s)}\max\big\{1,\|D^mu(x)\|\big\}
\int_{m-1}^{m} f(s)\, ds\\
&\leq& \frac{ \mathcal{H}^{N-1}(\mathbb{S}^{N-1})}{\lambda }\sum_{m=1}^{+\infty} 4^{m-1} \max\big\{1,\|D^mu(x)\|\big\}
\int_{m-1}^{m} f(s)\, ds\\&<&+\infty.
\end{eqnarray*}

If, in addition, $\supp\,u\subset \Omega$, then thanks to Proposition~\ref{prop:bilinear} and the first statement of Corollary~\ref{coro:05348953jfk},
we find that
$$
\|u\|_{\cX(\Omega)}^2=\|u\|_{L^2(\Omega)}^2+\int_{\Omega}Lu(x)u(x)\, dx<+\infty
$$
as claimed.
\end{proof}

\begin{remark}
We collect next some particular cases in which~$Lu(x)$ is well-defined.
\begin{enumerate}
\item As noted above, if~$f=0$ on~$(s^{\ast},+\infty)$ for some~$s^{\ast}>0$, then $Lu$ exists pointwisely in $\R^N$ for every $u\in C^{\infty}_b(\R^N)$.
\item Similar to the construction in \cref{special construction} it holds: For any~$u\in C^{\infty}_b(\R^N)$ let
$$
U_m=\sup_{x\in\R^N}\max\{1,\|D^mu(x)\|\}.
$$
If there is~$g\in L^1([0,\infty))$ such that for any $m\in \N$ the function $f$ satisfies
$$
0\leq f(s)\leq \frac{4^{1-m}}{C_mU_m}g(s) \quad\text{for $s\in[m-1,m)$,}
$$
where $C_m$ is the constant appearing in
Corollary~\ref{cor:bounds on cms}, then $Lu(x)$ exists for any $x\in \R^N$.
\end{enumerate}
\end{remark}

\paragraph{Example 2.} In Example~1 we focused on the interplay between different orders of operators. Next, we consider the case where~$\sigma_s$ differs from the surface measure of the sphere.
\begin{enumerate}
\item We may choose 
\begin{equation}\label{bvncmx45y4oifhewi098765}
\sigma_s=\sum_{k=1}^{N}\delta_{\theta_k}
\end{equation} for a family of linear independent vectors~$\theta_1,\ldots \theta_N\in \mathbb{S}^{N-1}$. With~$\mu=\delta_s$ for~$s>0$ it follows that~$L$ is the sum 
of one-dimensional (fractional) Laplacians given by
$$Lu(x)=\sum_{k=1}^{N}[(-\Delta)^su_{x,\theta_k}](x\cdot \theta_k),$$ where~$u_{x,e}(t)=u(x+(t-x\cdot e)e)$ for~$t\in \R$, $x\in \R^N$, $e\in \mathbb{S}^{N-1}$, and~$(-\Delta)^s$ is the one-dimensional fractional Laplacian.
\item With~$\sigma_s$ as in~\eqref{bvncmx45y4oifhewi098765}, we can then consider all choices of~$\mu$ as in Example~1. In particular, let~$s^{(1)},\ldots,s^{(n)}\in (0,+\infty)$ be a sequence of parameters, and for each~$k\in \{1,\ldots,n\}$ let~$\theta_k\in \mathbb{S}^{N-1}$. Then, with the choice
$$\mu=\sum_{k=1}^{n}\delta_{s^{(k)}}\qquad {\mbox{and}}\qquad\sigma_{s^{(k)}}=\delta_{\theta_k}$$ we have
$$
Lu(x)=\sum_{k=1}^{N}(-\Delta)^{s^{(k)}}u_{x,\theta_k}(x\cdot \theta_k).
$$
Particularly the choice~$\theta_k=e_k$, with~$e_k$ being the~$k$-th unit vector, is of interest. In this case we have
\begin{equation}\label{9dihf3iu4tumy9b5}
Lu(x)=\sum_{k=1}^{N}\Big(-\partial^2_{x_k}\Big)^{s^{(k)}}u(x)
\end{equation}
as a sum of one-dimensional fractional Laplacians. 

Let us emphasize that the examples in 1. and 2. in particular satisfy~\eqref{Eass}.
\item Finally, one may also consider competing lower order terms in some directions in addition to the above examples by choosing~$\mu^-$ nontrivial.
\end{enumerate}

\paragraph{Example 3.} As a final example, let us note that~$\sigma_s$ can also be chosen as a sum of different powers of the Laplacian acting on different lower dimensional spaces. More precisely, for~$x\in \R^N$ let~$x=(x_a,x_b)$, where~$x_a\in \R^k$ and~$x_b\in \R^{N-k}$, and for a set~$U\subset \R^N$ let 
$$
U_b:=\big\{x_b\;:\; (x_a,x_b)\in U\text{ for some~$x_a\in \R^{k}$}\big\}.
$$
Then, consider
$$\sigma_s(A):=\sum_{i=1}^{k}\delta_{e_i}(A)+\cH^{N-1-k}(A\cap \mathbb{S}^{N-k-1}).$$ The operator~$L_{m,s}$ then is represented by
$$
L_{m,s}u(x)= \sum_{i=1}^{k}(-\Delta)^s_{x_i}u(x_a,x_b)+(-\Delta)^s_{x_b}u(x_a,x_b).
$$

Similarly, the operator~$L$ can be given as a sum of different order operators in this setting, e.g.
$$
Lu(x)=\sum_{i=1}^{k}(-\Delta)^{s^{(i)}}_{x_i}u(x_a,x_b)+(-\Delta)^{s^{(k+1)}}_{x_b}u(x_a,x_b)
$$
for~$s^{(1)}, \ldots,s^{(k+1)}>0$.

\section{Functional analytic properties of the superposition}\label{sec:45}

Here we provide the functional analytic setting needed to deal with the
mixed order operator~$L$ given in~\eqref{defi:superposition op}. For this purpose, we denote by~$\cE_{m,0}$ the~$L^2$ scalar product for any~$m\in \N$.

In the following, $\Omega\subset \R^N$ is an open set and we assume that~$\mu^+$ and~$\mu^-$ satisfy~\eqref{mu-assumption1} and~\eqref{mu-assumption2} for some~$s_\ast$, $\gamma>0$. Moreover, $\nu_s$ are given with respect to a family of probability measures~$\sigma_s$ for~$s>0$ which satisfy~\eqref{simple ellipticity} and~\eqref{Mass}.

For all all~$u$, $v\in \cX(\Omega)$, let
\begin{equation}\label{65yt4revc987654gvreg4polkmn}\begin{split}
\cE_+(u,v)&:=\int_{[0,+\infty)}\cE_{2s_1,s}(u,v)\,\mu^+(ds)\\
{\mbox{and }}\quad
\cE_-(u,v)&:=\int_{[0,+\infty)}\cE_{2s_1,s}(u,u)\,\mu^-(ds)=\int_{[0,s_{\ast})}\cE_{2s_1,s}(u,v)\,\mu^-(ds),
\end{split}\end{equation}
where~$\cE_{2s_1,s}(u,u)$ is defined in~\eqref{098765099werfgthkmnbgfvd876} for~$s>0$, $\cE_{2,0}(u,u)=\|u\|_{L^2(\R^N)}^2$, and the notation in~\eqref{sis2s3s4} is in use. 

We also recall that the spaces~$\cX^s(\Omega)$ and~$\cX(\Omega)$
have been introduced in formula~\eqref{vbncmxe8ty849et93476643807} and Definition~\ref{def:chiomega} respectively.

To use the results of \cref{sec:var23}, we introduce the following \textit{integrated measure} with respect to a given parameter~$t\ge s_\ast$
such that~$\mu^+([s_\ast,t])>0$ (notice that the existence of such a parameter is guaranteed by the assumption on~$\mu$ in~\eqref{mu-assumption1}). 
That is, we let~$\tilde{\sigma}_t$ be a probability measure on~$\bS^{N-1}$ given for~$U\in \cB(\bS^{N-1})$ by
\begin{equation}\label{measure ast}
\tilde{\sigma}_t(U):=\frac{1}{\mu^+([s_\ast,t])}
\int_{[s_\ast,t]}\sigma_s(U)\,\mu^+(ds).
\end{equation}
We also take~$\tilde{e}_t\in \bS^{N-1}$ such that
\begin{equation}\label{measure ast and e}
\max_{e\in \bS^{N-1}}M_{s_\ast,\tilde{\sigma}_t}(e)=M_{s_\ast,\tilde{\sigma}_t}(\tilde{e}_t)=\int_{\bS^{N-1}}|\tilde{e}_t\cdot \theta|^{2s_\ast}\,
\tilde{\sigma}_t(d\theta).
\end{equation}

Moreover, we let
$$
\tilde{\nu}_t(U):= \int_0^{+\infty} \int_{\bS^{N-1}}\chi_U(r\theta) r^{-1-2t}\,\tilde{\sigma}_t(d\theta)\,dr\quad\text{for~$U\in \cB(\R^N)$}.
$$
Analogously to Section~\ref{sec:var23} we then define
\begin{eqnarray*}
\tilde{L}_{m,t}u(x)&=&\frac{\tilde{c}_{m,t}}{2} \int_{\R^N}\delta_mu(x,y) \,\tilde{\nu}_t(dy) \\
{\mbox{and }} \qquad \tilde{\cE}_{2t_1,t}(u,u)&=&\frac{\tilde{c}_{2t_1,t}}{2}\iint_{\R^{2N}}(\delta_{t_1}u(x,y))^2 \,dx\,\tilde{\nu}_t(dy)
\end{eqnarray*}
as the respective operator and associated bilinear form to the measure~$\tilde{\nu}_t$.

Finally, we let~$\tilde{\cX}^{t}(\Omega)$ be the associated Hilbert space, as defined in~\eqref{vbncmxe8ty849et93476643807} with~$\tilde{\nu}_t$ in place of~$\nu_t$.

With this notation, we have the following:

\begin{prop}\label{prop of x omega part2}
Let~$\Omega\subset\R^N$ be open and bounded and 
let~$\mu^+$ be a nonnegative Borel measure on~$[0,+\infty)$.

Then, for all~$t\ge s_\ast$ such that~$\mu^+([s_\ast,t])>0$ there exists~$C_t>0$
such that, for any~$u\in \cX(\Omega)$,
\begin{equation}\label{poincare gen}
{C}_t\mu^+([s_\ast,t])\|u\|_{L^2(\Omega)}^2\leq \cE_+(u,u)
\end{equation}
and
\begin{equation}\label{vcniewr637rygufh894whtigvnigy8hgrie}  
C_t\,\mu^+([s_\ast,t]) M_{s_\ast,\tilde{\sigma}_t}(\tilde{e}_{t})
\tilde{\cE}_{2(s_\ast)_1,s_\ast}(u,u)\leq \cE_+(u,u).
\end{equation}
In particular, $\cX(\Omega)\subset \tilde{\cX}^{s_\ast}(\Omega)$. 

Also,~$\cX(\Omega)$ is a uniformly convex Hilbert space with scalar product~$\cE_+(u,v)$ and the embedding~$\cX(\Omega)\hookrightarrow L^2(\Omega)$ is compact.
\end{prop}

\begin{proof}
We begin with the proof of the Poincar\'e inequality~\eqref{poincare gen}. For this, note that
by Proposition~\ref{poincare} we have, for any~$s\in[0,t]$, that
\begin{equation}\label{vbncmx7345y687342tu}
\cE_{2s_1,s}(u,u)\geq 2^{-1-4s_1}\binom{2s_1}{s_1}^2(\diam(\Omega))^{-2s}\|u\|_{L^2(\Omega)}^2.
\end{equation}
We observe that, using Gautschi's inequality, for~$m\in \N$ it holds that
\begin{align*}
2^{-1-4m}\binom{2m}{m}^2=\frac{(2m)!^2}{m!^4 2^{1+4m}}=\frac{(\Gamma(m+\frac12))^2}{2\pi (\Gamma(m+1))^2}\geq \frac{1}{2\pi (m+1)}.
\end{align*}
Hence, using this inequality with~$m:=s_1$, we see that 
\begin{equation*} \begin{split}&
2^{-1-4s_1}\binom{2s_1}{s_1}^2(\diam(\Omega))^{-2s}\geq 
\frac{1}{2\pi (s_1+1)}(\diam(\Omega))^{-2s}
\\&\qquad\geq
\frac{1}{2\pi (t_1+1)}\min\left\{1,(\diam(\Omega))^{-2t}\right\}=:{C}_t.
\end{split}\end{equation*}
This and~\eqref{vbncmx7345y687342tu} entail that, for any~$s\in[0,t]$,
$$
\cE_{2s_1,s}(u,u)\geq {C}_t\|u\|_{L^2(\Omega)}^2
$$ and therefore
\begin{eqnarray*}
\cE_+(u,u)
=\int_{[0,+\infty)}\cE_{2s_1,s}(u,u)\,\mu^+(ds)\geq\int_{[s_\ast,t]}\cE_{2s_1,s}(u,u)\,\mu^+(ds)
\geq {C}_t\mu^+([s_\ast,t])\|u\|_{L^2(\Omega)}^2,
\end{eqnarray*} which gives the desired inequality in~\eqref{poincare gen}.

We now check~\eqref{vcniewr637rygufh894whtigvnigy8hgrie}. For this, we exploit~\eqref{poincare gen} to see that
\begin{eqnarray*}&&\left(1+ \frac{1}{{C}_t}\right)\cE_+(u,u)\geq \cE_+(u,u)+\mu^+([s_\ast,t])\|u\|_{L^2(\R^N)}^2
\geq \int_{[s_\ast,t]} \left(\cE_{2s_1,s}(u,u)+
\|u\|_{L^2(\R^N)}^2\right)\,\mu^+(ds).
\end{eqnarray*}
Thus, using Proposition~\ref{fourier rep} and~\eqref{usoduevoltefprse769043-0876},
\begin{eqnarray*}
\left(1+ \frac{1}{{C}_t}\right)\cE_+(u,u)&\geq&
\int_{[s_\ast,t]} \left(\;\int_{\R^N} \left((M_{s,\sigma_s}(e_s))^{-1}
M_{s,\sigma_s}\left(\frac{\xi}{|\xi|}\right)|\xi|^{2s}+1\right)|\widehat{u}(\xi)|^2\,d\xi \right)\,\mu^+(ds)\\&\geq&
\int_{[s_\ast,t]} \left(\;\int_{\R^N} \left(
M_{s,\sigma_s}\left(\frac{\xi}{|\xi|}\right)|\xi|^{2s}+1\right)|\widehat{u}(\xi)|^2\,d\xi \right)\,\mu^+(ds).
\end{eqnarray*}
{F}rom this and Fubini's Theorem, recalling also that~$\sigma_s$ is a probability measure, we infer that
\begin{eqnarray*}
\left(1+ \frac{1}{{C}_t}\right)\cE_+(u,u)&\geq&\int_{\R^N} 
\left(\;\int_{[s_\ast,t]}\left(
M_{s,\sigma_s}\left(\frac{\xi}{|\xi|}\right)|\xi|^{2s}+1\right)\,\mu^+(ds)\right)|\widehat{u}(\xi)|^2\,d\xi 
\\&=&
\int_{\R^N}\left(\; \int_{[s_\ast,t]}\int_{\bS^{N-1}} \left(|\theta\cdot \xi|^{2s}+1\right)\,\sigma_s(d\theta)\,  \mu^+(ds)\right) |\hat{u}(\xi)|^2\, d\xi
\\&\geq&\int_{\R^N}\left(\; \int_{[s_\ast,t]}\int_{\bS^{N-1}}
|\theta\cdot \xi|^{2s_\ast} \,\sigma_s(d\theta)\,  \mu^+(ds)\right) |\hat{u}(\xi)|^2\, d\xi\\
&=&\mu^+([s_\ast,t])
\int_{\R^N}\left(\; \int_{\bS^{N-1}}
|\theta\cdot\xi|^{2s_\ast}\,\tilde\sigma_t(d\theta)\right) |\hat{u}(\xi)|^2\, d\xi,
\end{eqnarray*} where the measure~$\tilde\sigma_t$ is defined in~\eqref{measure ast}.

As a result,
\begin{eqnarray*}
\left(1+ \frac{1}{{C}_t}\right)\cE_+(u,u)&\geq&\mu^+([s_\ast,t])
\int_{\R^N}\left(\; \int_{\bS^{N-1}}
\left|\theta\cdot\frac{\xi}{|\xi|}\right|^{2s_\ast} \,\tilde\sigma_t(d\theta)\right)
|\xi|^{2s_\ast}|\hat{u}(\xi)|^2\, d\xi \\&=&\mu^+([s_\ast,t])
\int_{\R^N} M_{s_\ast,\tilde\sigma_t}\left(\frac{\xi}{|\xi|}\right)|\xi|^{2s_\ast}|\hat{u}(\xi)|^2\, d\xi\\&=&
\mu^+([s_\ast,t]) M_{s_\ast,\tilde{\sigma}_t}(\tilde{e}_{t})
\tilde{\cE}_{2(s_\ast)_1,s_\ast}(u,u),
\end{eqnarray*}
where~$\tilde{e}_{t}\in \bS^{N-1}$ is chosen as in~\eqref{measure ast and e}.
Rearranging gives the inequality~\eqref{vcniewr637rygufh894whtigvnigy8hgrie}.

To see the completeness of~$\cX(\Omega)$, let~$(u_n)_n$ be a Cauchy sequence with respect to~$\cE_+$. Then~\eqref{vcniewr637rygufh894whtigvnigy8hgrie} implies that~$(u_n)_n$ is a Cauchy sequence in~$\tilde{\cX}^{s_\ast}(\Omega)$, that is, there exists~$u\in \tilde{\cX}^{s_\ast}(\Omega)$ such that~$u_n\to u$ in~$\tilde{\cX}_{s_\ast}(\Omega)$ as~$n\to+\infty$. Passing to a subsequence, we may assume that
$$
u_n\to u\quad\text{in~$L^2(\Omega)$ and pointwisely almost everywhere as~$n\to +\infty$.}
$$
{F}rom here, the proof follows analogously to the proof of \cref{hilbert-space} with Fatou's Lemma.

The proof of the uniform convexity is analogous to the one of \cref{unif-convex}.

Finally, the compactness statement is a consequence of the continuous embedding~$\cX(\Omega)\hookrightarrow\tilde{\cX}^{s_\ast}(\Omega)$ by the first part and
the compact embedding~$\tilde{\cX}^{s_\ast}(\Omega)\hookrightarrow L^2(\Omega)$ by Theorem~\ref{compact}. 
\end{proof}

\begin{lemma}\label{prop of x omega part1}
Let~$\Omega\subset\R^N$ be open and bounded,
let~$\mu$ satisfy~\eqref{mu-assumption1} and~\eqref{mu-assumption2}
and let~$s_\ast$ be as in~\eqref{mu-assumption1} and~\eqref{mu-assumption2}. 

Let~$\tilde{s}\ge s_\ast$ be such that~$\mu^+([s_\ast,\tilde{s}])>0$ and
\begin{equation}\label{345678fghjew0-78lkjhgcx}
\tilde{\lambda}:=\inf_{e\in\bS^{N-1}}
M_{s_\ast, \tilde{\sigma}_{\tilde{s} }}(e)>0,
\end{equation}
where~$\tilde{\sigma}_{\tilde{s}}$ is the probability measure given in~\eqref{measure ast}.

Then, there exists~$\tilde{C}>0$, depending on~$\mu^+$, $s_\ast$,
$\tilde{s}$ and~$\tilde{\lambda}$, such that for all~$u\in \cX(\Omega)$,
\begin{equation}\label{sdfgy8932r576ygh-765432qfghokijhgf}
\cE_{s_\ast}(u,u)\leq \tilde{C} \cE_+(u,u).\end{equation}
In particular, $\cX(\Omega)\subset \cH^{s_\ast}_0(\Omega)$, where~$\cE_{s_\ast}$ 
is the Gagliardo semi-norm (defined in~\eqref{defesse097})
and
$$
\cH^{s_\ast}_0(\Omega):=\{u\in H^{s_\ast}(\R^N)\;:\; u\chi_{\R^N\setminus \Omega}\equiv 0\}.
$$

Moreover, the following embeddings hold:
\begin{enumerate}
\item If~$s_\ast<\frac{N}{2}$,
then~$\cX(\Omega)\subset L^{ 2_{ s_\ast }^{\ast} }(\R^N)$ with~$2_{s_\ast}^{\ast}:=\frac{2N}{N-2s_\ast}$.
\item If~$s_\ast\geq \frac{N}{2}$ and~$s_\ast-\frac{N}{2}\notin\N_0$, then~$\cX(\Omega)\subset C^{s_\ast-\frac{N}{2}}(\R^N)$.
\end{enumerate}
\end{lemma}

\begin{proof}
We can exploit
Proposition~\ref{prop of x omega part2} with~$t:=\tilde{s}$ and obtain
from~\eqref{poincare gen} that, for any~$u\in \cX(\Omega)$,
\begin{equation*}
{C}_{\tilde{s}}\mu^+([s_\ast,\tilde{s}])\|u\|_{L^2(\Omega)}^2\leq \cE_+(u,u).
\end{equation*}
As a result, using also Proposition~\ref{fourier rep} and~\eqref{usoduevoltefprse769043-0876},
\begin{eqnarray*}
\left(1+ \frac{1}{{C}_{\tilde{s}}}\right)\cE_+(u,u)
&\geq& \cE_+(u,u)+\mu^+([s_\ast,\tilde{s}])\|u\|_{L^2(\R^N)}^2\\
&\geq&
\int_{[s_\ast,\tilde{s}]} \left(\;\int_{\R^N} \left((M_{s,\sigma_s}(e_s))^{-1}
M_{s,\sigma_s}\left(\frac{\xi}{|\xi|}\right)|\xi|^{2s}+1\right)|\widehat{u}(\xi)|^2\,d\xi \right)\,\mu^+(ds)\\&\geq&
\int_{[s_\ast,\tilde{s}]} \left(\;\int_{\R^N} \left(
M_{s,\sigma_s}\left(\frac{\xi}{|\xi|}\right)|\xi|^{2s}+1\right)|\widehat{u}(\xi)|^2\,d\xi \right)\,\mu^+(ds)\\
&=&\int_{\R^N} 
\left(\;\int_{[s_\ast,\tilde{s}]}\left(
M_{s,\sigma_s}\left(\frac{\xi}{|\xi|}\right)|\xi|^{2s}+1\right)\,\mu^+(ds)\right)|\widehat{u}(\xi)|^2\,d\xi\\&=&
\int_{\R^N}\left(\; \int_{[s_\ast,\tilde{s}]}\int_{\bS^{N-1}} \left(|\theta\cdot \xi|^{2s}+1\right)\,\sigma_s(d\theta)\,  \mu^+(ds)\right) |\hat{u}(\xi)|^2\, d\xi \\&\geq&\int_{\R^N}\left(\; \int_{[s_\ast,\tilde{s}]}\int_{\bS^{N-1}}
|\theta\cdot \xi|^{2s_\ast} \,\sigma_s(d\theta)\,  \mu^+(ds)\right) |\hat{u}(\xi)|^2\, d\xi \\&=&\mu^+([s_\ast,\tilde{s}])
\int_{\R^N}\left(\; \int_{\bS^{N-1}}
|\theta\cdot\xi|^{2s_\ast}\,\tilde\sigma_{\tilde{s}}(d\theta)\right) |\hat{u}(\xi)|^2\, d\xi,
\end{eqnarray*}
where the measure~$\tilde\sigma_{\tilde{s}}$ is defined in~\eqref{measure ast}.

Consequently, recalling~\eqref{345678fghjew0-78lkjhgcx},
\begin{eqnarray*}
\left(1+ \frac{1}{{C}_{\tilde{s}}}\right)\cE_+(u,u)
&\geq& \mu^+([s_\ast,\tilde{s}])
\int_{\R^N}\left(\; \int_{\bS^{N-1}}
\left|\theta\cdot\frac{\xi}{|\xi|}\right|^{2s_\ast} \,\tilde\sigma_{\tilde{s}}(d\theta)\right)
|\xi|^{2s_\ast}|\hat{u}(\xi)|^2\, d\xi \\
&=&\mu^+([s_\ast,\tilde{s}])
\int_{\R^N} M_{s_\ast,\tilde\sigma_{\tilde{s}}}
\left(\frac{\xi}{|\xi|}\right)|\xi|^{2s_\ast}|\hat{u}(\xi)|^2\, d\xi \\
&\ge& \mu^+([s_\ast,\tilde{s}])\tilde{\lambda}
\int_{\R^N}|\xi|^{2s_\ast}|\hat{u}(\xi)|^2\, d\xi\\
&=& \mu^+([s_\ast,\tilde{s}])\tilde{\lambda}\,\cE_{s_\ast}(u,u).
\end{eqnarray*}
This  establishes~\eqref{sdfgy8932r576ygh-765432qfghokijhgf}.

The second part of Lemma~\ref{prop of x omega part1}
follows immediately from Corollary~\ref{cor:sobolev embedding}.
\end{proof}

\begin{remark}
Following the proofs of \cref{prop of x omega part2} and \cref{prop of x omega part1} the following adjustments also hold for general open sets of~$\R^N$:
\begin{enumerate}
\item Let $\mu^+$ be a nonnegative Borel measure on~$[0,+\infty)$, $t\geq s_*$ be such that~$\mu^+([s_*,t])>0$. Then~$\cX(\Omega)\subset \tilde{X}^{s_*}(\Omega)$ for any open set $\Omega\subset \R^N$ and it holds
$$
\cE_{2(s_*)_1,s_*}(u,u)\leq \frac{1+\mu^+([s_*,t])}{\mu^+([s_*,t])M_{s_*,\tilde{\sigma_t}}(\tilde{e}_t)}\Big(\|u\|_{L^2(\R^N)}^2+\cE_+(u,u)\Big)\quad\text{for all $u\in \cX(\Omega)$.}
$$
\item Under the assumptions of \cref{prop of x omega part1} (but with $\Omega$ no necessarily bounded), it holds
$$
\cE_{s_*}(u,u)\leq \frac{1+\mu^+([s_*,\tilde{s}])}{\mu^+([s_*,\tilde{s}])\tilde{\lambda}}\Big(\|u\|_{L^2(\R^N)}^2+\cE_{+}(u,u)\Big)\quad\text{for all $u\in \cX^(\Omega)$.}
$$
In particular, $\cX(\Omega)\subset \cH^{s_*}_0(\Omega)$ and the embeddings as stated in points~$1$ and~$2$ of
Lemma~\ref{prop of x omega part1} hold.
\end{enumerate}
\end{remark}

We will now deal with the problem of ``reabsorbing'' the energy with the ``wrong'' sign given by~$\cE_-$ into the energy with the positive sign
given by~$\cE_+$. The precise statement is as follows:

\begin{prop}\label{prop of x omega part3}
Let~$\Omega\subset\R^N$ be open and bounded and
let~$\mu$ satisfy~\eqref{mu-assumption1} and~\eqref{mu-assumption2}.
Let~$s_\ast$ be as in~\eqref{mu-assumption1} and~\eqref{mu-assumption2} and let~$\tilde{s}\ge s_\ast$ be such that~$\mu^+([s_\ast,\tilde{s}])>0$ and
\begin{equation*}
\inf_{e\in\bS^{N-1}}M_{s_\ast, \tilde{\sigma}_{\tilde{s} }}(e)>0,
\end{equation*}
where~$\tilde{\sigma}_{\tilde{s}}$ is the probability measure given in~\eqref{measure ast}.

Then, $\cE_-$ is well-defined and there exists~$C>0$, depending on~$\mu$, $s_\ast$, $\tilde{s}$
and~$\lambda$ (as given in~\eqref{simple ellipticity}), such that,
for all~$u\in \cX(\Omega)$,
$$
\cE_-(u,u)\leq C\gamma \cE_+(u,u), 
$$ where~$\gamma$ is as in~\eqref{mu-assumption2}.
\end{prop}

\begin{proof}
First note that if~$\mu^-([0,+\infty))=0$, then nothing needs to be shown. Hence, from now on we assume that~$\mu^-([0,+\infty))>0$.

We observe that, thanks to \cref{relation to hs} and the assumption
in~\eqref{simple ellipticity},
\begin{equation}\label{vcn832432othjkshgerk}
\cE_{2s_1,s}(u,u)\leq (M_{s,\sigma_s}(e_s))^{-1}\cE_{s}(u,u)\leq \frac{1}{\lambda} \cE_s(u,u).
\end{equation}

Now we would like to exploit Proposition~\ref{prop:spaces ordered}
with~$\sigma_s(d\theta):=({\mathcal{H}}^{N-1}(\bS^{N-1}))^{-1} d\cH^{N-1}_{\theta}$ for all~$s\in[0,\tilde{s}]$. 
With this choice,
the conditions in~\eqref{definition measure}
and~\eqref{assumption basis} are satisfied. Moreover, for all~$e\in\bS^{N-1}$,
\begin{eqnarray*}
\frac1{{\mathcal{H}}^{N-1}(\bS^{N-1})}\int_{\bS^{N-1}}|e\cdot\theta|^{2s}\,d\cH^{N-1}_{\theta}
&=&\frac1{{\mathcal{H}}^{N-1}(\bS^{N-1})}\int_{\bS^{N-1}}|\theta_1|^{2s}\,d\cH^{N-1}_{\theta}
\\& \ge
&\frac1{{\mathcal{H}}^{N-1}(\bS^{N-1})}\int_{\bS^{N-1}}
|\theta_1|^{2\tilde{s}}\,d\cH^{N-1}_{\theta},
\end{eqnarray*}
and so~\eqref{simple ellipticity} is also fulfilled in~$[0,\tilde{s}]$ with
$$\lambda:=\frac1{{\mathcal{H}}^{N-1}(\bS^{N-1})}\int_{\bS^{N-1}}
|\theta_1|^{2\tilde{s}}\,d\cH^{N-1}_{\theta}.$$

Furthermore, in this setting, we have that, for any~$0<t\le s\le \tilde{s}$ and for any~$e\in\bS^{N-1}$,
\begin{eqnarray*} &&{\mathcal{H}}^{N-1}(\bS^{N-1})
M_{t,\sigma_t}(e)=\int_{\bS^{N-1}}|e\cdot\theta|^{2t}\,d\cH^{N-1}_{\theta}=
\int_{\bS^{N-1}}|\theta_1|^{2t}\,d\cH^{N-1}_{\theta}\le 1
\\&&\qquad\le\frac{\displaystyle
\int_{\bS^{N-1}}|\theta_1|^{2s}\,d\cH^{N-1}_{\theta}}{\displaystyle
\int_{\bS^{N-1}}|\theta_1|^{2\tilde{s}}\,d\cH^{N-1}_{\theta}}
=\frac{\displaystyle
{\mathcal{H}}^{N-1}(\bS^{N-1})}{\displaystyle
\int_{\bS^{N-1}}|\theta_1|^{2\tilde{s}}\,d\cH^{N-1}_{\theta}}
M_{s,\sigma_s}(e),
\end{eqnarray*} and this entails that~\eqref{suppcontster897654}
is satisfied with
$$\Lambda:=\left(\;\int_{\bS^{N-1}}|\theta_1|^{2\tilde{s}}\,d\cH^{N-1}_{\theta}\right)^{-1}.$$

As a result, we are in the position of employing Proposition~\ref{prop:spaces ordered} in this case and we obtain that, for all~$s\in[0,\tilde{s}]$,
\begin{eqnarray*}
\cE_s(u,u)\le 
{\mathcal{H}}^{N-1}(\bS^{N-1}) \left(\;\int_{\bS^{N-1}}
|\theta_1|^{2\tilde{s}}\,d\cH^{N-1}_{\theta}\right)^{-2}\big(1+\overline{C}\big)\cE_{\tilde{s}}(u,u),
\end{eqnarray*}
where (recall~\eqref{explicitconstant of poincare}) $$
\overline{C}:=\sup_{s\in[0,\tilde{s}]}2^{4s_1+1}\binom{2s_1}{s_1}^{-2}(\diam(\Omega))^{-2s}. $$
We point out that~$\overline{C}$ depends only on~$\tilde{s}$.

Plugging this information into~\eqref{vcn832432othjkshgerk} we find that, for all~$s\in[0,\tilde{s}]$,
$$
\cE_{2s_1,s}(u,u)\leq \frac{(1+\overline{C}){\mathcal{H}}^{N-1}(\bS^{N-1}) }{\lambda} \left(\;\int_{\bS^{N-1}}
|\theta_1|^{2\tilde{s}}\,d\cH^{N-1}_{\theta}\right)^{-2}
\cE_{\tilde{s}}(u,u).$$
As a result,
\begin{equation}\label{vbcnxiory3826493658gfsdjk}\begin{split}&
\cE_-(u,u)=\int_{[0,s_{\ast})}\cE_{2s_1,s}(u,u)\,\mu^-(ds)
\\&\qquad\le 
\frac{(1+\overline{C}){\mathcal{H}}^{N-1}(\bS^{N-1}) }{\lambda} \left(\;\int_{\bS^{N-1}}
|\theta_1|^{2\tilde{s}}\,d\cH^{N-1}_{\theta}\right)^{-2} \mu^-([0,s_{\ast}))
\cE_{\tilde{s}}(u,u).
\end{split}\end{equation}

Moreover, by Lemma~\ref{prop of x omega part1} we have that
\begin{equation*}
\tilde{C}\cE_+(u,u)\geq \cE_{\tilde{s}}(u,u).
\end{equation*}
This, combined with~\eqref{vbcnxiory3826493658gfsdjk} and~\eqref{mu-assumption2}, gives that
\begin{eqnarray*}
\cE_-(u,u)&\le&  
\frac{(1+\overline{C}){\mathcal{H}}^{N-1}(\bS^{N-1}) }{\lambda} \left(\;\int_{\bS^{N-1}}
|\theta_1|^{2\tilde{s}}\,d\cH^{N-1}_{\theta}\right)^{-2}
\gamma \mu^+([s_{\ast},+\infty))
\cE_{\tilde{s}}(u,u)\\&\le& 
\frac{(1+\overline{C}){\tilde{C}\mathcal{H}}^{N-1}(\bS^{N-1}) }{\lambda} \left(\;\int_{\bS^{N-1}}
|\theta_1|^{2\tilde{s}}\,d\cH^{N-1}_{\theta}\right)^{-2}
\gamma \,\cE_+(u,u),\end{eqnarray*}
which leads to the desired result.
\end{proof}

\begin{remark}\label{dirichlet prob}
Though we aim at studying the nonlinear problem in the forthcoming section, let us remark that a combination of Propositions~\ref{prop of x omega part2} and~\ref{prop of x omega part3} immediately gives the following statement.

Let the assumptions in Proposition~\ref{prop of x omega part3} be satisfied and let $C$ be given as stated there. If either~$\mu^-([0,+\infty))=0$ or~$\gamma$ in~\eqref{mu-assumption2}
is sufficiently small, then for any~$f\in L^2(\Omega)$ there exists a unique~$u\in \cX(\Omega)$ such that
\begin{eqnarray*}\left\{\begin{aligned}
		Lu&=f&&\text{in~$\Omega$,}\\
		u&=0 &&\text{in~$\R^N\setminus \Omega$,}
	\end{aligned}\right.\end{eqnarray*}
in the sense that
$$
\cE(u,\phi)=\int_{\Omega}f\phi\, dx\qquad\text{for all $\phi \in \cX(\Omega)$.}
$$
\end{remark}

\section{The nonlinear problem and its variational setting}\label{sec:67}

In this section we deal with the application of the theory on superposition operators of mixed order developed so far to the study of nonlinear problems of the type~\eqref{prob-gen}.

We recall the definitions in~\eqref{65yt4revc987654gvreg4polkmn}
and we set
$$ \cE(u,u):=\cE_+(u,u)-\cE_-(u,u).$$
We have that:

\begin{lemma}\label{bfndi3o2r73lemma}
Let~$\Omega\subset\R^N$ be open and bounded and
let~$\mu$ satisfy~\eqref{mu-assumption1} and~\eqref{mu-assumption2}. Suppose that~\eqref{Mass} and~\eqref{final assumption} are fulfilled.

Then, there exists~$\gamma_0>0$, depending on~$\mu$, $s_{\#}$, $\tilde{s}$ and~$\lambda$ (as given in~\eqref{simple ellipticity}), such that for all~$\gamma\in[0,\gamma_0]$ there exists~$c_1>0$, depending on~$\gamma$, such that
\begin{equation}\label{embedding-mp2}
\cE(u,u)\geq c_1\|u\|_{L^2(\Omega)}^2
\end{equation}
and
\begin{equation}\label{embedding-mp} \cE(u,u)\geq
c_1\|u\|_{L^{2^{\ast}}(\Omega)}^2,
\end{equation} where~$2^\ast$ is as in~\eqref{2 sharp}.

In particular, $\cE$ defines a scalar product on~$\cX(\Omega)$.

Moreover,
\begin{equation}\label{bvn86801qsx5tyh9o}
{\mbox{the embedding~$\cX(\Omega)\hookrightarrow L^q(\Omega)$ is compact for any~$q\in[1,2^{\ast})$.}}\end{equation}
\end{lemma}

\begin{proof}
By Proposition~\ref{prop of x omega part3} we have that, for all~$u\in \cX(\Omega)$,
\begin{equation}\label{vbncmxtur84u3ogfebwjugfo765}
\cE(u,u)=\cE_+(u,u)-\cE_-(u,u)\geq (1-C\gamma)\cE_+(u,u),
\end{equation}
where~$C>0$ depends on~$\mu$, $s_{\#}$, $\tilde{s}$ and~$\lambda$.

Thus, from formula~\eqref{poincare gen} in Proposition~\ref{prop of x omega part2}, if~$\gamma$ is sufficiently small,
\begin{equation*}
\cE(u,u)\geq (1-C\gamma){C}_{\tilde{s}}\mu^+([s_\ast,\tilde{s}])\|u\|_{L^2(\Omega)}^2,
\end{equation*} which gives~\eqref{embedding-mp2}. 

Moreover, thanks to~\eqref{final assumption} we can exploit Lemma~\ref{prop of x omega part1} with~$s_{\#}$ in place of~$s_\ast$. In this way, we see that~$\cX(\Omega)\subset L^{ 2^{\ast} }(\R^N)$. This fact and~\eqref{vbncmxtur84u3ogfebwjugfo765}
entail~\eqref{embedding-mp}.

We now check the compact embedding in~\eqref{bvn86801qsx5tyh9o}.
We begin with the additional assumption that~$\Omega$ has a smooth boundary. Since~$\Omega$ is a bounded open set, we have that~$\cX(\Omega)\subset \cH^{s_{\#}}_0(\Omega)$ by Lemma~\ref{prop of x omega part1}. This and~\cite[Remark~1 on page~233]{T10}
(see also~\cite[Corollary~7.2]{MR2944369} when~$s_{\#}\in(0,1)$)
give the desired compact embedding if~$s_{\#}\in\left(0,\frac{N}2\right)$.

If instead~$s_{\#}\ge\frac{N}2$ 
and~$s_{\#}-\frac{N}{2}\notin\N_0$, then
we know by Lemma~\ref{prop of x omega part1} that~$\cX(\Omega)\subset C^{s_{\#}-\frac{N}{2}}(\R^N)$, and therefore the compact embedding follows from the Ascoli-Arzel\`a Theorem.

Finally, if~$s_{\#}\ge\frac{N}2$ 
and~$s_{\#}-\frac{N}{2}\in\N_0$, we consider~$\overline{s}<s_{\#}$
such that~$\overline{s}-\frac{N}{2}\notin\N_0$ and we notice that~$\cH^{s_{\#}}_0(\Omega)\subset \cH^{\overline{s}}_0(\Omega)$. In this way, from the previous cases it follows that the compact embedding holds true for~$\overline{s}$, and consequently for~$s_{\#}$.

If instead~$\Omega$ is an arbitrary open bounded set, we may consider a bounded open set~$\Omega'\subset \R^N$ with smooth boundary and~$\Omega\subset \Omega'$. Then, by definition we have that~$\cX(\Omega)\subset \cX(\Omega')$ and thus the general case follows from the special case.
\end{proof}

Throughout the rest of this section, in light of Lemma~\ref{bfndi3o2r73lemma}, we assume that the parameter~$\gamma$ in assumption~\eqref{mu-assumption2} is sufficiently small, possibly in dependance of~$\mu$, $s_{\#}$, $\tilde{s}$ and~$\lambda$.

\begin{defi}
	A function~$u\in \cX(\Omega)$ is called a solution of~\eqref{prob-gen} if, for all~$v\in \cX(\Omega)$,
	\begin{equation*}
		\cE(u,v)=\int_{\Omega} f(x,u)v\, dx.
	\end{equation*}
\end{defi}

We emphasize that in this definition~$\Omega$ can be an arbitrary open subset of~$\R^N$ and~$f$ needs to be such that the integration on the right-hand side is finite for~$u$, $v\in \cX(\Omega)$.

We also stress that this definition is consistent with Proposition~\ref{prop:bilinear}.

\subsection{Mountain pass structure and proof of Theorem~\texorpdfstring{\ref{mountain pass solution}}{1.2}}\label{mp-solutions proof}

In this section we follow the framework introduced in~\cite{SV12} to find nontrivial solutions of~\eqref{prob-gen}.

Let~$J:\cX(\Omega)\to \R$ be the functional given by
$$
J(u):=\frac12 \cE(u,u)-\int_{\Omega}F(x,u(x))\,dx.
$$
A simple computation shows that~$J$ is a~$C^1$ functional and if~$u\in\cX(\Omega)$ is a critical point of~$J$ then, for all~$v\in \cX(\Omega)$,
$$
0=J'(u)v=\cE(u,v)-\int_{\Omega}f(x,u(x))v(x)\,dx.
$$
In particular, critical points of~$J$ are solutions of~\eqref{prob-gen}.

In light of this observation, we aim at employing the Mountain Pass Theorem to find a critical point of~$J$. To this end, in the forthcoming
Propositions~\ref{mp step1} and~\ref{mp step2} we check that~$J$ has  the mountain pass structure. 

\begin{prop}\label{mp step1}
There exist~$\rho$, $\beta>0$ such that for any~$u\in \cX(\Omega)$ with~$\cE(u,u)=\rho$ we have that~$J(u)\geq \beta$.
\end{prop}

\begin{proof}
Thanks to the assumptions on~$f$,
we can use~\cite[Lemma~3]{SV12} and find that
for any~$\epsilon>0$ there exists~$\delta>0$ such that, for a.e.~$x\in \Omega$ and~$t\in \R$,
\begin{equation*}
|F(x,t)|\leq \eps |t|^2+\delta|t|^q,
\end{equation*} 
where~$q\in(2,2^{\ast})$ is as in~\eqref{nonlinearprob-assumptions}.

As a consequence, for any~$u\in \cX(\Omega)$ and~$\eps>0$, we have that
\begin{align*}
J(u)&\geq \frac12\cE(u,u)-\eps\|u\|_{L^2(\Omega)}^2-\delta \|u\|_{L^q(\Omega)}^q\\
&\geq \frac12\cE(u,u)-\eps|\Omega|^{\frac{2^{\ast}-2}{2^{\ast}}}\|u\|_{L^{2^{\ast}}(\Omega)}^2-\delta |\Omega|^{\frac{2^{\ast}-q}{2^{\ast}}}\|u\|_{L^{2^{\ast}}(\Omega)}^q.
\end{align*}
Therefore, exploiting~\eqref{embedding-mp},
\begin{align*}
J(u)&\geq \left(\frac12-\eps c_1|\Omega|^{\frac{2^{\ast}-2}{2^{\ast}}}\right)\cE(u,u)-\delta |\Omega|^{\frac{2^{\ast}-q}{2^{\ast}}}c_1^{\frac{q}{2}}\cE(u,u)^{\frac{q}{2}}.
\end{align*}

Now we choose
$$\eps:= \frac{1}{4}\left(c_1|\Omega|^{\frac{2^{\ast}-2}{2^{\ast}}}\right)^{-1}$$ and we find that
$$
J(u)\geq \frac14 \cE(u,u)\Big(1-\kappa \cE(u,u)^{\frac{q-2}{2}}\Big),
$$
where~$\kappa>0$ is a constant independent of~$u$.

Accordingly,
$$
\inf_{\substack{u\in \cX(\Omega)\\ \cE(u,u)=\rho}}J(u)\geq \frac{1}{4}\rho\Big(1-\kappa \rho^{\frac{q-2}{2}}\Big)
$$
and this quantity is positive for~$\rho$ sufficiently small since~$q>2$. In particular, taking~$\rho:=(2\kappa)^{-\frac{2}{q-2}}$ we conclude that
\begin{equation*}
\inf_{\substack{u\in \cX(\Omega)\\ \cE(u,u)=\rho}}J(u)\geq \frac{1}{8}\rho=:\beta>0.\qedhere
\end{equation*}
\end{proof}

\begin{prop}\label{mp step2}
There exists a nonnegative function~$e\in \cX(\Omega)$ such that~$\cE(e,e)>\rho$ and~$J(e)<\beta$ with~$\rho$ and~$\beta$ given by \cref{mp step1}.
\end{prop}

\begin{proof}
We recall, by~\cite[Lemma~4]{SV12}, that
there exist integrable functions~$m$, $M:\Omega\to (0,+\infty)$ satisfying
$$
F(x,t)\geq m(x)|t|^{\mu}-M(x)\quad\text{for a.e.~$x\in \Omega$ and~$t\in \R$},
$$ where~$\mu>2$ is as in~\eqref{nonlinearprob-assumptions}.

Let~$u\in \cX(\Omega)$ be a function given by assumption~\eqref{nonempty xomega}. Then~$A:=\cE(u,u)<+\infty$ and thus, for any~$t>0$, $$
J(tu)=\frac{t^2}{2}A-\int_{\Omega}F(x,tu(x))\,dx\leq \frac{t^2}{2}A-t^{\mu}\int_{\Omega}m(x)|u(x)|^{\mu}\,dx+\int_{\Omega}M(x)\,dx.
$$
As a result, since~$\mu>2$,
we see that~$J(tu)\to -\infty$ as~$t\to+\infty$, and the claim follows by choosing~$e=:Tu$ with~$T$ sufficiently large.
\end{proof}

The following two propositions show that the Palais-Smale condition is satisfied.

\begin{prop}\label{mp step3}
Let~$c\in \R$ and let~$(u_j)_j\subset \cX(\Omega)$ be a sequence such that
\begin{eqnarray}
&&\lim_{j\to+\infty}J(u_j)= c \nonumber
\\{\mbox{and }}
\quad &&\label{ps2}\lim_{j\to+\infty}
\sup_{\substack{\phi\in \cX(\Omega)\\ \cE(\phi,\phi)=1}}|J'(u_j)\phi|= 0.
\end{eqnarray}

Then, $(u_j)_j$ is bounded in~$\cX(\Omega)$.
\end{prop}

\begin{proof}
The proof is analogous to the one of~\cite[Proposition~11]{SV12}.
\end{proof}

\begin{lemma}\label{mp step4-pre}
Let~$(u_j)_j\subset \cX(\Omega)$ be a bounded sequence satisfying~\eqref{ps2}. 

Then, there exists~$u_{\infty}\in \cX(\Omega)$ such that, up to a subsequence, as~$j\to +\infty$,
\begin{equation}\label{0987iuygbfsrbq23wsedxcfv}\begin{split}
&u_j\weakto u_{\infty}\quad\text{in~$\cX(\Omega)$,}\\
&u_j\to u_{\infty}\quad\text{in~$L^p(\Omega)$ for all~$p\in[1,2^{\ast})$} \\ {\mbox{and }}\qquad
&u_j(x)\to u_{\infty}(x)\quad \text{for almost every~$x\in \Omega$.}
\end{split}\end{equation}
\end{lemma}

\begin{proof}
The weak convergence follows from the boundedness of the sequence~$u_j$ and the fact that~$\cX(\Omega)$ is a Hilbert space, see Proposition~\ref{prop of x omega part2}.

The strong convergence in~$L^p(\Omega)$ for all~$p\in[1,2^{\ast})$ (and therefore the a.e. pointwise convergence) is a consequence
of the weak convergence in~$\cX(\Omega)$ and the compact embedding in~\eqref{bvn86801qsx5tyh9o}.
\end{proof}

\begin{prop}\label{mp step4}
Let~$(u_j)_j\subset \cX(\Omega)$ be a bounded sequence satisfying~\eqref{ps2}. 

Then, there exists~$u_{\infty}\in \cX(\Omega)$ such that, up to a subsequence, 
$$\lim_{j\to+\infty}\cE(u_j-u_{\infty},u_j-u_{\infty})=0.$$
\end{prop}

\begin{proof}
{F}rom Lemma~\ref{mp step4-pre}
we deduce the existence of~$u_{\infty}\in\cX(\Omega)$ such that the convergences
in~\eqref{0987iuygbfsrbq23wsedxcfv} hold true.
Namely, we have that
\begin{equation}\label{vdshg34y5rhe5yury}
\lim_{j\to+\infty} \cE(u_j,\phi)=\cE(u_{\infty},\phi)\quad\text{for all~$\phi\in \cX(\Omega)$,}
\end{equation}
and~$u_j\to u_{\infty}$ in~$L^q(\R^N)$, with~$q\in(2,2^{\ast})$
as in~\eqref{nonlinearprob-assumptions},
and pointwisely a.e. in~$\R^N$ as~$j\to+\infty$.

Also, there exists~$v\in L^q(\R^N)$ such that 
$$
|u_j(x)|\leq v(x)\quad\text{for a.e.~$x\in \R^N$ and all~$j\in \N$},
$$ see~\cite[Theorem~IV.9]{MR697382}.

Gathering these pieces of information and using
the first assumption on the nonlinearity~$f$ in~\eqref{nonlinearprob-assumptions} (and its continuity with respect to the second variable), we obtain from the Dominated Convergence Theorem that
\begin{eqnarray}
\lim_{j\to+\infty}\int_{\Omega}f(x,u_j(x))u_j(x)\,dx&=&\int_{\Omega}f(x,u_\infty(x))u_\infty(x)\,dx\label{cnmzbvtyw4uiy48370} \\
\text{and }\quad \lim_{j\to+\infty}\int_{\Omega}f(x,u_j(x))u_\infty(x)\,dx&=&\int_{\Omega}f(x,u_\infty(x))u_\infty(x)\,dx.\label{cnmzbvtyw4uiy4837}
\end{eqnarray}
As a consequence, if~$u_{\infty}\neq 0$,
we use~\eqref{ps2} with~$\varphi$ replaced by (a renormalization of) $u_\infty$ and we obtain that
$$
\lim_{j\to+\infty}\cE(u_j,u_{\infty})-\int_{\Omega}f(x,u_j(x))u_{\infty}(x)\,dx=\lim_{j\to+\infty} J'(u_j)u_{\infty}= 0.$$
This, together with~\eqref{vdshg34y5rhe5yury} and~\eqref{cnmzbvtyw4uiy4837}, leads to
\begin{equation}\label{vdskjht4u3iy65t8uh}
\cE(u_{\infty},u_{\infty})=\int_{\Omega} f(x,u_{\infty}(x))u_{\infty}(x)\,dx.
\end{equation}

Next note that
\begin{eqnarray*}&&
\lim_{j\to+\infty}\left|\cE(u_j,u_j)-\int_{\Omega}f(x,u_j(x))u_j(x)\,dx\right|=\lim_{j\to+\infty}\left|
J'(u_j)u_j\right|
\leq \lim_{j\to+\infty}\sup_{\substack{\phi\in \cX(\Omega)\\\cE(\phi,\phi)=1}}|J'(u_j)\phi|= 0
.\end{eqnarray*}
{F}rom this and~\eqref{cnmzbvtyw4uiy48370} we conclude that
$$ \lim_{j\to+\infty}
\cE(u_j,u_j)= \int_{\Omega} f(x,u_{\infty}(x))u_{\infty}(x)\,dx.
$$
This and~\eqref{vdskjht4u3iy65t8uh} thus entail that
$$
\lim_{j\to+\infty}\cE(u_j,u_j)=\cE(u_{\infty},u_{\infty}).
$$

As a result, recalling also~\eqref{vdshg34y5rhe5yury},
\begin{eqnarray*}
&&\lim_{j\to+\infty}\cE(u_j-u_{\infty},u_j-u_{\infty})=
\lim_{j\to+\infty}
\cE(u_j,u_j)+\cE(u_{\infty},u_{\infty})-2\cE(u_j,u_\infty)
=0,\end{eqnarray*}
as desired.
\end{proof}

We can now complete the proof of Theorem~\ref{mountain pass solution}.

\begin{proof}[Proof of Theorem~\ref{mountain pass solution}]
By Propositions~\ref{mp step1}, \ref{mp step2}, \ref{mp step3}, and~\ref{mp step4}, we see that the functional~$J$ satisfies the assumptions of the Mountain Pass Theorem and therefore there exist a critical point~$u\in \cX(\Omega)$ of~$J$ with
$$
J(u)\geq\beta>0=J(0)
$$
so that~$u\not\equiv 0$.
\end{proof}

\subsection{Critical growth problem with jumping nonlinearities and proof of Theorem~\texorpdfstring{\ref{critical case}}{1.3}}

In this setting, we look for critical points of the energy functional~$E:\cX(\Omega)\to\R$ given by
\begin{equation}\label{defi:functional}
E(u):=\frac12\cE(u,u)-\frac12\int_{\Omega}\big(a(u^-)^2+b(u^+)^2\big)\,dx-\frac{1}{2^\ast}\int_{\Omega}|u|^{2^\ast}\,dx.
\end{equation}
Critical points of~$E$ are indeed solutions of~\eqref{nonlinearprob}.

The main step of to establish Theorem~\ref{critical case} is analogous to~\cite[Proposition~3.4]{DPSV23} and requires the use of~\cite[Theorem~4.1]{PS23}. We begin with the following:

\begin{prop}\label{linking ps1}
Let~$c>0$. 
Then, every Palais-Smale sequence at the level~$c$, denoted by~$(PS)_c$ sequence, of the functional~$E$ in~\eqref{defi:functional} has a subsequence that converges weakly to a nontrivial critical point of~$E$.
\end{prop}

\begin{proof}
Let~$(u_j)_j\subset \cX(\Omega)$ be a~$(PS)_c$ sequence of~$E$, that is,
\begin{equation}\label{eq1:linking ps1}
\lim_{j\to+\infty}E(u_j)=c\qquad\text{and}\qquad \lim_{j\to+\infty} |E'(u_j)v|=0\quad\text{for all~$v\in \cX(\Omega)$.}
\end{equation}
Thus, choosing~$v:=u_j$, we obtain that 
\begin{align*}
&0=\lim_{j\to+\infty} E'(u_j)u_j=\lim_{j\to+\infty}
\cE(u_j,u_j)-\int_{\Omega} a(u_j^-)^2+b(u_j^+)^2\,dx-\int_{\Omega}|u_j|^{2^{\ast}}\,dx \\
&\qquad=\lim_{j\to+\infty}2E(u_j)+\left(\frac{2}{2^{\ast}}-1\right)\int_{\Omega}|u_j|^{2^{\ast}}\,dx
=2c+\left(\frac{2}{2^{\ast}}-1\right)\lim_{j\to+\infty}\int_{\Omega}|u_j|^{2^{\ast}}\,dx.
\end{align*}
This entails that there exists~$j_0\in \N$ such that
$$
\left(\frac{1}{2}-\frac{1}{2^{\ast}}\right)\int_{\Omega}|u_j|^{2^{\ast}}\,dx\leq c+1\qquad\text{for all~$j\geq j_0$.}
$$

From this and H\"older's inequality we have that, for all~$j\geq j_0$,
\begin{align*}&
\int_{\Omega} a(u_j^-)^2+b(u_j^+)^2\,dx\leq \max\{a,b\}\|u_j\|_{L^2(\Omega)}^2\leq \max\{a,b\}|\Omega|^{\frac{2^{\ast}-2}{2^{\ast}}}\|u_j\|_{L^{2^{\ast}}(\Omega)}^2\\
&\qquad\leq \left(\frac{2\cdot 2^{\ast}(c+1)}{2^{\ast}-2}\right)^{\frac{2}{2^{\ast}}}\max\{a,b\}|\Omega|^{\frac{2^{\ast}-2}{2^{\ast}}}.
\end{align*}
Thus, using also the first formula in~\eqref{eq1:linking ps1}, we can find some~$j_1\geq j_0$ such that, for every~$j\geq j_1$,
\begin{align*}
\cE(u_j,u_j)&=E(u_j)+\frac12\int_{\Omega} a(u_j^-)^2+b(u_j^+)^2\,dx+\frac{1}{2^{\ast}}\int_{\Omega}|u_j|^{2^{\ast}}\,dx\\
&\leq c+1+\frac12\left(\frac{2\cdot 2^{\ast}(c+1)}{2^{\ast}-2}\right)^{\frac{2}{2^{\ast}}}\max\{a,b\}|\Omega|^{\frac{2^{\ast}-2}{2^{\ast}}}+\frac{2(c+1)}{2^{\ast}-2}.
\end{align*}
In particular, $(u_j)_j$ is bounded in~$\cX(\Omega)$
(recall also~\eqref{embedding-mp2} in Lemma~\ref{bfndi3o2r73lemma}), and \cref{mp step4-pre} implies the existence of some~$u_{\infty}\in \cX(\Omega)$ such that, up to a subsequence, as~$j\to +\infty$,
\begin{eqnarray*}&&
u_j\weakto u_{\infty}\quad\text{in~$\cX(\Omega)$,}\\
&&u_j\to u_{\infty}\quad\text{in~$L^p(\Omega)$ for all~$p\in[1,2^{\ast})$,} \\
{\mbox{and }} &&
u_j(x)\to u_{\infty}(x)\quad \text{for almost every~$x\in \Omega$.}
\end{eqnarray*}
Arguing as done to obtain Proposition~\ref{mp step4}, we thus conclude that 
$$
\lim_{j\to+\infty} \cE(u_j-u_{\infty},u_j-u_{\infty})=0
$$
and~$u_j\to u_{\infty}$ in~$\cX(\Omega)$. 
This implies that~$u_{\infty}$ is a critical point of~$E$. 

To show that~$u_{\infty}$ is nontrivial, note that we have~$E(u_{\infty})=c$ and 
\begin{align*}
0=E'(u_{\infty})u_{\infty}=2E(u_{\infty})+\left(\frac{2}{2^{\ast}}-1\right)\int_{\Omega}|u_{\infty}|^{2^{\ast}}\,dx=2c+\left(\frac{2}{2^{\ast}}-1\right)\int_{\Omega}|u_{\infty}|^{2^{\ast}}\,dx,
\end{align*}
which implies that~$u_{\infty}\equiv 0$ is not possible. The claim follows.
\end{proof}

For the reader's convenience we state next~\cite[Theorem~4.1]{PS23}. To do so, let~$N_k$ be the direct sum of the eigenspaces associated with the eigenvalues~$\lambda_{1},\ldots,\lambda_k$.

\begin{thm}[Theorem 4.1, \cite{PS23}]\label{thm:ps}
Let~$F\in C^1(\R)$  and consider the functional~$E:\cX(\Omega)\to\R$ defined as
$$
E(u):=\cE(u,u)-b\int_{\Omega}(u^+)^2\,dx+a\int_{\Omega}(u^-)^2\,dx-\int_{\Omega}F(u)\,dx.
$$
Assume that~$F$ is nonnegative with~$F(0)=0$ and that
$$
\lim_{\|u\|_{X(\Omega)}\to0}\frac{\displaystyle \int_{\Omega}F(u)}{\cE(u,u)}=0.
$$
Assume further that there exists~$c^{\ast}>0$ such that for each~$c\in(0,c^{\ast})$ every~$(PS)_c$-sequence of~$E$ has a subsequence that converges weakly to a nontrivial critical point of~$E$. 

If~$l\in\N$, $l\geq 2$ is such that~$(a,b)\in Q_l$, $b<A_{l-1}(a)$ and that there exists~$e\in \cX(\Omega)\setminus N_{l-1}$ such that
\begin{equation}\label{thm:ps-to show}
\sup_{u\in Q} E(u)<c^{\ast}
\end{equation}
where~$Q=\{v+t e\;:\; v\in N_{l-1},\ t\geq 0\}$, then there exists a nontrivial solution~$u\in \cX(\Omega)$ to the problem
$$
Lu=bu^+-au^-+f(u)\quad\text{in~$\Omega$,}\quad u=0 \quad\text{in~$\R^N\setminus \Omega$,}
$$
where~$f(u)=F'(u)$.
\end{thm}

\begin{proof}[Proof of Theorem~\ref{critical case}]
The claim of Theorem~\ref{critical case} will follow as an application
of Theorem~\ref{thm:ps}. The details are as follows.

Taking~$f(u):=|u|^{2^{\ast}-2}u$ and~$F(u):=\frac{1}{2^{\ast}}|u|^{2^{\ast}}$ we see that~$F$ satisfies all the assumptions of Theorem~\ref{thm:ps}. Moreover, by Proposition~\ref{linking ps1} we can choose any~$c_{\ast}>0$. 

It remains to find~$e\in \cX(\Omega)\setminus N_{l-1}$ such that~\eqref{thm:ps-to show} holds. For this let
$$
E_k:=\{u\in X(\Omega)\;:\; Lu=\lambda_ku\}\quad\text{and thus} \quad N_k=\bigoplus_{i=1}^k E_i.
$$
Note that we have
$$
\lambda_k\|u\|_{L^2(\Omega)}^2\leq \lambda_l\|u\|_{L^2(\Omega)}^2\quad \text{for all~$u\in E_k$, $k=1,\ldots,l$.}
$$
Observe further that, by H\"older's inequality,
$$
\|u\|_{L^2(\Omega)}^{2^{\ast}}\leq |\Omega|^{2^{\ast}-2}\|u\|_{L^{2^{\ast}}(\Omega)}^{2^{\ast}}.
$$
Consequently, we have for~$u\in E_k$, with~$k=1,\ldots,l$, that
\begin{align*}
E(u)&=\frac12\cE(u,u)-\frac12\int_{\Omega}a(u^-)^2+b(u^+)^2\,dx-\frac{1}{2^{\ast}}\int_{\Omega}|u|^{2^{\ast}}\,dx\\
&\leq \frac{\lambda_l-\min\{a,b\}}{2} \int_{\Omega}|u|^2\,dx-|\Omega|^{2-2^{\ast}}\frac{1}{2^{\ast}}\|u\|_{L^2(\Omega)}^{2^{\ast}}.
\end{align*}
As in the proof of~\cite[Theorem~1.1]{DPSV23} (see in particular from formula~(4.1) there till the end of the proof), we obtain that
$$
E(u)\leq \left(\frac12 - \frac{1}{2^{\ast}}\right)|\Omega|\Big(\lambda_l-\min\{a,b\}\Big)^{\frac{2^{\ast}}{2^{\ast}-2}}=:c_{\ast},
$$
as desired.
\end{proof}

\appendix

\section{Proof of formula~\texorpdfstring{\eqref{agaapa0}}{(21)}}\label{app:agaapa0}

We recall from~\eqref{defdelta} that
$$\delta_m f(0,t)=\sum_{k=-m}^m(-1)^k\binom{2m}{m-k}f(kt)
=\sum_{k=-m}^m(-1)^k\binom{2m}{m-k}e^{ikt}.$$
Now we change index~$h:=k+m$ and we obtain that
\begin{eqnarray*}&&
\delta_m f(0,t)=\sum_{h=0}^{2m}(-1)^{h-m}\binom{2m}{2m-h}e^{i(h-m)t}
=(-1)^m e^{-imt} \sum_{h=0}^{2m}(-1)^{h}\binom{2m}{2m-h}e^{iht}.
\end{eqnarray*}
Hence, using the Binomial Theorem we conclude that
\begin{eqnarray*}&&
\delta_m f(0,t)
=(-1)^m e^{-imt} (e^{it}-1)^{2m}
=(-1)^m e^{-imt} (e^{it}-1)^{m}(e^{it}-1)^{m}\\
&&\qquad =(1-e^{-it})^{m}(1-e^{it})^{m}
=\Big( (1-e^{-it})(1-e^{it})\Big)^{m}
=\big( 2-e^{it}-e^{-it}\big)^{m}\\
&&\qquad = 2^m (1- \cos t)^m,
\end{eqnarray*}
as desired.

\section*{Acknowledgements}
SD and EV are members of the Australian Mathematical Society (AustMS). 
This work has been supported by the Australian Future Fellowship
FT230100333 and by the Australian Laureate Fellowship FL190100081.

\bibliographystyle{alpha}

\begin{thebibliography}{DPLSV25b}

\bibitem[ABF19]{zbMATH07109934}
P.~R.~S. Antunes, D.~Buoso, and P.~Freitas.
\newblock On the behavior of clamped plates under large compression.
\newblock {\em SIAM J. Appl. Math.}, 79(5):1872--1891, 2019.

\bibitem[ADF{\etalchar{+}}19]{zbMATH07024011}
Nicola Abatangelo, Serena Dipierro, Mouhamed~Moustapha Fall, Sven Jarohs, and
  Alberto Salda{\~n}a.
\newblock Positive powers of the {Laplacian} in the half-space under
  {Dirichlet} boundary conditions.
\newblock {\em Discrete Contin. Dyn. Syst.}, 39(3):1205--1235, 2019.

\bibitem[AdTEJ20]{MR4149690}
Natha\"{e}l Alibaud, F\'{e}lix del Teso, J{\o}rgen Endal, and Espen~R.
  Jakobsen.
\newblock The {L}iouville theorem and linear operators satisfying the maximum
  principle.
\newblock {\em J. Math. Pures Appl. (9)}, 142:229--242, 2020.

\bibitem[AJS18]{zbMATH06994026}
Nicola Abatangelo, Sven Jarohs, and Alberto Salda{\~n}a.
\newblock Integral representation of solutions to higher-order fractional
  {Dirichlet} problems on balls.
\newblock {\em Commun. Contemp. Math.}, 20(8):36, 2018.
\newblock Id/No 1850002.

\bibitem[AJSn18]{AJS18a}
Nicola Abatangelo, Sven Jarohs, and Alberto Salda\~{n}a.
\newblock Positive powers of the {L}aplacian: from hypersingular integrals to
  boundary value problems.
\newblock {\em Commun. Pure Appl. Anal.}, 17(3):899--922, 2018.

\bibitem[AL17]{zbMATH06756312}
Jos{\'e}~M. Arrieta and Pier~Domenico Lamberti.
\newblock Higher order elliptic operators on variable domains. {Stability}
  results and boundary oscillations for intermediate problems.
\newblock {\em J. Differ. Equations}, 263(7):4222--4266, 2017.

\bibitem[BCCI12]{MR2911421}
Guy Barles, Emmanuel Chasseigne, Adina Ciomaga, and Cyril Imbert.
\newblock Lipschitz regularity of solutions for mixed integro-differential
  equations.
\newblock {\em J. Differential Equations}, 252(11):6012--6060, 2012.

\bibitem[BDVV23]{MR4391102}
Stefano Biagi, Serena Dipierro, Enrico Valdinoci, and Eugenio Vecchi.
\newblock A {H}ong-{K}rahn-{S}zeg\"{o} inequality for mixed local and nonlocal
  operators.
\newblock {\em Math. Eng.}, 5(1):Paper No. 014, 25, 2023.

\bibitem[BJK10]{MR2653895}
Imran~H. Biswas, Espen~R. Jakobsen, and Kenneth~H. Karlsen.
\newblock Viscosity solutions for a system of integro-{PDE}s and connections to
  optimal switching and control of jump-diffusion processes.
\newblock {\em Appl. Math. Optim.}, 62(1):47--80, 2010.

\bibitem[BK05]{MR2095633}
Richard~F. Bass and Moritz Kassmann.
\newblock Harnack inequalities for non-local operators of variable order.
\newblock {\em Trans. Amer. Math. Soc.}, 357(2):837--850, 2005.

\bibitem[BKS22]{MR4490672}
David Berger, Franziska K\"{u}hn, and Ren\'{e}~L. Schilling.
\newblock L\'{e}vy processes, generalized moments and uniform integrability.
\newblock {\em Probab. Math. Statist.}, 42(1):109--131, 2022.

\bibitem[BM91]{M89}
G.~Bourdaud and Y.~Meyer.
\newblock Fonctions qui op\`erent sur les espaces de {S}obolev.
\newblock {\em J. Funct. Anal.}, 97(2):351--360, 1991.

\bibitem[Bre83]{MR697382}
Ha\"{\i}m Brezis.
\newblock {\em Analyse fonctionnelle}.
\newblock Collection Math\'{e}matiques Appliqu\'{e}es pour la Ma\^{i}trise.
  [Collection of Applied Mathematics for the Master's Degree]. Masson, Paris,
  1983.
\newblock Th\'{e}orie et applications. [Theory and applications].

\bibitem[BVDV21]{MR4313576}
Stefano Biagi, Eugenio Vecchi, Serena Dipierro, and Enrico Valdinoci.
\newblock Semilinear elliptic equations involving mixed local and nonlocal
  operators.
\newblock {\em Proc. Roy. Soc. Edinburgh Sect. A}, 151(5):1611--1641, 2021.

\bibitem[CDF{\etalchar{+}}22]{MR4500878}
Giuseppe~Maria Coclite, Serena Dipierro, Giuseppe Fanizza, Francesco Maddalena,
  and Enrico Valdinoci.
\newblock Dispersive effects in a scalar nonlocal wave equation inspired by
  peridynamics.
\newblock {\em Nonlinearity}, 35(11):5664--5713, 2022.

\bibitem[CG92]{MR1181350}
Mabel Cuesta and Jean-Pierre Gossez.
\newblock A variational approach to nonresonance with respect to the
  {F}u\v{c}ik spectrum.
\newblock {\em Nonlinear Anal.}, 19(5):487--500, 1992.

\bibitem[CS96]{MR1325915}
G.~M. Constantine and T.~H. Savits.
\newblock A multivariate {F}a\`a di {B}runo formula with applications.
\newblock {\em Trans. Amer. Math. Soc.}, 348(2):503--520, 1996.

\bibitem[CS16]{MR3485125}
Xavier Cabr\'{e} and Joaquim Serra.
\newblock An extension problem for sums of fractional {L}aplacians and 1-{D}
  symmetry of phase transitions.
\newblock {\em Nonlinear Anal.}, 137:246--265, 2016.

\bibitem[Dan93]{MR1215262}
E.~N. Dancer.
\newblock Generic domain dependence for nonsmooth equations and the open set
  problem for jumping nonlinearities.
\newblock {\em Topol. Methods Nonlinear Anal.}, 1(1):139--150, 1993.

\bibitem[Dan99]{MR1725568}
E.~N. Dancer.
\newblock Remarks on jumping nonlinearities.
\newblock In {\em Topics in nonlinear analysis}, volume~35 of {\em Progr.
  Nonlinear Differential Equations Appl.}, pages 101--116. Birkh\"{a}user,
  Basel, 1999.

\bibitem[Dan77]{MR499709}
E.~N. Dancer.
\newblock On the {D}irichlet problem for weakly non-linear elliptic partial
  differential equations.
\newblock {\em Proc. Roy. Soc. Edinburgh Sect. A}, 76(4):283--300, 1976/77.

\bibitem[DD94]{MR1303035}
E.~N. Dancer and Yi~Hong Du.
\newblock Competing species equations with diffusion, large interactions, and
  jumping nonlinearities.
\newblock {\em J. Differential Equations}, 114(2):434--475, 1994.

\bibitem[DK20]{DK20}
Bart{\l}omiej Dyda and Moritz Kassmann.
\newblock Regularity estimates for elliptic nonlocal operators.
\newblock {\em Anal. PDE}, 13(2):317--370, 2020.

\bibitem[dlLV09]{MR2542727}
Rafael de~la Llave and Enrico Valdinoci.
\newblock A generalization of {A}ubry-{M}ather theory to partial differential
  equations and pseudo-differential equations.
\newblock {\em Ann. Inst. H. Poincar\'{e} C Anal. Non Lin\'{e}aire},
  26(4):1309--1344, 2009.

\bibitem[DNPV12]{MR2944369}
Eleonora Di~Nezza, Giampiero Palatucci, and Enrico Valdinoci.
\newblock Hitchhiker's guide to the fractional {S}obolev spaces.
\newblock {\em Bull. Sci. Math.}, 136(5):521--573, 2012.

\bibitem[DPLSV25a]{MR4821750}
Serena Dipierro, Edoardo Proietti~Lippi, Caterina Sportelli, and Enrico
  Valdinoci.
\newblock A general theory for the {$(s, p)$}-superposition of nonlinear
  fractional operators.
\newblock {\em Nonlinear Anal. Real World Appl.}, 82:Paper No. 104251, 24,
  2025.

\bibitem[DPLSV25b]{MR4882759}
Serena Dipierro, Edoardo Proietti~Lippi, Caterina Sportelli, and Enrico
  Valdinoci.
\newblock The {N}eumann condition for the superposition of fractional
  {L}aplacians.
\newblock {\em ESAIM Control Optim. Calc. Var.}, 31:Paper No. 25, 45, 2025.

\bibitem[DPLSV25c]{DPLSV-MN}
Serena Dipierro, Edoardo Proietti~Lippi, Caterina Sportelli, and Enrico
  Valdinoci.
\newblock Some nonlinear problems for the superposition of fractional operators
  with {N}eumann boundary conditions.
\newblock {\em Math. Nachr.}, 2025.

\bibitem[DPLV23]{MR4651677}
Serena Dipierro, Edoardo Proietti~Lippi, and Enrico Valdinoci.
\newblock ({N}on)local logistic equations with {N}eumann conditions.
\newblock {\em Ann. Inst. H. Poincar\'{e} C Anal. Non Lin\'{e}aire},
  40(5):1093--1166, 2023.

\bibitem[DPSV24]{DPSV23}
Serena Dipierro, Kanishka Perera, Caterina Sportelli, and Enrico Valdinoci.
\newblock An existence theory for superposition operators of mixed order
  subject to jumping nonlinearities.
\newblock {\em Nonlinearity}, 37(5):27, 2024.
\newblock Id/No 055018.

\bibitem[DPSV25]{MR4942302}
Serena Dipierro, Kanishka Perera, Caterina Sportelli, and Enrico Valdinoci.
\newblock An existence theory for nonlinear superposition operators of mixed
  fractional order.
\newblock {\em Commun. Contemp. Math.}, 27(8):Paper No. 2550005, 2025.

\bibitem[dSOR22]{MR4477810}
Bruna~C. dos Santos, Sergio~M. Oliva, and Julio~D. Rossi.
\newblock A local/nonlocal diffusion model.
\newblock {\em Appl. Anal.}, 101(15):5213--5246, 2022.

\bibitem[DTV25]{JACK}
Serena Dipierro, Jack Thompson, and Enrico Valdinoci.
\newblock Some nonlocal formulas inspired by an identity of {J}ames {S}imons.
\newblock {\em M\"unster J. Math.}, 2025.

\bibitem[DV21]{MR4249816}
Serena Dipierro and Enrico Valdinoci.
\newblock Description of an ecological niche for a mixed local/nonlocal
  dispersal: an evolution equation and a new {N}eumann condition arising from
  the superposition of {B}rownian and {L}\'{e}vy processes.
\newblock {\em Phys. A}, 575:Paper No. 126052, 20, 2021.

\bibitem[EG15]{MR3409135}
Lawrence~C. Evans and Ronald~F. Gariepy.
\newblock {\em Measure theory and fine properties of functions}.
\newblock Textbooks in Mathematics. CRC Press, Boca Raton, FL, revised edition,
  2015.

\bibitem[FKV15]{FKV15}
Matthieu Felsinger, Moritz Kassmann, and Paul Voigt.
\newblock The {D}irichlet problem for nonlocal operators.
\newblock {\em Math. Z.}, 279(3-4):779--809, 2015.

\bibitem[GGS10]{zbMATH05712793}
Filippo Gazzola, Hans-Christoph Grunau, and Guido Sweers.
\newblock {\em Polyharmonic boundary value problems. {Positivity} preserving
  and nonlinear higher order elliptic equations in bounded domains}, volume
  1991 of {\em Lect. Notes Math.}
\newblock Berlin: Springer, 2010.

\bibitem[GH23]{GH23}
Florian Grube and Thorben Hensiek.
\newblock Maximum principle for stable operators.
\newblock {\em Math. Nachr.}, 296(12):5684--5702, 2023.

\bibitem[GK22]{MR4469224}
Prashanta Garain and Juha Kinnunen.
\newblock On the regularity theory for mixed local and nonlocal quasilinear
  elliptic equations.
\newblock {\em Trans. Amer. Math. Soc.}, 375(8):5393--5423, 2022.

\bibitem[GK25]{MR4884561}
Tomasz Grzywny and Mateusz Kwa\'{s}nicki.
\newblock Liouville's theorems for {L}\'{e}vy operators.
\newblock {\em Math. Ann.}, 391(4):5857--5910, 2025.

\bibitem[Gri11]{zbMATH05960425}
Pierre Grisvard.
\newblock {\em Elliptic problems in nonsmooth domains}, volume~69 of {\em
  Class. Appl. Math.}
\newblock Philadelphia, PA: Society for Industrial {and} Applied Mathematics
  (SIAM), reprint of the 1985 hardback ed. edition, 2011.

\bibitem[JW16]{JW16}
Sven Jarohs and Tobias Weth.
\newblock Symmetry via antisymmetric maximum principles in nonlocal problems of
  variable order.
\newblock {\em Ann. Mat. Pura Appl. (4)}, 195(1):273--291, 2016.

\bibitem[JW19]{JW19}
Sven Jarohs and Tobias Weth.
\newblock On the strong maximum principle for nonlocal operators.
\newblock {\em Math. Z.}, 293(1-2):81--111, 2019.

\bibitem[JW20]{JW20}
Sven Jarohs and Tobias Weth.
\newblock Local compactness and nonvanishing for weakly singular nonlocal
  quadratic forms.
\newblock {\em Nonlinear Anal.}, 193:111431, 15, 2020.

\bibitem[KLV93]{zbMATH00223815}
Bernhard Kawohl, Howard~A. Levine, and Waldemar Velte.
\newblock Buckling eigenvalues for a clamped plate embedded in an elastic
  medium and related questions.
\newblock {\em SIAM J. Math. Anal.}, 24(2):327--340, 1993.

\bibitem[KP02]{MR1916029}
Steven~G. Krantz and Harold~R. Parks.
\newblock {\em A primer of real analytic functions}.
\newblock Birkh\"{a}user Advanced Texts: Basler Lehrb\"{u}cher. [Birkh\"{a}user
  Advanced Texts: Basel Textbooks]. Birkh\"{a}user Boston, Inc., Boston, MA,
  second edition, 2002.

\bibitem[MN19]{MN19}
Roberta Musina and Alexander~I. Nazarov.
\newblock A note on truncations in fractional {Sobolev} spaces.
\newblock {\em Bull. Math. Sci.}, 9(1):7, 2019.
\newblock Id/No 1950001.

\bibitem[Nar73]{MR346855}
Raghavan Narasimhan.
\newblock {\em Analysis on real and complex manifolds}.
\newblock Advanced Studies in Pure Mathematics, Vol. 1. Masson \& Cie,
  \'{E}diteurs, Paris; North-Holland Publishing Co., Amsterdam-London American
  Elsevier Publishing Co., Inc., New York, second edition, 1973.

\bibitem[Par67]{MR226684}
K.~R. Parthasarathy.
\newblock {\em Probability measures on metric spaces}.
\newblock Probability and Mathematical Statistics, No. 3. Academic Press, Inc.,
  New York-London, 1967.

\bibitem[PS13]{PS13}
Kanishka Perera and Martin Schechter.
\newblock {\em Topics in critical point theory}, volume 198 of {\em Camb.
  Tracts Math.}
\newblock Cambridge: Cambridge University Press, 2013.

\bibitem[PS23]{PS23}
Kanishka Perera and Caterina Sportelli.
\newblock New linking theorems with applications to critical growth elliptic
  problems with jumping nonlinearities.
\newblock {\em J. Differ. Equations}, 349:284--317, 2023.

\bibitem[Rom17]{zbMATH06715608}
Giulio Romani.
\newblock Positivity for fourth-order semilinear problems related to the
  {Kirchhoff}-love functional.
\newblock {\em Anal. PDE}, 10(4):943--982, 2017.

\bibitem[SKM93]{SKM93}
Stefan~Grigor'evich Samko, Anatolii~Aleksandrovich Kilbas, and Oleg~Igorevich
  Marichev.
\newblock {\em Fractional integrals and derivatives: theory and applications.
  {Transl}. from the {Russian}}.
\newblock New York, NY: Gordon {and} Breach, 1993.

\bibitem[SV12]{SV12}
Raffaella Servadei and Enrico Valdinoci.
\newblock Mountain pass solutions for non-local elliptic operators.
\newblock {\em J. Math. Anal. Appl.}, 389(2):887--898, 2012.

\bibitem[Tri78]{T78}
Hans Triebel.
\newblock {\em Interpolation theory, function spaces, differential operators},
  volume~18 of {\em North-Holland Math. Libr.}
\newblock Elsevier (North-Holland), Amsterdam, 1978.

\bibitem[Tri10]{T10}
Hans Triebel.
\newblock {\em Theory of function spaces}.
\newblock Mod. Birkh{\"a}user Class. Basel: Birkh{\"a}user, reprint of the 1983
  original edition, 2010.

\bibitem[WRO]{MARV}
Marvin Weidner and Xavier Ros-Oton.
\newblock Improvement of flatness for nonlocal free boundary problems.
\newblock {\em To appear in J. Eur. Math. Soc. (JEMS)}.

\end{thebibliography}
\newcommand{\etalchar}[1]{$^{#1}$}

\end{document}